\documentclass[reqno,11pt]{amsart}
\usepackage{amsmath,amsfonts,amssymb,amsxtra,latexsym,amscd,enumerate,amsthm,verbatim}

\usepackage[margin=1.40in]{geometry}
\numberwithin{equation}{section}

\def\ep{\epsilon}

\def\R{\mathbb{R}}

\def\C{\mathbb{C}}
\def\Z{\mathbb{Z}}

\def\N{\mathbb{N}}

\def\Z{\mathbb{Z}}

\def\ep{\epsilon}

\DeclareMathOperator{\sgn}{sgn}
\DeclareMathOperator{\supp}{supp}

\newtheorem{theorem}{Theorem}[section]
\newtheorem{lemma}[theorem]{Lemma}

\newtheorem{proposition}[theorem]{Proposition}
\newtheorem{definition}[theorem]{Definition}
\newtheorem{remark}[theorem]{Remark}

\def\beq{\begin{equation}}
\def\eeq{\end{equation}}

\def\beq{\begin{equation}}
\def\eeq{\end{equation}}

\begin{document}

\title[The 3D gravity water waves]{Long-term regularity of 3D gravity water waves}

\author{Fan Zheng}
\address{Instituto de Ciencias Matem\'aticas}
\email{fan.zheng@icmat.es}


\begin{abstract}
We study a fundamental model in fluid mechanics--the 3D gravity water wave equation, in which an incompressible fluid occupying half the 3D space flows under its own gravity. In this paper we show long-term regularity of solutions whose initial data is small but not localized.

Our results include: almost global wellposedness for unweighted Sobolev initial data and global wellposedness for weighted Sobolev initial data with weight $|x|^\alpha$, for any $\alpha > 0$. In the periodic case, if the initial data lives on an $R$ by $R$ torus, and $\epsilon$ close to the constant solution, then the life span of the solution is at least $R/\epsilon^2(\log R)^2$.
\end{abstract}

\maketitle

\setcounter{tocdepth}{1}

\tableofcontents

\date{\today}

\maketitle

\hfill {\it I live upstream and you downstream}

\hfill {\it From night to night of you I dream}

\hfill {\it Unlike the strewam you're not in view}

\hfill {\it Though we both drink from River Blue}

\hfill {\it ---Chinese Song Lyrics by Li Zhi-yi \cite{Xu}}

\section{Introduction}\label{Intro}

\subsection{The equations in the Zakharov formulation}
We consider the motion of an incompressible irrotational inviscid fluid
that occupies a domain
\begin{equation}\label{Om-def}
\Omega(t) = \{(x, z) \in \R^2 \times \R: z < h(x, t)\}
\end{equation}
in the three dimensional space, with a free boundary
\begin{equation}\label{Gm-def}
\Gamma(t) = \{(x, h(x, t)): x \in \R^2\}.
\end{equation}
In addition to the pressure, assumed to be 0 on the boundary $\Gamma(t)$,
the only external force acting on the fluid is gravity,
normalized to be $(0, 0, -1)$.

Let $v$ be the velocity of the fluid and $p$ be its pressure.
Within the domain $\Omega(t)$, the velocity $v$ satisfies
\begin{align}
\label{incomp}
\nabla_{x,z} \cdot v &= 0,
\tag{incompressibility}\\
\label{irrot}
\nabla_{x,z} \times v &= 0,
\tag{irrotationality}\\
\label{Euler}
v_t + v \cdot \nabla_{x,z}v &= -\nabla_{x,z}p - (0, 0, 1).
\tag{the Euler equation}
\end{align}
The velocity of the fluid on the boundary $\Gamma(t)$ dictates how it moves:
\begin{equation}\label{bdry}
\partial_t + v \cdot \nabla_{x,z} \text{ is tangent to } (t, \Gamma(t)).
\tag{the boundary condition}
\end{equation}

Since $v$ is irrotational, there is a velocity potential $\Phi$ on $\Omega(t)$, uniquely determined up to an additive constant, such that $v=\nabla_{x,z}\Phi$. The \ref{incomp} condition then translates into the harmonicity of $\Phi$ on $\Omega(t)$. Therefore $\Phi$ is uniquely determined (at least assuming sufficient regularity and decay of $v$ at infinity) by its boundary value
\[
\phi(x,t)=\Phi(x,h(x),t).
\]

We can now reformulate \ref{Euler} and \ref{bdry} in the Zakharov system,
whose derivation can be found in Section 11.1.1 of \cite{SS} or Section 1.1.4 of \cite{La-rev}.
\begin{equation}\label{Zakharov}
\begin{cases}
h_t=G(h)\phi,\\
\phi_t=-h-\frac12|\nabla\phi|^2+\frac{(G(h)\phi+\nabla h\cdot\nabla\phi)^2}{2(1+|\nabla h|^2)}
\end{cases}
\end{equation}
where $G(h)\phi=\sqrt{1+|\nabla h|^2}\partial_n\Phi$ is the
Dirichlet-to-Neumann operator
\footnote{It has become a convention in the water wave problem to put the factor of $\sqrt{1+|\nabla h|^2}$ in front of the usual normal derivative.}
associated to the domain $\Omega(t)$.

Note that the system (\ref{Zakharov}) has a conserved {\bf energy}
(Proposition 2.1 of \cite{GeMaSh2}):
\begin{equation}\label{E-def}
E
=\int_{\R^2} \frac12(\phi G(h)\phi + h^2)dx.
\end{equation}
Next we look at the evolution for the {\bf vorticity} $\omega=\nabla\times v$.
Taking the curl of \ref{Euler} we get its evolution equation
\begin{align*}
\omega_t
&=-(v\cdot\nabla)\omega-\omega(\nabla\cdot v).
\end{align*}
Therefore the \ref{irrot} assumption is preserved by the flow.

The trivial solution $(h, \phi) = (0, 0)$ is an equilibrium of the system (\ref{Zakharov}). Our goal is to investigate the long-term stability of this equilibrium.

To linearize the system (\ref{Zakharov}), we define the Fourier multipliers $|\nabla|$ and $\Lambda$.
\begin{definition}\label{Lambda-def}
Let $\mathcal F$ denote the Fourier transform. We define
\begin{align*}
\mathcal F(|\nabla|u)(\xi)&=|\xi|\mathcal Fu(\xi), &
\mathcal F(\Lambda u)(\xi)&=\sqrt{|\xi|}\mathcal Fu(\xi), &
\mathcal F(\langle\nabla\rangle u)(\xi)&=\sqrt{1+|\xi|^2}\mathcal Fu(\xi).
\end{align*}
\end{definition}

Near the equilibrium $(h,\phi)=(0,0)$, the Dirichlet-to-Neumann operator is, as a first order approximation, $G(h)\phi=\partial_z\Phi$,
where $\Phi$ solves the Dirichlet problem in the half space $\{(x,z):z\le0\}$ with the boundary value $\phi$ on the surface $\{(x,z):z=0\}$,
and the boundary condition is $\nabla\Phi\to0$ uniformly as $z\to-\infty$.
Using the Fourier transform, it is easy to see that the Dirichlet problem has a unique solution given, in the frequency space, by
\[
\mathcal F_x\Phi(\xi,z)=  e^{z|\xi|}\mathcal F\phi(\xi).
\]
Hence
\[
\mathcal F[G(h)\phi](\xi) = (\mathcal F_x\partial_z\Phi)(\xi, 0)
= |\xi|\mathcal F\phi(\xi)
\]
or $G(h)\phi = |\nabla|\phi$, to the first order. Now the system (\ref{Zakharov}) linearizes to
\[
\begin{cases}
h_t = |\nabla|\phi,\\
\phi_t = -h,
\end{cases}
\]
which can then be written in matrix form as
\[
\frac{d}{dt}
\begin{pmatrix}
h\\
\phi
\end{pmatrix}
=
\begin{pmatrix}
0 & |\nabla|\\
-1 & 0
\end{pmatrix}
\begin{pmatrix}
h\\
\phi
\end{pmatrix}
\]
and the eigenvalues of the matrix on the right-hand side are (formally)
\[
\pm\sqrt{-|\nabla|} = \pm i\sqrt{|\nabla|} = \pm i\Lambda
\]
with the corresponding eigenvectors
\begin{equation}\label{U-pm-def}
U_\mp = h \mp i\Lambda\phi.
\end{equation}

Thus the system (\ref{Zakharov}) is dispersive at the linear level.
The full system however, is fully nonlinear in terms of $U$ because the right-hand side contains terms quadratic in $\nabla h$ and $\nabla\phi$.
To obtain a quasilinear evolution equation, we will make use of the Alinhac good unknown, introduced in \cite{Alin},
\begin{equation}\label{tld-U-pm-def}
\tilde U_\pm=h\pm i|\nabla|^{1/2}(\phi-T_{\partial_z\Phi}h|_{z=h(x)}),
\end{equation}
where $T_{\partial_z\Phi}h$ is a paraproduct, to be defined in section \ref{ParaCalc}. Our main results are stated in terms of the variable $\tilde U=\tilde U_+$. Note that by Proposition \ref{U-tld=U} below, $U$ and $\tilde U$ are comparable in terms of energy estimates.

\subsubsection{The periodic setup}
We will also show long-term wellposedness of the 3D gravity water wave (GWW) equation on a torus of size $R \times R$, which we denote by $(\R/R\Z)^2$. Thus the domain that the fluid occupies is now
\begin{equation}\label{Om-per-def}
\Omega(t) = \{(x, z) \in (\R/R\Z)^2 \times \R: z < h(x, t)\}
\end{equation}
with a free boundary
\begin{equation}\label{Gm-per-def}
\Gamma(t) = \{(x, h(x, t)): x \in (\R/R\Z)^2\}
\end{equation}
on which we still assume the pressure is 0. Because gravitational acceleration has the dimension $[LT^{-2}]$, it can still be normalized to $(0, 0, -1)$ by reparametrizing the time variable.

As before, the flow is assumed to be incompressible and irrotational,
satisfying the Euler equation. The boundary moves according to the local velocity.

When we try to put the above system in the Zakharov formulation (\ref{Zakharov}) however,
there is one notable difference: the velocity field always integrates to a velocity potential in the Euclidean case, but not in the periodic one.
There are cohomological obstructions to that, namely the line integral of the velocity field along any homology class of the torus must be 0:
\[
\int_\gamma v = 0,\quad \forall \gamma \in H^1((\R/R\Z)^2, \R).
\]
This is known as the {\bf zero momentum condition}. It can always be ensured using a change of variable ($\tilde v = v - v_0$, $\tilde p(x, t) = p(x - v_0t, t)$). Under such a condition, the velocity potential $\Phi$ can always be found such that $v = \nabla_{x,z}\Phi$,
and is uniquely determined by its boundary value $\phi$ as before.
In this setup, the GWW equation can still be put in the Zakharov formulation (\ref{Zakharov}), with the meaning of the variables unchanged.

The periodic GWW equation shares many of the conservation laws with the Euclidean one. For example, the energy
\begin{equation}\label{E-per-def}
E = \int_{(\R/R\Z)^2} \frac12(\phi G(h)\phi + h^2)
\end{equation}
is still invariant under the flow. The fluid will remain irrotational if it is so initially. The {\bf zero momentum condition} is also preserved because $v_t$ is a gradient under the irrotationality assumption,
see (1.11) of \cite{Zh} for example.
%
%

With the definition of the operators as given in Definition \ref{Lambda-def}, the Zakharov system (\ref{Zakharov}) can still be linearized in terms of the variable $U_\pm$ given by (\ref{U-pm-def}). We will still use the Alinhac good unknown $\tilde U_\pm$ as defined in (\ref{tld-U-pm-def}).

Periodic solutions of nonlinear dispersive equations have been studied in the context of the {\it large box limit} (see for example \cite{BuGeHaSh}),
in which the equation is posed on a torus of size $R$, and in the limit
$R \to \infty$, it was observed that the behavior of the solution can be approximated by a wave kinetic equation, known as the {\it continuous resonance equation} and believed to describe {\it weak turbulence}.
To initiate the study of weak turbulence for the GWW equation,
it is necessary to first obtain its long-term regularity on a large torus.
Such a setting also fills in the gap between the global wellposedness result in the Euclidean setting and cubic lifespan on the unit torus, to be detailed later.

\subsection{The main theorems}
Now we can state the main results about the 3D gravity water wave (GWW) equation in this paper. Throughout the paper $N \in \N$ and $\alpha \in (0, 1)$ are fixed, and all implicit constants only depend on them unless stated otherwise.
We also assume that $\ep$ is sufficiently small.

\begin{definition}\label{Z-def}
In the Euclidean space we define
\[
\|u\|_Z = \|(1 + |x|)^\alpha\langle\nabla\rangle^8u\|_{L^2}.
\]
On the $R \times R$ torus we define
\begin{align*}
\|u\|_Z &= \|(1 + \|x\|)^{2/3}\langle\nabla\rangle^8u\|_{L^2}, &
\|x\| &= d(x, (R\Z)^2).
\end{align*}
In other words, $\|x\|$ is the geodesic distance from the center of the torus, fixed once and for all.
\end{definition}

\begin{theorem}\label{Thm1}
If $N \ge 11$ then there is a constant $c > 0$ such that if
\begin{equation}\label{HN0}
\|h_0\|_{H^{N+1/2}}+\||\nabla|^{1/2}\phi_0\|_{H^N}\le\ep,
\end{equation}
then there is
\begin{equation}\label{T1}
T_\ep \approx \exp(c/\ep)
\end{equation}
such that (\ref{Zakharov}) has a unique solution with initial data $(h_0,\phi_0)$ that satisfies $\tilde U\in C([0,T_\ep],H^N)$.
\end{theorem}

\begin{theorem}\label{Thm2}
If $N \ge \max(33/(\alpha - \alpha^2),8/\alpha^2)$, $\alpha \in (0,1)$ and
\begin{equation}\label{Z0}
\|h_0\|_{H^{N+1/2}}+\||\nabla|^{1/2}\phi_0\|_{H^N}+\|h_0\|_Z+\||\nabla|^{1/2}\phi_0\|_Z\le\ep,
\end{equation}
then (\ref{Zakharov}) has a global unique solution with initial data $(h_0,\phi_0)$ such that $\tilde U(t)$ is bounded and continous in $H^N$ for all $t\ge 0$.
\end{theorem}

\begin{remark}
(i) To the knowledge of the author, Theorem \ref{Thm1} is the first almost global wellposedness result of the 3D GWW equation with nonlocalized Sobolev initial data. In 2D similar arguments show a lifespan $\gtrsim \ep^{(-4)+}$ for such initial data, an improvement over the $\ep^{-2}$ lifespan in \cite{AlDe,HuIfTa2DaG,IoPu,Wu2DaG}.

(ii) Theorem \ref{Thm2} improves on the global wellposedness results of Germain--Masmoudi--Shatah \cite{GeMaSh2} and Wu \cite{Wu3DG} by reducing the weights on the initial data to an arbitrarily small polynomial power.

(iii) The framework used to prove these theorems is very flexible, and applies to many other nonlinear dispersive equations.
\end{remark}

In the periodic case we also assume that the period $R$ is sufficiently large.

\begin{theorem}\label{Thm1-per}
If $N\ge11$ then there is a constant $c>0$ such that if $1\ll R\le \exp(c/\ep)$ and (\ref{HN0}) holds, then there is
\begin{equation}\label{T1-per}
T_{R,\ep} \approx R/(\ep^2(\log R)^2)
\end{equation}
such that (\ref{Zakharov}) has a unique solution with initial data $(h_0,\phi_0)$ that satisfies $\tilde U\in C([0,T_{R,\ep}],H^N)$.
\end{theorem}

\begin{theorem}\label{Thm2-per}
If $N \ge 41$, $R \ge \ep^{-3/5}$ and (\ref{Z0}) holds,
then there is
\begin{equation}\label{T2-per}
T_{R,\ep}' \approx R^{12/11-O(1/N)}\ep^{-6/11+O(1/N)}
\end{equation}
such that (\ref{Zakharov}) has a unique solution with initial data $(h_0,\phi_0)$ that satisfies $\tilde U\in C([0,T_{R,\ep}],H^N)$.
\end{theorem}

\begin{remark}
(i) For specific values of the constants in the exponents see Proposition \ref{bootstrap2-per}.

(ii) If the assumption made on $R$ and $\ep$ in Theorem \ref{Thm1-per} does not hold, then we have $R\gtrsim\ep^{-100}$ (say), and Theorem \ref{Thm2-per} gives a better bound. Either lifespan is longer than $R/\ep$, which is the most one can hope for without using the normal form, for no decay can be expected of the $L^2$ norm of the solution, giving a lower bound of $\gtrsim\ep/R$ of its $L^\infty$ norm.

(iii) When $R=1$, Theorem \ref{Thm1-per} gives a lifespan $\gtrsim \ep^{-2}$, which matches what can be obtained from the normal form transformation as used in \cite{Wu3DG,GeMaSh2}. For larger $R$ however, the normal form transformation alone does not give a longer lifespan; there global wellposedness was obtained from decay estimates, which requires localized initial data satisfying (\ref{Z0}) instead of (\ref{HN0}). Scaling to $R=1$ is not easy either, for the initial data would be rough due to spatial shrinking.

(iv) Unlike Faou-Germain-Hani \cite{FaGeHa} and Buckmaster--Germain--Hani--Shatah \cite{BuGeHaSh}, our proofs of Theorems \ref{Thm1-per} and \ref{Thm2-per} do not rely on the number-theoretic properties of the resonance set. Hence it is straightforward to generalize our results to nonsquare tori with bounded aspect ratios.
\end{remark}

\subsection{Background}
The water wave problem is central to the study of many fluid phenomena, for example, the propagation of waves in the ocean. As such, it enjoys a long history of active research, of which we will only mention a selected few, referring the interested reader to \cite{BrGrNi-rev,IoPu-rev,La-rev} for more comprehensive references.

\subsubsection{Local wellposedness}
Local wellposedness of the water wave problem was first studied by Nalimov \cite{Na}, Shinbrot \cite{Shin}, Yosihara \cite{YoGWW,YoCGWW} and Craig \cite{Cr}.
All these works assume that the surface is close to a flat one,
in which case the Rayleigh--Taylor sign condition \cite{Tay}
\begin{equation}\label{Taylor-sign}
\nabla_{n(x,t)}p(x,t) < 0
\end{equation}
always holds, where $n$ is the outward pointing unit normal to $\Omega(t)$.

A breakthrough came when Wu \cite{Wu2DL} showed that (\ref{Taylor-sign}) always holds, as long as the surface does not cross itself,
thus extending the local wellposedness results to large initial data in 2 dimensions \cite{Wu2DL} and 3 dimensions \cite{Wu3DL}.

After Wu's breakthrough, local wellposedness of the water wave problem has been generalized in many aspects: the motion of a drop (with vorticity in Chris\-to\-dou\-lou--Lindblad \cite{ChLi} and Lindblad \cite{Lind};
with capillarity in Beyer--G\"unther \cite{BeGu};
with both in Coutand--Shkoller \cite{CoSh} and Shatah--Zeng \cite{ShZe};
with another surrounding fluid in Shatah--Zeng \cite{ShZe2,ShZe3}),
the motion of vortex sheets in Cheng--Coutand--Shkoller \cite{ChCoSh},
irrotational gravity-capillary waves in Ambrose--Masmoudi \cite{AmMa2DL,AmMa3DL}, rotational gravity water waves in Zhang--Zhang \cite{ZhZh},
the existence of nontrivial bottom topography in Lannes \cite{La},
Ming--Zhang \cite{MiZh} and Alazard--Burq--Zuily \cite{AlBuZu,AlBuZu2},
the presence of an emerging bottom in de Poyferr\'e \cite{DePoy},
and the existence of angled crests in Kinsey--Wu \cite{KiWu} and Wu \cite{Wucrest}.

All the results above are local in nature, that is, the lifespan of the solution, if calculated explicitly, is no longer than the reciprocal of the size of the initial data. We now review results yielding longer lifespans.

\subsubsection{Almost global and global wellposedness}
The first global wellposedness results for water waves were obtained for 3D GWW, by Germain--Masmoudi--Shatah \cite{GeMaSh2} and Wu \cite{Wu3DG}.
Soon afterwards, the corresponding result for 3D capillary water waves was obtained by Germain--Masmoudi--Shatah \cite{GeMaSh3}, and for 3D gravity-capillary water waves by Deng--Ionescu--Pausader--Pusateri \cite{DeIoPaPu}.
In the presence of a flat bottom, wellposedness of the 3D gravity (resp. capillary) water waves was shown by Wang \cite{Wa3DL,Wa3DG} (resp. \cite{Wa3Dc}).

In two dimensions (i.e., one-dimensional surface), long-term wellposedness was first shown by Wu \cite{Wu2DaG} (see also \cite{HuIfTa2DaG}),
who showed that the lifespan of 2D GWW is exponential in terms of the reciprocal of the size of the initial data.
Global wellposedness was later shown by Ionescu--Pusateri \cite{IoPu},
Alazard--Delort \cite{AlDe,AlDe2} and Ifrim--Tataru \cite{IfTa2DG}.
Global wellposedness was also shown for some solutions with infinite energy by Wang \cite{Wa2DG}. For 2D capillary water waves, global wellposedness was shown by Ionescu--Pusateri \cite{IoPu2} and Ifrim--Tataru \cite{IfTa2Dc} (assuming an additional momentum condition).

We remark that all the results above assumed that the initial data is very localized. For example, in Wu \cite{Wu3DG}, the initial data is required to satisfy (schematically)
\[
\sum_{j\le l} \|\Gamma^j(\text{initial data})\|_{L^2} \le \ep
\]
for some $l \ge 17$, where $\Gamma$ is the one of the vector fields in the set
\[
\{\partial_1, \partial_2, x_1\partial_1 + x_2\partial_2, x_1\partial_2 - x_2\partial_1\}.
\]
This amounts to putting a weight comparable to $|x|^{17}$ in the Sobolev norm. In Germain--Masmoudi--Shatah \cite{GeMaSh2}, the weight is smaller, but is still comparable to $|x|$.

It is therefore interesting to see from Theorems \ref{Thm1} and \ref{Thm2} that the weights can be reduced to almost the minimum,
so that the initial data can be made (almost) nonlocalized.
For example, if two small stirs appear far apart in the same body of water,
a situation poetically portrayed in the lead of this paper,
the initial data is small in unweighted Sobolev spaces but large in weighted ones, to which previous global wellposedness results are inapplicable, yet one expects that the water will remain placid for a long time, because it will take long for the two stirs to significantly interact with each other, when they have already dispersed considerably. That this is the case is shown by Theorem \ref{Thm1}. Such ``multi-bump" initial data was studied in the context of quasilinear wave equations by Anderson--Pasqualotto \cite{AnPa}.

\subsubsection{Strichartz estimates}
Strichartz estimates of quasilinear dispersive equations are also attracting research. Pioneering work in quasilinear wave equations was done by Bahouri--Chemin \cite{BaCh}, Tataru \cite{TatStr,TatStr2,TatStr3} and Blair \cite{Bl}, with extensions to Schr\"odinger equations in Staffilani--Tataru \cite{StTa}, Burq--Gerard--Tzvetkov \cite{BuGeTz} and Robbiano--Zuily \cite{RoZu} (see also Lebeau \cite{Le}). For water waves, Strichartz estimates for 2D gravity-capillary waves were shown in Christianson--Hur--Staffilani \cite{ChHuSt} in the infinite depth case, and Alazard--Burq--Zuily \cite{AlBuZuStr2D} in the finite depth case. For 2D finite depth gravity water waves, Strichartz estimates were shown in Ai \cite{Ai2} (which may be improved by techniques developed in \cite{AIT}). In general dimensions, Strichartz estimates were shown by Alazard--Burq--Zuily \cite{AlBuZuStr3D} (for better results see \cite{Ai}) for gravity water waves, and de Poyferr\`e--Ngyuen \cite{DeNg} and Ngyuen \cite{Ng} for capillary-gravity water waves.

All previous results on Strichartz estimates for the water wave equation are local in time, that is, they only hold for an unspecified time (depending on the norms of the initial data). For example, \cite{AlBuZuStr3D} showed that (for 3D GWW) if initially the data $(h_0, \phi_0, V_0, B_0) \in H^{s+1/2} \times H^{s+1/2} \times H^s \times H^s$,\footnote{The meaning of these quantities will become clear later.}
where $s$ is above some threshold, the Taylor sign condition (\ref{Taylor-sign}) holds, and the depth is bounded from below uniformly, then for $r$ in some range depending on $s$, there is $T > 0$ such that
\begin{align*}
(h, \phi, V, B)
&\in C([0, T] \to H^{s+1/2} \times H^{s+1/2} \times H^s \times H^s)\\
&\cap L^2([0, T] \to W^{r+1/2,\infty} \times W^{r+1/2,\infty} \times W^{r,\infty} \times W^{r,\infty}).
\end{align*}
Although this result is purely local, it does suggest the right space to bound the solution in, namely the space $L^2W^{r,\infty}$.
This fits nicely into the quartic energy estimates, to be discussed later.
Therefore extension of such Strichartz estimates is crucial to the proof of the (almost) global results in Theorem \ref{Thm1} and Theorem \ref{Thm2},
though we will work at the level $r = 6$, much higher than the sharp results cited above, because we are not aiming at optimal regularity in this paper. Similar ideas were used to show long-term regularity of the periodic Euler--Poisson equation in 2D by the author \cite{Zh}.

\subsubsection{Nonlocalized initial data}
As mentioned previously, local wellposedness of water waves with initial data in (unweighted) Sobolev spaces is known. Local wellposedness is also known for (horizontally) periodic data, that is, functions living on the torus $(\R/\Z)^n$, see Schweizer \cite{Schw} and Kukavica--Tuffaha--Vicol \cite{KuTuVi} for the Euler equation in a domain with a flat bottom and Ambrose--Masmoudi \cite{AmMa2DL} for infinitely deep gravity-capillary water waves in 2D.

In the case of GWW, lifespans of solutions can be extended to ``cubic",
that is, a time of existence $\approx \ep^{-2}$, where $\ep$ is the size of the initial data. Indeed, this is an immediate consequence of the normal form transformation, to be discussed later, and can be derived from results in \cite{AlDe,HuIfTa2DaG,IoPu,Wu2DaG} for 2D GWW, and \cite{Wu3DG,GeMaSh2} for 3D GWW, see also \cite{IfTa2DG,IoPu2,IfTa2Dv,Su} for 2D capillary water waves and 2D GWW with constant vorticity and point vortices.
Cubic lifespans for GWW with more general nonlocalized data (Euclidean Sobolev + periodic Sobolev) were also obtained by Su \cite{Su2},
in the context of modulation approximation. It was also shown for the free boundary problem for incompressible self-gravitating fluids with zero or constant vorticity by Bieri--Miao--Shahshahani--Wu \cite{BiMiShWu}.

In general, if the nonlinear terms are of degree $d$,
conventional wisdom expects a lifespan $\approx \ep^{1-d}$.
If there is more algebraic structure that enables more iterations of the normal form transformation, the lifespan can be longer (in the sense of a higher exponent), as shown by Berti--Feola--Pusateri in \cite{BeFePu} for 2D periodic GWW, and by Delort--Szeftel in \cite{DeSz} for nonlinear Klein--Gordon equations on spheres. A certain genericity of parameters also allows more iterations of the normal form transformation, as shown by Delort in \cite{De2} for quasilinear Klein--Gordon equations on the circle, and Berti--Delort \cite{BeDe} for 2D periodic gravity-capillary water waves. It is therefore surprising to know that the lifespan of the 3D GWW equation, even without locality of initial data and genericity of parameters, can extend well beyond what its algebraic form predicts.

\subsection{Main Ideas}
Starting from Shatah \cite{Sh}, Poincar\'e's {\it normal form} method (see \cite{Ar,Ch} for book reference) has proved successful in the study of long-term solutions of nonlinear evolutions. The key observation is that, if the dispersion relation $\Lambda$ is {\it nonresonant}, in the sense that
\[
\Lambda(\xi) \pm \Lambda(\eta) \pm \Lambda(\xi - \eta) \neq 0
\]
for any frequencies $\xi$ and $\eta$, and any choice of signs,
then one can make a quadratic correction to the solution or the energy (the details are very similar) so that the equation is transformed into one with cubic nonlinearity, that is,
\[
\partial_tU = N(U, U, U),
\]
where the right-hand side is at least cubic in $U$. If the initial data is localized, and the dispersion relation $\Lambda$ is non-degenerate, then the $L^\infty$ norm of its evolution is expected to decay like $t^{-n/2}$ in $\R^n$. For example, in 3D GWW the decay rate is $t^{-1}$, so two factors of $U$ on the right-hand side bring an integrable decay rate of $t^{-2}$; whence the global wellposedness results of Germain--Masmoudi--Shatah \cite{GeMaSh2} and Wu \cite{Wu3DG}.

In the nonlocal setting, the normal form transformation also makes the nonlinearity cubic, yet there is much less decay. so new ingredients are needed to show almost global and global wellposedness. If one uses vector fields a l\`a Klainerman, their coefficients will introduce weights in the Sobolev norm of the initial data. If one tries space-time resonance, the weight arises from differentiation in the frequency space. Thus new ingredients are called for to treat nonlocal initial data.

Following \cite{Zh}, we will take advantage of the absence of resonance in time, but will not care about space resonance. The proof combines three parts.

\setlength{\leftmargini}{1.8em}
\begin{itemize}
  \item[(1)] Quartic energy estimates to control high frequencies of the solution;
\smallskip
  \item[(2)] Strichartz estimates to control the $L^\infty$ norm of low frequencies;
\smallskip
  \item[(3)] The $Z$-norm estimates and interpolation with the Strichartz estimates.
\end{itemize}
We will also sketch how to generalize this framework to the periodic setting.

\subsubsection{Quartic energy estimates}
As discussed above, to perform the normal form transformation on the GWW equation, we need to check the nonresonance condition.
This is almost true, thanks to the concavity of the dispersion relation
$\Lambda(\xi) = \sqrt{|\xi|}$ away from $\xi = 0$,
but fails when one of $\xi$, $\eta$ and $\xi - \eta$ vanishes.
This is not a problem however, thanks to a certain ``null condition" present in the GWW equation, which implies that in such cases the multiplier $m(\xi, \eta)$ in the nonlinearity vanishes,
so the contribution of zero (or more generally, very small) frequencies can be ignored (see (\ref{m-Soo}) for a precise statement),
and the normal form transformation works the same way as if there were no resonance, to make the nonlinearity cubic.

For cubic nonlinearity, the energy estimate usually reads
\[
\frac{d}{dt}E \lesssim E\|U\|_{W^{r,\infty}}^2,
\]
where $r$ is an integer specific to the equation in question.
To adapt such an estimate to the quasilinear GWW equation,
we need to fight against the loss of derivatives, both in the equation itself and in the process of the normal form transformation;
meanwhile we also need to deal with the Dirichlet-to-Neumann operator,
which is nonlocal in nature. Both tasks can be fulfiiled using the paradifferential calculus established in \cite{AlMe,AlBuZu,AlBuZu2}.

\subsubsection{Strichartz estimates}
Using Gronwall's inequality on the quartic energy estimate we get
\[
E(t)\le E(0)\exp\left( C\int_0^t \|U(s)\|_{W^{r,\infty}}^2ds \right).
\]
Hence the growth of the energy $E$ is controlled by the $L^2W^{r,\infty}$ norm of the solution, which is exactly the natural spacetime norm produced by the Strichartz estimates of linear dispersive equations in 2D. Such linear estimates have been extended to the full GWW equation in \cite{AlBuZuStr3D,Ai}.

For long-term regularity, however, we need to work out the time dependence of the implicit constants. An estimate like
\[
\|U\|_{L^2([0,t])W^{r,\infty}}\lesssim \|U(0)\|_{H^s}
\]
would fit perfectly into the energy estimate above to give global wellposedness. This, however, is too good to be true, even for the linear Schr\"odinger equation in 2D, due to the notorious endpoint failure, as shown by Keel--Tao \cite{KeTa}.
Nevertheless we can save the day if we tolerate some logarithmic loss.
Indeed, repeating the classical $TT^*$ argument shows that the endpoint failure can be traced back to the divergence of the integral of $1/t$ near $t = 0$ and $t = \infty$.

The divergence near $t = 0$ is due to the diminishing dispersive effect as $t \to 0$, which forces us to rely increasingly on the derivative losing embedding $H^{1+}\subset L^\infty$. We solve this problem by allowing some regularity gap between the energy estimates and the dispersive estimates.
Exactly how much gap is necessary has been worked out in \cite{Ai,Ai2,AIT},
but for the purpose of global wellposedness we care less about the sharp regularity threshold than the decay of solution in time, so the argument is much easier.

Regarding the divergence of the integral of $1/t$ near $t = \infty$,
we simply accept the logarithmic loss in the estimate, which then becomes (morally)
\[
\|U\|_{L^2([0,t])W^{r,\infty}}\lesssim (\log t)\|U(0)\|_{H^s}.
\]
Put in the energy estimate, this gives
\[
E(t) \le E(0)\exp(C(\log t)^2\ep^2),
\]
assuming that the size of the solution remains $\ep$ over the period during which we can still control the solution. Thus to close the estimate we need
\[
(\log t)\ep \le 1/C,
\]
which then gives an almost global (exponential in $1/\ep$) lifespan of
\[
T_\ep = \exp(C\ep^{-1}).
\]

For the nonlinear term $N$, by Duhamel's formula, it suffices to bound its $L^1H^{r+O(1)}$ norm. By H\"older's inequality and the Sobolev multiplication theorem, the multilinear estimate
\begin{equation}\label{trilinear}
L^2L^\infty \times L^2L^\infty \times L^\infty H^{s} \to L^1H^{r+O(1)}
\end{equation}
is possible. All three norms on the left-hand side can be bounded,
the first two using the Strichartz estimates, and the third one using the energy estimates. The loss of derivatives in $N$ is of no concern,
again thanks to the gap between $r$ and $s$.

The above argument exploits a different aspect of the Strichartz estimates than in the semilinear world. For semilinear equations we are more concerned with the regularity of the solution, and the Strichartz estimates reveal that the solution is more regular if we average over time, without which we could be stuck at $L^2$-type conservation laws.
For quasilinear equations however, regularity is less of a concern because the equation is usually too supercritical for semilinear techniques to work. In such cases the Strichartz estimates are used to control the decay rate of the solution. It suggests that if we average over time the solution decays more than expected. Indeed, for the linear Schr\"odinger equation in 2D, if the initial data is of size 1 in the Sobolev space $H^s$ for some $s \gg 1$, then at any fixed time $t > 0$, the best bound we have of the $L^\infty$ norm of the solution is still 1, because the linear evolution is an isometry of $H^s$, and the solution could refocus to a bump function at that time. The Strichartz estimates then say that the refocusing, even if it happens, will not last long and the solution will quickly disperse again, leading to an $L^2L^\infty$ norm of size $O(\log t)$, instead of $\approx \sqrt t$ from trivial integration. Incidentally, this ``square root saving" is a hallmark of the $TT^*$ machinery.

\subsubsection{Interpolation with weighted Sobolev estimates of the profile}
As noted above, unweighted Sobolev data narrowly misses global lifespan, while weighted Sobolev data does have global wellposedness (for example, in \cite{GeMaSh2} the weight is $|x|$). In \cite{Zh} the author reduced the weight to $|x|^{2/3}$ for a similar non-resonant dispersive equation using more careful decomposition of the profile. Here the author combines that technique with an interpolation argument to further reduce the weight to $x^\alpha$ for any $\alpha > 0$.

We assume that the following norm
\[
\|U\|_{L^\infty H^s} + \|t^{\alpha-/2}U\|_{L^2W^{r,\infty}}
+ \underbrace{\||x|^\alpha \Upsilon\|_{L^\infty H^{r'}}}_{=L^\infty Z-\text{norm}}
\]
is small, where $\Upsilon = e^{it\Lambda}U$ is the {\it profile}.
We already have the quartic energy estimate
\begin{align*}
\|U(t)\|_{H^s}^2 &\lesssim_s \|U(0)\|_{H^s}^2e^M, &
M &\lesssim_r \int_0^t \|U(\tau)\|_{W^{r,\infty}}^2d\tau
\lesssim \|t^{\alpha-/2}U\|_{L^2W^{r,\infty}}^2.
\end{align*}
For the second term, interpolation between
\begin{align*}
\|t^{0-}U\|_{L^2W^{r,\infty}}
&\lesssim_r \|\Upsilon\|_{L^\infty H^{r+O(1)}},
\tag{Strichartz estimates}\\
\|t^{0.5-}U\|_{L^2W^{r,\infty}}
&\lesssim_r \||x|\Upsilon\|_{L^\infty H^{r+O(1)}}
\tag{$1/t$-decay estimates}
\end{align*}
gives
\[
\|t^{\alpha-/2}U\|_{L^2W^{r,\infty}}
\lesssim_r \||x|^\alpha\Upsilon\|_{L^\infty H^{r+O(1)}}
=\||x|^\alpha\Upsilon\|_{L^\infty H^{r'}}
\]
if we let $r' = r + O(1)$. We are thus left with the $Z$-norm before closing the estimate.

Since extreme frequencies are easier to control, we can assume that all frequencies involved are of unit size. Then after time $t$, the wave packets spread a distance comparable to $t$, so
\[
\||x|^\alpha\Upsilon(t)\|_{H^{r'}}
\lesssim_{r'} t^\alpha\|\Upsilon(t)\|_{L^2}
= t^\alpha\|U(t)\|_{L^2}.
\]
If we use (\ref{trilinear}), the estimate almost closes itself, except that we require $t^{\alpha/2}$ for the weight in the $L^2W^{r,\infty}$ norm, but we only have $t^{\alpha-/2}$ in the assumption.

We will use the technique in \cite{Zh} to decompose the profile more carefully, with some adaption to the new bootstrap assumption.
Since $N$ is trilinear in $U$, we put $N = N(U, U, U)$. We decompose only {\it one} (instead of all three in \cite{Zh}) of the $\Upsilon$ factors:
for some $0 < c < C$ we have
\begin{align*}
\Upsilon(t) &= \Upsilon_1(t) + \Upsilon_2(t), &
\supp\Upsilon_1(t) &\subset B(0, Ct^\beta), &
\supp\Upsilon_2(t) &\subset \R^2 \backslash B(0, ct^\beta).
\end{align*}
Let $N(t) = N_1(t) + N_2(t)$, where $N_j(t) = N(U(t), U(t), e^{-it\Lambda}\Upsilon_j(t))$, $j = 1$, 2. For $N_2$ we use
\[
\|N_2(t)\|_{L^2} \lesssim \|U(t)\|_{L^\infty}^2\|e^{-it\Lambda}\Upsilon_2(t)\|_{H^{O(1)}} = \|U(t)\|_{L^\infty}^2\|\Upsilon_2(t)\|_{H^{O(1)}} \lesssim t^{-\alpha\beta}\|U(t)\|_{L^\infty}^2\|\Upsilon_2(t)\|_Z.
\]
For $N_1$ we have, by the $1/t$-decay estimates,
\[
\|N_1(t)\|_{L^2} \lesssim \|U(t)\|_{L^\infty}\|U(t)\|_{H^{O(1)}}\|e^{-it\Lambda}\Upsilon_1(t)\|_{L^\infty}
\lesssim t^{-1+(1-\alpha)\beta}\|U(t)\|_{L^\infty}\|U(t)\|_{H^{O(1)}}\|\Upsilon_1(t)\|_Z
\]
where the factor $t^{(1-\alpha)\beta}$ accounts for the difference between the $L^1$ norm and the $Z$ norm. Then
\[
\|t^\alpha N_1(t)\|_{L^1L^2}
\lesssim t^{(1-\alpha)(\beta-1/2)+}\|t^{\alpha-/2}U(t)\|_{L^2L^\infty}
\|U(t)\|_{L^\infty H^{O(1)}}\|\Upsilon_1(t)\|_{L^\infty Z}.
\]
An optimal choice of $\beta = (1 - \alpha)/2$ produces a nontrivial saving and closes the $Z$-norm estimate.

\subsubsection{The periodic setting}
The proof in the periodic case has two differences compared to the Euclidean proof. The first one is that on the torus the $L^\infty$ norm decays less because wave packets can wrap around the torus. After time $t$, wave packets of unit frequency wrap around the torus $t/R$ times, both horizontally and vertically, adding an extra factor of $(t/R + 1)^2$ to linear dispersive estimates, see Lemma \ref{dispersive-per} for details. The second difference is that we now aim to extend the lifespan of the solution rather than reduce the weights. To do so we fix the weight at $|x|^{2/3}$. For the Euler--Poisson equation considered in \cite{Zh}, this would imply a decay rate of $t^{-4/3}$ of the solution and a lifespan of $R^{10/9}/\ep^{2/3}$ in the periodic case: both are optimal within this framework, see the discussion in section 1.6 of \cite{Zh} for details. The GWW equation is more complicated, for the group velocity is $|\nabla\Lambda(\xi)|=|\xi|^{-1/2}/2$, so low frequency wave packets wrap around the torus more. As a result, the exponent of 2/3 is likely to be slightly suboptimal, but the optimal choice is likely to lead to ugly fractions in the exponents of $R$ and $\ep$, so the author sticks to the weight $|x|^{2/3}$ for the sake of clear presentation.

\subsection{Organization}
The rest of the paper is organized as follows: In section \ref{Local} we establish local wellposedness of the GWW system and state the main bootstrap propositions. In section \ref{LinEst} we introduce paradifferential calculus and derive the linear dispersive estimates and multilinear paraproduct estimates to be used in the rest of the paper. In section \ref{ParLin} we paralinearize the system and then obtain the quartic energy estimates in section \ref{EneEst}. In section \ref{StrEst} we prove Theorem \ref{Thm1}, and in section \ref{ZEst} we prove Theorem \ref{Thm2}. In section \ref{Period} we generalize Theorem \ref{Thm1} and Theorem \ref{Thm2} to the periodic setting. In the appendices we record technical estimates related to the paralinearization process and the remainders of the Dirichlet-to-Neumann operator.

\subsection{Notation}
Here we introduce the notation to be used throughout this paper.

Let $\varphi$ be a smooth cutoff function that is 1 on $B(0,3/4)$ and vanishes outside $B(0,3/2)$. Let
\begin{align*}
\varphi_j&=\varphi(x/2^j)-\varphi(x/2^{j-1}), &
\varphi_{\le j}&=\varphi(x/2^j).
\end{align*}
Let $P_k$ be the Littlewood-Paley projection onto frequency $2^k$, so that
\begin{align*}
\mathcal F(P_ku)&=\varphi_k \mathcal Fu, &
\mathcal F(P_{\le k}u)&=\varphi_{\le k} \mathcal Fu.
\end{align*}
In the periodic setting, we define the Fourier transform $\mathcal Fu$ of a function $u$ on $(\R/R\Z)^2$ as
\begin{align*}
\mathcal Fu(\xi) &= \int_{(\R/R\Z)^2} e^{-ix\cdot\xi} u(x)dx, &
\xi & \in (2\pi\Z/R)^2.
\end{align*}
The inverse Fourier transform is then given by
\begin{align*}
\mathcal F^{-1}u(x)&=R^{-2}\sum_{\xi\in(2\pi\Z/R)^2}e^{ix\cdot\xi}u(\xi), &
x & \in (\R/R\Z)^2.
\end{align*}
Thus $P_k$ vanishes if $2^k < 1/R$. In particular $P_k$ excludes the zero frequency. In both settings we define the H\"older--Zygmund norm for $r > 0$:
\[
\|u\|_{C_*^r} = \|P_{\le0}u\|_{L^\infty} + \sup_{k>0} 2^{kr}\|P_ku\|_{L^\infty}.
\]

For decomposition in the Euclidean physical space, we define
\begin{align*}
Q_ju&=\varphi_ju\ (j\ge1), & Q_0&=id-\sum_{j\ge1} Q_j.
\end{align*}
In the periodic setting we define a similar decomposition, using the geodesic distance instead of the Euclidean distance, i.e., we let
\begin{align*}
\varphi_j^R(x) &= \varphi_j(y), & x &\in (\R/R\Z)^2, &
y &\in \R^2, & \|x\| := d(x, (R\Z)^2) &= |y|.
\end{align*}
We then let $Q_ju = \varphi_j^Ru$, noting that $Q_j$ vanishes if $2^j > \sqrt{2}R$. We finally define
\begin{align*}
k^+&=\max(k,0), & k^-&=\min(k,0).
\end{align*}

We also state some facts about the $Z$ norm. As a rough approximation,
one can think of the $Z$ norm as $W^{8,\frac2{1+\alpha}+}$.

\begin{lemma}\label{Z-bound}
(i) Let $2/(1 + \alpha) < p \le 2$. For $k \in \Z$ we have
\[
\|P_ku\|_{W^{8,p}} \lesssim_p \|u\|_Z.
\]


(ii) For $k \in \Z$ we have
\[
\|(1+|x|)^\alpha P_ku\|_{L^2}\lesssim 2^{-8k^+}\|u\|_Z.
\]

(iii) Calderon--Zygmund operators are bounded on $Z$.


(iv) $\|u\|_Z\approx \|(1+|x|)^\alpha\|\langle\nabla\rangle^8P_ku\|_{\ell^2_k}\|_{L^2}=\|\|P_ku\|_Z\|_{\ell^2_k}$.
\end{lemma}
\begin{proof}
Lemma 1.4 of \cite{Zh} shows (i)--(iii) for the (essentially same) case of $\alpha=2/3$. To get (iv) we use the vector-valued version of the Corollary to Theorem 2 in Section V.4 of \cite{St}, as detailed in Section I.6.4 of loc. cit.
\end{proof}

In the periodic case, since the weight still belongs to the class $A_2$,
Lemma \ref{Z-bound} carries over. Moreover we have $\|u\|_Z\lesssim R^{2/3}\|u\|_{H^8}$, so the $C_*^6$ and $Z$-norms are both automatically continuous in time once we have local wellposedness.

\subsection{Acknowledgements}
The author wants to thank his advisor Alexandru Ionescu for his constant help and unfailing encouragement throughout the completion of this work.
He is also grateful to the partial support of the ERC Advanced Grant 788250.

\section{Local wellposedness and bootstrap propositions}\label{Local}

\subsection{Local wellposedness}
The local wellposedness of the 3D GWW equation was shown in Theorem 7.1 of \cite{Wu3DL} and Theorem 1.2 of \cite{AlBuZu2}.
By the regularization argument in Section 6 of \cite{AlBuZu2},
it then follows from the {\it a priori} energy estimate (\ref{growth-Em-X})
that the 3D GWW equation is also locally wellposed in the energy space of our choice. More precisely, we have

\begin{proposition}\label{local-existence}
Assume $N\ge11$ and (\ref{HN0}) with $\epsilon$ sufficiently small.
Then there is $\tilde U\in C([0,1],H^N)\cap C^1([0,1],H^{N-1})$ solving (\ref{Zakharov}) with initial data $U_0$.
\end{proposition}

It follows from Proposition \ref{local-existence} and the embedding $H^7 \subset C_*^6$ that
\begin{proposition}\label{C6-fin-cont}
If $N\ge11$, $T>0$ and $\sup_{t\in[0,T]}\|U(t)\|_{H^N}$ is sufficiently small, then $\|U(t)\|_{C_*^6}$ is finite and continuous on $[0,T]$.
\end{proposition}

In the periodic case, since the machinery of paradifferential calculus can be generalized straightforwardly (see section \ref{LinEst-per}),
we know that the 3D GWW equation is locally wellposed on the torus, in the energy space of our choice by the energy estimate (\ref{growth-Em-X-per}).
More precisely, Propositions \ref{local-existence} and \ref{C6-fin-cont} continue to hold in the periodic case. The $Z$-norm is also bounded and continuous because the weight $\lesssim R^{2/3}$.

\subsection{Bootstrap propositions}
In this subsection we lay out the bootstrap propositions and use them to show Theorems \ref{Thm1}, \ref{Thm2}, \ref{Thm1-per} and \ref{Thm2-per}.
Thoughout the paper we put
\begin{align*}
\mathcal L&=\log(t+1), & \mathcal L_R&=\log R.
\end{align*}
We will use $\mathcal L$ in the Euclidean case and in subsection \ref{ZEst-per} of the periodic case, and $\mathcal L_R$ in subsections \ref{LinEst-per} and \ref{StrEst-per} of the periodic case.

In the Euclidean case the bootstrap propositions are the following:
\begin{proposition}\label{bootstrap1}
Fix $N\ge11$. Assume (\ref{HN0}) holds with $\ep$ small enough. Also assume
\begin{equation}\label{growthX1}
\begin{aligned}
\|\tilde U\|_{L^\infty([0,t])H^N}&\le\ep_1,\\ \|U\|_{L^2([0,t])C_*^6}&\le\ep_2,
\end{aligned}
\end{equation}
with $\ep_1$, $\ep_2$ small enough. Then
\begin{align}
\label{growth-Em-X}
\|\tilde U\|_{L^\infty([0,t])H^N}&\lesssim \ep+\ep_1^{3/2}+\sqrt{\mathcal L}\cdot\ep_1\ep_2,\\
\label{growthX2}
\|U\|_{L^2([0,t])C_*^6}&\lesssim \ep_1(\sqrt{\mathcal L}+\ep_2).
\end{align}
\end{proposition}

\begin{proof}[Proof of Theorem \ref{Thm1}]
We can choose $\ep_1\approx\ep$, $T_0\approx\exp(c/\ep)$,
and $\ep_2\approx\sqrt\ep$ such that for $t\le T_0$,
(\ref{growth-Em-X}) and (\ref{growthX2}) give (\ref{growthX1}) with the strict inequality.
Now the result follows from local wellposedness (Proposition \ref{local-existence}) and continuity of the $C_*^6$ norm (Proposition \ref{C6-fin-cont}).
\end{proof}

\begin{proposition}\label{bootstrap2}
Fix $N \ge \max(33/(\alpha - \alpha^2),8/\alpha^2)$ and $\alpha \in (0,1)$.
Assume (\ref{Z0}) holds with $\ep$ small enough.
Define the profile $\Upsilon(t)=e^{it\Lambda}U(t)$. Assume
\begin{equation}\label{growthZ}
\begin{aligned}
\|\tilde U\|_{L^\infty([0,t])H^N}&\le\ep_1,\\
\|(1+s)^{(\alpha-\delta)/2}P_kU(s)\|_{L^2([0,t])L^\infty}
&\le2^{-(7-\alpha)k^++k^-/2}\ep_1,\quad k\in\Z,\\
\sup_{[0,t]} \|\Upsilon\|_Z&\le\ep_1,
\end{aligned}
\end{equation}
with $\ep_1$, $\delta>0$ small enough. Then
\begin{align}
\label{growth-Em-Z}
\|\tilde U\|_{L^\infty([0,t])H^N}&\lesssim \ep+\ep_1^{3/2},\\
\label{growth-L2Loo}
\|(1+s)^{(\alpha-\delta)/2}P_kU(s)\|_{L^2([0,t])L^\infty}
&\lesssim_\delta 2^{-(7-\alpha)k^++k^-/2}(\ep+\ep_1^2),\quad k\in\Z,\\
\sup_{[0,t]}\|\Upsilon\|_Z&\lesssim \ep+\ep_1^2.
\label{growthZ2}
\end{align}
\end{proposition}

\begin{proof}[Proof of Theorem \ref{Thm2}]
We can choose $\ep_1\approx\ep$ such that for all $t\ge 0$, (\ref{growth-Em-Z}) and (\ref{growthZ2}) give (\ref{growthZ}) with the strict inequality. Now the result follows from local wellposedness (Proposition \ref{local-existence}), continuity of the $C_*^6$ norm (Proposition \ref{C6-fin-cont}) and continuity of the $Z$ norm, which can be shown using a similar argument to Section 3 of \cite{IoPa2}.
\end{proof}

In the periodic case the bootstrap assumptions are the following:
\begin{proposition}\label{bootstrap1-per}
Fix $N\ge11$. Assume (\ref{HN0}) holds with $\ep$ small enough. Also assume
\begin{equation}\label{growthX1-per}
\begin{aligned}
\|\tilde U\|_{L^\infty([0,t])H^N}&\le\ep_1,\\ \|U\|_{L^2([0,t])C_*^6}&\le\ep_2,
\end{aligned}
\end{equation}
with $\ep_1$, $\ep_2$ small enough. Then (\ref{growth-Em-X}) still holds.
In addition we have
\begin{align}
\label{growth-Em-X-per}
\|\tilde U\|_{L^\infty([0,t])H^N}&\lesssim \ep+\ep_1^{3/2}+\sqrt{\mathcal L_R}\cdot\ep_1\ep_2,\\
\label{growthX2-per}
\|U\|_{L^2([0,t])C_*^6}&\lesssim \ep_1(\sqrt{\mathcal L_R(1+t/R)}+\ep_2).
\end{align}
\end{proposition}

\begin{proof}[Proof of Theorem \ref{Thm1-per}]
We can choose $\ep_1\approx\ep$, $T_0\approx R/(\ep^2(\log R)^2)$,
and $\ep_2\approx(\log R)^{-1/2}$ such that for $t\le T_0$,
(\ref{growth-Em-X-per}) and (\ref{growthX2-per}) give (\ref{growthX1-per}) with the strict inequality. Then the proof is similar to that of Theorem \ref{Thm1}. Note that we need $R\le\exp(c/\ep)$ to recover the $L^2C_*^6$ norm.
\end{proof}

\begin{proposition}\label{bootstrap2-per}
Fix $N \ge 41$. Assume (\ref{Z0}) holds with $\ep$ small enough.
Define the profile $\Upsilon(t)=e^{it\Lambda}U(t)$. Assume
\begin{equation}\label{growthZ-per}
\begin{aligned}
\|\tilde U\|_{L^\infty([0,t])H^N}&\le\ep_1,\\
\|(1+s)^{1/3-\delta}P_kU(s)\|_{L^2([0,t])L^\infty}&\le2^{-19k^+/3+k^-/2}\ep_2,\quad k\in\Z,\\
\sup_{[0,t]} \|\Upsilon\|_Z&\le\ep_1,\\
t&\le R^{2-\delta}.
\end{aligned}
\end{equation}
with $\ep_1$, $\delta > 0$ small enough ($\ep_2$ is not assumed to be small). Then
\begin{align}
\label{growth-Em-Z-per}
\|\tilde U\|_{L^\infty([0,t])H^N}&\lesssim \ep+\ep_1^{3/2}+t^{6/5}\ep_1^2/R^{4/3},\\
\nonumber
\|(1+s)^{1/3-\delta}P_kU(s)\|_{L^2([0,t])L^\infty}
&\lesssim_\delta 2^{-19k^+/3+k^-/2}(\ep_1\ep_2+(t/R+1)^{3/2}C(t,R)),\\
\label{growth-L2Loo-per}
&\quad k\in\Z,\\
\sup_{[0,t]}\|\Upsilon\|_Z&\lesssim C(t,R),
\label{growthZ2-per}
\end{align}
where we put $\beta = 5/(N - 8)$ and
\begin{align*}
C(t,R)&=\ep+(t^{5/3+6/N}/R^2+1)\ep_1^2+(t^{11/3+13.6/N}/R^4+1)\ep_1^3\\
&+[(1+t)^{-\frac{2}{3}+}(t/R+1)^{4/3}]^{2(1-\beta)/3}\ep_1^2\ep_2^2.
\end{align*}
\end{proposition}

\begin{proof}[Proof of Theorem \ref{Thm2-per}]
We can choose $\ep_1\approx\ep$, $T_0\approx\frac{R^{12/(11+40.8/N)}}{\ep^{6/(11+40.8/N)}}$ and $\ep_2\approx(T_0/R)^{3/2}\ep_1$ such that for $t\le T_0$, (\ref{growth-Em-Z-per}) and (\ref{growthZ2-per}) give (\ref{growthZ-per}) with the strict inequality, and $T_0 \le R^{2-O(1/N)}$.
Then the proof is similar to that of Theorem \ref{Thm2}. Note that the cubic term in $C(t,R)$ imposes the most restrictions on $R$ and $\epsilon$.
\end{proof}

\section{Linear dispersive and multilinear paraproduct estimates}\label{LinEst}

\subsection{Linear dispersive estimates}
The first ingredient in the proof of global existence is dispersive estimates. Here we only deal with the Euclidean case, relegating the periodic case to section \ref{Period}.
\begin{lemma}\label{dispersive}
For $k \in \Z$, $1 \le p \le q \le \infty$ with $1/p + 1/q = 1$ we have
\footnote{In the Euclidean case better bounds are possible,
but we state them as is for better uniformity with the periodic case.}
\begin{align}
\label{dispersive1}
\|P_ke^{-it\Lambda}u\|_{L^q}&\lesssim [(1+2^{k/2}t)^{-1}2^{2k}]^{1/p-1/q}\|u\|_{L^p}\\
\label{dispersive2}
&\lesssim [(1+t)^{-1}2^{(3k^++k)/2}]^{1/p-1/q}\|u\|_{L^p},\\
\|P_ke^{-it\Lambda}u\|_{L^\infty}&\lesssim 2^{(3k^++k)/2}(1+t)^{-(1/p-1/q)}\|u\|_{L^p}.
\label{dispersive3}
\end{align}
\end{lemma}
\begin{proof}
To show (\ref{dispersive1}), by interpolation and unitarity of $e^{-it\Lambda}$ we can assume $p=1$ and $q=\infty$.

Since $\Lambda$ is homogeneous of degree $1/2$, the evolution operator $e^{it\Lambda}$ is invariant under the scaling $(x,t)\mapsto(2^kx,2^{k/2}t)$. Hence we can assume $k=0$.

Now, when $t\le1$, the result follows from the Bernstein inequality and unitarity of $e^{it\Lambda}$:
\[
\|P_0e^{-it\Lambda}u\|_{L^\infty}\lesssim \|P_0e^{-it\Lambda}u\|_{L^2}
=\|P_0u\|_{L^2}\lesssim \|P_0u\|_{L^1}\lesssim \|u\|_{L^1}.
\]
When $t>1$, $1+t\approx t$, so the result follows from Theorem 1 (a) of \cite{GuPeWa}.

To get (\ref{dispersive2}) we simply use the inequality $1+2^{k/2}t\ge2^{k^-/2}(1+t)$.

To get (\ref{dispersive3}) we also need the Bernstein inequality
\begin{equation}\label{dispersive-Bernstein}
\|P_ke^{-it\Lambda}u\|_{L^\infty}\lesssim 2^{2k/q}\|P_ke^{-it\Lambda}u\|_{L^q}.
\end{equation}
\end{proof}

\begin{lemma}\label{dispersiveTT*}
(i) For $k\in\Z$ we have
\begin{align*}
\|P_ke^{-is\Lambda}u\|_{L^2([0,t])L^\infty}&\lesssim c_{k,t}\|u\|_{L^2}, &
c_{k,t}&=2^{3k/4}(\sqrt{k^+}+\sqrt{\mathcal L}).
\end{align*}

(ii)
\[
\|e^{-is\Lambda}u\|_{L^2([0,t])W^{6,\infty}}\lesssim \sqrt{\mathcal L}\|u\|_{H^7}.
\]
\end{lemma}
\begin{proof}
Let
\begin{align*}
T_k&:L_x^2\to L^2([0,t])L_x^\infty, & u&\mapsto P_ke^{-is\Lambda}u.
\end{align*}
Then
\begin{align*}
T_k^*&:L^2([0,t])L_x^1\to L_x^2, & u&\mapsto \int_0^t P_ke^{is\Lambda}u(s)ds.
\end{align*}
Then $\|T_k\|=\|T_k^*\|=\|T_kT_k^*\|^{1/2}$, where
\begin{align*}
T_kT_k^*&:L^2([0,t])L_x^1\to L^2([0,t])L_x^\infty, & u&\mapsto \int_0^t P_k^2e^{i(s'-s)\Lambda}u(s')ds'.
\end{align*}

By (\ref{dispersive1}),
\[
\|P_k^2e^{i(s'-s)\Lambda}u(s')\|_{L_x^\infty}\lesssim (1+2^{k/2}|s'-s|)^{-1}2^{2k}\|u(s)\|_{L_x^1}.
\]
Then by Young's inequality,
\begin{align*}
\|T_kT_k^*u\|_{L^2([0,t])L_x^\infty}&\lesssim \|(1+2^{k/2}|\cdot|)^{-1}\|_{L^1([0,t])}2^{2k}\|u\|_{L^2([0,t])L_x^1}\\
&=2^{3k/2}\log(1+2^{k/2}t)\|u\|_{L^2([0,t])L^1_x}
\lesssim c_{k,t}^2\|u\|_{L^2([0,t])L^1_x},
\end{align*}
where in the last line we have used the bound $1+2^{k/2}t\le2^{k^+/2}(1+t)$.

Summing (i) over $k\in\Z$ we get (ii).
\end{proof}

\begin{lemma}\label{dispersiveZ}
(i) For $k \in \Z$ and $\alpha \in (0, 1)$ we have
\[
\|P_ke^{-it\Lambda}u\|_{L^\infty}
\lesssim 2^{-(7-\alpha)k^++k^-/2}(1+t)^{(-\alpha)+}\|u\|_Z.
\]

(ii)
\[
\|e^{-it\Lambda}u\|_{W^{6,\infty}}\lesssim (1+t)^{(-\alpha)+}\|u\|_Z.
\]

(iii)
\[
\|(1+t)^{\alpha-/2}P_ke^{-it\Lambda}u\|_{L_t^2L_x^\infty}
\lesssim 2^{-(6+6(1-\alpha)/5)k^++k^-/2}\|u\|_Z.
\]
\end{lemma}
\begin{proof}
The first bound follows from (\ref{dispersive2}) and (\ref{dispersive-Bernstein}) (with $q = \frac2{1-\alpha}-$ and $p = \frac2{1+\alpha}+$),
after distinguishing the cases $k \ge 0$ and $k < 0$.
To get (ii) we sum (i) over $k \in \Z$.
For (iii), note that the endpoint cases $\alpha = 0$ and 1 follow from Lemma \ref{dispersiveTT*} (i) and Lemma \ref{dispersiveZ} (i) respectively. By complex interpolation (iii) is obtained.
\end{proof}

\subsection{Paradifferential calculus}\label{ParaCalc}
We will use Weyl quantization of paradifferential operators as laid out in Section 3.2 of \cite{DeIoPa} and Section A.1.2 of \cite{DeIoPaPu}.
\begin{definition}
Given a symbol $a = a(x, \zeta): \R^2 \times (\R^2 \backslash 0) \to \C$, define the operator $T_a$ using the following recipe:
\[
\mathcal F(T_af)(\xi)=C\int
\varphi_{\le-10}\left( \frac{|\xi-\eta|}{|\xi+\eta|} \right)\mathcal F_x a\left( \xi-\eta,\frac{\xi+\eta}2 \right)\hat f(\eta)d\eta,
\]
where $C$ is a normalization constant such that $T_1=\text{id}$.
\end{definition}
\begin{remark}
With the inclusion of the factor $\varphi_{\le-10}$, only low frequencies of $a$ and high frequencies of $f$ are involved in $T_af$.
\end{remark}

The next lemma follows directly from the definition.
\begin{lemma}\label{para2diff}
(i) If $a$ is real valued, then $T_a$ is self-adjoint.

(ii) If $a(x,-\zeta)=\overline{a(x,\zeta)}$ and $f$ is real valued, so is $T_af$.

(iii) If $a=P(\zeta)$, then $T_af=P(D)f$ is a Fourier multiplier.

(iv) For $k\in\Z$ we have $P_kT_a(P_{\le k-2}f)=0$.

(v) If $a=a(x)$, then for $k\in\Z$ we have $P_kT_{P_{\le k-20}a}f=P_k(P_{\le k-20}a\cdot f)$.
\footnote{Correcting a typo in the statement of this lemma in \cite{Zh},
which does not affact the proofs there.\label{cor}}
\end{lemma}

The following symbol norm will be used.
\begin{definition}\label{Lmp-def}
For $p\in[1,\infty]$ and $m\in\R$ define
\[
\begin{aligned}
|a|(x,\zeta)
&=\sum_{|I|\le9} |\zeta|^{|I|}|\partial_{\zeta_I} a(x,\zeta)|, &
\|a\|_{\mathcal L_m^p}
&=\sup_{\zeta\in\R^2} (1+|\zeta|)^{-m}\||a|(x,\zeta)\|_{L^p_x}.
\end{aligned}
\]
\end{definition}

Here $m$ is the {\it order} of the symbol, in the sense of H\"ormander.
\begin{lemma}\label{Sm-Lm}
A multiplier whose symbol is of class $S^m_{1,0}$ has finite $\mathcal L_m^\infty$ norm.
\end{lemma}

For functions independent of $\zeta$, the $\mathcal L_m^p$ norm agrees with the $L^p$ norm.
\begin{lemma}\label{Lmq=Lq}
If $a=a(x)$ and $m\ge0$, then $\|a\|_{\mathcal L_m^p}=\|a\|_{L^p}$.
\footnotemark[4]
\end{lemma}

The norm of the product of two operators can be bounded using Leibniz's rule and H\"older's inequality. The result is
\begin{lemma}\label{prod-Lm}
For fixed $m\in\R$, $p,q,r\in[1,\infty]$ with $1/p=1/q+1/r$ we have
\[
\|ab\|_{\mathcal L_{m+n}^p}\lesssim \|a\|_{\mathcal L_m^q}\|b\|_{\mathcal L_n^r}.
\]
\end{lemma}

Paradifferential operators in $\mathcal L_m^q$ act like differential operators of order $m$ with $L^q$ coefficients.
\begin{lemma}\label{Taf-Lp}
(i) For fixed $m\in\R$, $p,q,r\in[1,\infty]$ with $1/p=1/q+1/r$ we have
\[
\|P_kT_af\|_{L^p}\lesssim 2^{mk^+}\|a\|_{\mathcal L_m^q}\|P_{[k-2,k+2]}f\|_{L^r},\quad k\in\Z.
\]
(ii) For fixed $m$, $s\in\R$ we have $\|T_af\|_{H^s}\lesssim
\|a\|_{\mathcal L_m^q}\|f\|_{H^{s+m}}$.
\end{lemma}
\begin{proof}
For (i) see Lemma 3.4 (i) of \cite{DeIoPa}. Note only 8 $\zeta$-derivatives and no $x$-derivatives of the symbol $a$ is needed in the proof, and the restriction on the order of the symbol is unnecessary.

(ii) follows from (i) (with $p=r=2$ and $q=\infty$) after taking a weighted $\ell^2$ sum in $k\in\Z$.
\end{proof}

Paradifferential operators extract the ``quasilinear'' part of products,
leaving ``semilinear'' remainders.
\begin{definition}\label{H-def}
Given two functions $f$ and $g$, define
\[
H(f,g)=fg-T_fg-T_gf.
\]
\end{definition}

\begin{lemma}\label{PkH-Lp}
(i) For fixed $p,q,r\in[1,\infty]$ with $1/p=1/q+1/r$ we have
\[
\|P_kH(f,g)\|_{L^p}\lesssim \|P_{>k-20}f\|_{L^q}\|P_{>k-20}g\|_{L^r}.
\]

(ii) For fixed $1<q,r\le\infty$ and $1<p<\infty$ with $1/p=1/q+1/r$ we have
\[
\|H(f,g)\|_{L^p}\lesssim \|f\|_{L^q}\|g\|_{L^r}.
\]

(iii) For fixed $0\le m<s$ with $m\in\Z$ we have
\[
\|H(f,g)\|_{H^s}\lesssim \|f\|_{W^{m,\infty}}\|g\|_{H^{s-m}}.
\]
\end{lemma}
\begin{proof}
For (i) see the proof of Lemma A.4 (i) of \cite{DeIoPa}.
(ii) is the Coifman--Meyer theorem (Theorem 3.7 of \cite{MuSc}).
For (iii), taking a weighted $\ell^2$ sum over $k\ge0$ of (i) (with $p=r=2$ and $q=\infty$) and using $\|P_{>k-20}f\|_{L^\infty}\lesssim 2^{-mk}\|f\|_{W^{m,\infty}}$ and $\sum_{k=0}^\infty 2^{(s-m)k}\|P_{>k-20}g\|_{L^2}^2\lesssim \|g\|_{H^{s-m}}$ give the desired bound for $\|P_{\ge0}H(f,g)\|_{H^s}$.
For $\|P_{<0}H(f,g)\|_{H^s}\approx\|P_{<0}H(f,g)\|_{L^2}$ the desired bound comes from (ii).
\end{proof}

Next we give Taylor expansions of compositions of paradifferential operators. First we state their commutator estimates.
\begin{definition}\label{Eaf-def}
Given symbols $a_1,\dots,a_n$, define the operator
\[
E(a_1,\dots,a_n)=T_{a_1}\cdots T_{a_n}-T_{a_1\cdots a_n}.
\]
\end{definition}

Roughly speaking, the operator $E(a_1,\dots,a_n)$ is one order smoother than either term on the right, so it can be thought of as an ``error term''.
\begin{lemma}\label{Eaf-Lp}
(i) For fixed $m_j\in\R$, $p,q_j,r\in[1,\infty]$ ($j=1,\dots,n$) with $1/p=1/q_1+\cdots+1/q_n+1/r$ we have
\[
\|P_kE(a_1,\dots,a_n)f\|_{L^p}
\lesssim 2^{(\sum_{j=1}^n m_j-1)k^+}\prod_{j=1}^n (\|a_j\|_{\mathcal L_{m_j}^{q_j}}+\|\nabla_x a_j\|_{\mathcal L_{m_j}^{q_j}}) \|P_{[k-2n,k+2n]}f\|_{L^r}.
\]

(ii) For fixed $s,m_j\in\R$ we have
\[
\|E(a_1,\cdots,a_n)f\|_{H^s}
\lesssim \prod_{j=1}^n (\|a_j\|_{\mathcal L_{m_j}^{q_j}}+\|\nabla_x a_j\|_{\mathcal L_{m_j}^{q_j}}) \|f\|_{H^{s+\sum_{j=1}^n m_j-1}}.
\]
\end{lemma}
\begin{proof}
The proof of (i) is similar to, but easier than that of Lemma 3.15 of \cite{Zh}.
(ii) follows from (i) as in Lemma \ref{Taf-Lp}.
\end{proof}

We need a more precise estimate for the commutator of two operators.
\begin{definition}\label{E1f-def}
Given two symbols $a$ and $b$, define their Poisson bracket
\[
\{a,b\}=\nabla_xa\nabla_\zeta b-\nabla_\zeta a\nabla_xb.
\]
Define the operator
\[
E_1(a,b)=E(a,b)-\frac{i}{2}T_{\{a,b\}}.
\]
\end{definition}
\begin{lemma}[Cf. Proposition A.5 of \cite{DeIoPaPu}]\label{E1f-Lp}
(i) For fixed $m_1,m_2\in\R$, $p,q_1,q_2,r\in[1,\infty]$ with $1/p=1/q_1+1/q_2+1/r$ we have
\begin{align*}
\|P_kE_1(a,b)f\|_{L^p}
&\lesssim 2^{(m_1+m_2-2)k^+}(\|a\|_{\mathcal L^{q_1}_{m_1}}+\|\nabla_xa\|_{\mathcal L^{q_1}_{m_1}}+\|\nabla_x^2a\|_{\mathcal L^{q_1}_{m_1}})\\
&\times (\|b\|_{\mathcal L^{q_2}_{m_2}}+\|\nabla_xb\|_{\mathcal L^{q_2}_{m_2}}+\|\nabla_x^2b\|_{\mathcal L^{q_2}_{m_2}})\|P_{[k-4,k+4]}f\|_{L^r}.
\end{align*}

(ii) For fixed $s,m_1,m_2\in\R$ we have
\begin{align*}
\|E_1(a,b)f\|_{H^s}
&\lesssim (\|a\|_{\mathcal L^{q_1}_{m_1}}+\|\nabla_xa\|_{\mathcal L^{q_1}_{m_1}}+\|\nabla_x^2a\|_{\mathcal L^{q_1}_{m_1}})\\
&\times (\|b\|_{\mathcal L^{q_2}_{m_2}}+\|\nabla_xb\|_{\mathcal L^{q_2}_{m_2}}+\|\nabla_x^2b\|_{\mathcal L^{q_2}_{m_2}})\|f\|_{H^{s+m}}.
\end{align*}
When $a=a(x)$ and $b=b(x)$ are independent of $\zeta$,
their Poisson bracket vanishes, so $E(a,b)=E_1(a,b)$.
\end{lemma}

Roughly speaking, the operator $E_1(a,b)$ is two orders smoother than $T_aT_b$.
\begin{proof}
As before (ii) follows from (i). To show (i), note that
as in the proof of Lemma 3.15 of \cite{Zh} we can assume $k>0$,
$a=P_{\le k-20}a$ and $b=P_{\le k-20}b$, in which case we have
\[
\mathcal F(P_kE_1(a,b)f)(\xi)
=C^2\iint \hat A(\xi,\eta,\zeta)\hat f(\zeta)d\eta d\zeta,
\]
where
\begin{align*}
\hat A&=\mathcal F_xa\left( \xi-\eta,\frac{\xi+\eta}2 \right)
\mathcal F_xb\left( \eta-\zeta,\frac{\eta+\zeta}2 \right)\\
&-\mathcal F_xa\left( \xi-\eta,\frac{\xi+\zeta}2 \right)
\mathcal F_xb\left( \eta-\zeta,\frac{\xi+\zeta}2 \right)-\hat B,\\
\hat B&=\frac{i}{2}\mathcal F_x(\nabla_xa)\left( \xi-\eta,\frac{\xi+\zeta}2 \right)\mathcal F_x(\nabla_\zeta b)\left( \eta-\zeta,\frac{\xi+\zeta}2 \right)\\
\displaybreak[0]
&-\frac{i}{2}\mathcal F_x(\nabla_\zeta a)\left( \xi-\eta,\frac{\xi+\zeta}2 \right)\mathcal F_x(\nabla_xb)\left( \eta-\zeta,\frac{\xi+\zeta}2 \right)\\
&=\mathcal F_xa\left( \xi-\eta,\frac{\xi+\zeta}2 \right)\frac{\eta-\xi}2\cdot(\nabla_\zeta\mathcal F_xb)\left( \eta-\zeta,\frac{\xi+\zeta}2 \right)\\
&+(\nabla_\zeta\mathcal F_xa)\left( \xi-\eta,\frac{\xi+\zeta}2 \right)
\frac{\eta-\zeta}2\cdot\mathcal F_xb\left( \eta-\zeta,\frac{\xi+\zeta}2 \right).
\end{align*}
By Taylor's theorem,
\[
\hat A=\int_0^1 R(t)(1-t)dt,
\]
where
\begin{align*}
R(t)&=\frac{d^2}{dt^2}\left[ \mathcal F_xa\left( \xi-\eta,\frac{\xi+\zeta+t(\eta-\zeta)}2 \right)\mathcal F_xb\left( \eta-\zeta,\frac{\xi+\zeta+t(\eta-\xi)}2 \right) \right]\\
&=-\frac14(\mathcal F_x\nabla_\zeta^2a)\left( \xi-\eta,\frac{\xi+\zeta+t(\eta-\zeta)}2 \right)\cdot(\mathcal F_x\nabla_x^2b)\left( \eta-\zeta,\frac{\xi+\zeta+t(\eta-\xi)}2 \right)\\
&-\frac12(\mathcal F_x\nabla_\zeta\nabla_xa)\left( \xi-\eta,\frac{\xi+\zeta+t(\eta-\zeta)}2 \right)\cdot(\mathcal F_x\nabla_\zeta\nabla_xb)\left( \eta-\zeta,\frac{\xi+\zeta+t(\eta-\xi)}2 \right)\\
&-\frac14(\mathcal F_x\nabla_x^2a)\left( \xi-\eta,\frac{\xi+\zeta+t(\eta-\zeta)}2 \right)\cdot(\mathcal F_x\nabla_\zeta^2b)\left( \eta-\zeta,\frac{\xi+\zeta+t(\eta-\xi)}2 \right),
\end{align*}
where the dot products denote contractions between two 2-tensors.
Now (i) follows in the same way as Lemma 3.15 of \cite{Zh}.
\end{proof}

\subsection{Multilinear paraproduct estimates}
We also need to bound multilinear paraproducts.
For further reference the reader is referred to Section 3.1 of \cite{DeIoPa} and Section A.1.1 of \cite{DeIoPaPu}.
\begin{definition}\label{Soo-def}
If $m$ is a Schwartz function on $(\R^2)^n$, define
\begin{align*}
\|m\|_{S^\infty}&=\|\mathcal Fm\|_{L^1},\\
\|m\|_{S^\infty_{k_1,\cdots,k_n;k}}&=\|\varphi_k(\xi_1+\cdots+\xi_n)m(\xi_1,\cdots,\xi_n)\varphi_{k_1}(\xi_1)\cdots\varphi_{k_n}(\xi_n)\|_{S^\infty}.
\end{align*}
\end{definition}

Straight from the definition, we get the embedding $S^\infty\subset L^\infty$.

The next lemma allows us to estimate the $S^\infty$ norm of various symbols.
\begin{lemma}\label{Soo-Cn}
(i) $\|m_1m_2\|_{S^\infty}\le\|m_1\|_{S^\infty}\|m_2\|_{S^\infty}$.

(ii) For $k_1,\dots,k_n,k\in\Z$ we have
\[
\|m\cdot\otimes_{j=1}^n\varphi_{k_j}\|_{S^\infty}
\lesssim_n \sum_{l=0}^{n+1} \sum_{j=0}^n 2^{lk_j}\|
\varphi_{[k_j-1,k_j+1]}\nabla_j^lm\|_{L^\infty}.
\]
\end{lemma}
\begin{proof}
For (i) and (ii), see (3.4) and Lemma 3.3 of \cite{DeIoPa}, respectively.
\end{proof}

The $L^p$ boundedness of a paraproduct of functions is well known.
\begin{lemma}\label{paraprod}
Fix $p,p_j\in[1,\infty]$ ($j=1,\dots,n$) and $1/p=1/p_1+\cdots+1/p_n$. Let
\[
\mathcal Ff(\xi)
=\iiint_{\xi_1+\dots+\xi_n=\xi} m(\xi_1,\dots,\xi_n)\prod_{j=1}^n \mathcal F f_j(\xi_j).
\]
Then
\[
\|f\|_{L^p}\lesssim \|m\|_{S^\infty}\prod_{j=1}^n\|f_j\|_{L^{p_j}}.
\]
\end{lemma}
\begin{proof}
See Lemma 5.2 (ii) of \cite{IoPu} for the case of $n=2$, the general case being no different.
\end{proof}

\section{Paralinearization of the Zakharov system}\label{ParLin}
In this section we transform the fully nonlinear Zakharov system (\ref{Zakharov}) to a quasilinear one.

\subsection{Fixed-point formulation of the Dirichlet-to-Neumann operator}\label{BVP}
Recall that in the Zakharov formulation (\ref{Zakharov}),
the problem is recast in terms of the height of the fluid $h$,
and the boundary value $\phi$ of the velocity potential $\Phi$.
The evolution of the height $h$ is
\begin{equation}\label{ht=Gh-phi}
h_t=G(h)\phi,
\end{equation}
where $\phi(x,t)=\Phi(x,h(x,t),t)$ and
\begin{equation}\label{Gh-def}
G(h)\phi=\sqrt{1+|\nabla h|^2}\partial_n\Phi
\end{equation}
is the Dirichlet-to-Neumann operator, which we now express in terms of paradifferential operators acting on $\phi$ and $h$, plus a good error term. The process is similar to the case of gravity-capillary waves done in Appendix B of \cite{DeIoPaPu}.

We start from the velocity potential $\Phi$, which is harmonic on
\[
\Omega(t)=\{(x,z)\in\R^2\times\R: z<h(x,t)\}
\]
and has the boundary value $\phi$ on
\[
\Gamma(t)=\{(x,h(x,t)): x\in\R^2\}.
\]
In other words, it solves the Dirichlet problem
\begin{equation}\label{phi-Diri}
\begin{cases}
\Delta_{x,z}\Phi=0\text{ on $\Omega(t)$},\\
\Phi(x,h(x,t),t)=\phi(x,t),\\
\Phi(x,z)\to0\text{ as $z \to -\infty$ in a sense to be clarified later.}
\end{cases}
\end{equation}
The Dirichlet-to-Neumann operator can then be expressed as
\begin{equation}\label{Gh-phi}
G(h)\phi=[\partial_z\Phi-\nabla h\cdot\nabla_x\Phi]_\Gamma=B-\nabla h\cdot V,
\end{equation}
where $B=[\partial_z\Phi]_\Gamma$ and $V=[\nabla_x\Phi]_\Gamma$ are vertical and horizontal components of the boundary velocity, respectively.

Next we flatten the boundary $\Gamma(t)$ using a change of variable. Let
\begin{equation}\label{u-def}
\begin{aligned}
u(x,y,t)&=\Phi(x,h(x,t)+y,t), & (x,y)&\in\R^2\times(-\infty,0),\\
\text{i.e. }\Phi(x,z,t)&=u(x,z-h(x,t),t), & (x,z)&\in\Omega(t).
\end{aligned}
\end{equation}
Then $u$ solves the elliptic boundary value problem
\begin{equation}\label{u-Diri}
\begin{cases}
(1+|\nabla h|^2)\partial_y^2u+\Delta_xu-2\nabla h\cdot\partial_y\nabla_xu-\Delta h\partial_yu=0,\\
u(x,0,t)=\phi(x,t),\\
\nabla_{x,y}u(x,y,t)\to0\text{ as $y \to -\infty$.}
\end{cases}
\end{equation}
We have
\begin{equation}\label{BV-u}
\begin{cases}
B=\partial_yu|_{y=0},\\
V=[\nabla_xu-\nabla h\partial_yu]_{y=0}=\nabla\phi-B\nabla h
\end{cases}
\end{equation}
and the Dirichlet-to-Neumann operator becomes
\begin{equation}\label{DN-u}
G(h)\phi=[(1+|\nabla h|^2)\partial_yu-\nabla h\cdot\nabla_xu]_{y=0}.
\end{equation}

Here are some useful estimates of the solution $u$ to (\ref{u-Diri}),
whose proofs can be found in section \ref{BVPEst}.
\begin{proposition}\label{u-Hs}
Fix $s>j\ge1$.

(i) If $\|\nabla h\|_{H^s}<c_s$ is sufficiently small then
\[
\|\nabla_{x,y}^ju\|_{L_y^2H_x^{s-j+1}}\lesssim \||\nabla|^{1/2}\phi\|_{H^s}.
\]

(i') If in addition we have $s>j+1$ then $\|\nabla_{x,y}^ju\|_{H_x^{s-j}}\to0$ as $y\to-\infty$.

(ii) If $\|\nabla h\|_{H^{s+1/2}}<c_s$ is sufficiently small then
\[
\|\nabla_{x,y}^ju\|_{L_y^\infty H_x^{s-j+1}}\lesssim \||\nabla|^{1/2}\phi\|_{H^{s+1/2}}.
\]
\end{proposition}

Applying Proposition \ref{u-Hs} (ii) to (\ref{BV-u}) and (\ref{DN-u}),
under the same condition we have
\begin{equation}\label{Gh-Hs}
\|B(h)\|_{H^s}+\|V(h)\|_{H^s}+\|G(h)\phi\|_{H^s}\lesssim_s \||\nabla|^{1/2}\phi\|_{H^{s+1/2}}\lesssim \|\phi\|_{H^{s+1}}.
\end{equation}

We will also use Besov norms in our estimates. For $r>0$ define
\[
\|u\|_{B^r_{p,q}}=\|P_{<0}u\|_{L^p}+\|2^{rk}\|P_ku\|_{L^p}\|_{\ell^q_{k\ge0}}.
\]
Let $C_*^r=B^r_{\infty,\infty}$ denote the H\"older--Zygmund norm,
and $X^r$ denote the norm $B^r_{\infty,2}$.
\begin{proposition}\label{u-Cr}
Fix $1\le j<r+1$. If $\|\nabla h\|_{H^{r+1}}<c_r$ is small enough then:

(i) For any integer $k<0$,
\[
\|\nabla_{x,y}^ju\|_{L_y^2X_x^{r-j+1}}\lesssim |k|\||\nabla|^{1/2}\phi\|_{X^r}+2^k\||\nabla|^{1/2}\phi\|_{L^2}.
\]

(i') If $r>j$ then $\|\nabla_{x,y}^ju\|_{X_x^{r-j}}\to0$ as $y\to-\infty$.

(ii)
\[
\|\nabla_{x,y}^ju\|_{L_y^\infty C_*^{r-j+1}}\lesssim \||\nabla|^{1/2}\phi\|_{C_*^{r+1/2}}.
\]
\end{proposition}

Applying Proposition \ref{u-Cr} (ii) to (\ref{BV-u}) and (\ref{DN-u}),
under the same condition we have
\begin{equation}\label{Gh-Cr}
\|B\|_{C_*^r}+\|V\|_{C_*^r}+\|G(h)\phi\|_{C_*^r}
\lesssim_r \||\nabla|^{1/2}\phi\|_{C_*^{r+1/2}}
\lesssim_r \|\phi\|_{C_*^{r+1}}.
\end{equation}

For future use we also need quadratic error bounds for $u-e^{y|\nabla|}\phi$.
\begin{proposition}\label{u-Hs2}
Fix $s>1$.

(i) If $\|\nabla h\|_{H^s}<c_s$ is sufficiently small then for any integer $k<0$,
\[
\|\nabla_{x,y}(u-e^{y|\nabla|}\phi)\|_{L_y^2H_x^s}
\lesssim \|\nabla h\|_{L^\infty}\||\nabla|^{1/2}\phi\|_{H^s}
+\|\nabla h\|_{H^s}C_{k,1}[|\nabla|^{1/2}\phi],
\]
where
\begin{equation}\label{Ckr-def}
C_{k,r}[f]=|k|\|f\|_{C_*^r}+2^k\|f\|_{L^2}.
\end{equation}

(ii) If $\|\nabla h\|_{H^{s+1/2}}<c_s$ is sufficiently small then for any integer $k<0$,
\[
\|\nabla_{x,y}(u-e^{y|\nabla|}\phi)\|_{L_y^\infty H_x^s}
\lesssim \|\nabla h\|_{L^\infty}\||\nabla|^{1/2}\phi\|_{H^{s+1/2}}
+\|\nabla h\|_{H^{s+1/2}}C_{k,1}[|\nabla|^{1/2}\phi].
\]
\end{proposition}

\subsection{Paralinearization of the Zakharov system}
Define
\begin{equation}\label{w-def}
w=u-T_{\partial_yu}h
\end{equation}
whose restriction to the boundary $\Gamma$ is the ``Alinhac good unknown".
Define the paradifferential operator
\begin{equation}\label{P-def}
\mathcal Pu=T_{1+|\nabla h|^2}\partial_y^2u+\Delta_xu-2T_{\nabla h}\partial_y\nabla_xu-T_{\Delta h}\partial_yu.
\end{equation}
By (B.32) and (B.33) of \cite{DeIoPaPu} and the definition of $E$,
the elliptic equation in (\ref{u-Diri}) can be written as
\begin{equation}\label{Pw}
\mathcal Pw=\mathcal S_0+\mathcal C_0
\end{equation}
where
\begin{equation}\label{S0}
\mathcal S_0=\mathcal S_0[h,u]=2H(\partial_y\nabla_xu,\nabla h)+H(\partial_yu,\Delta h)
\end{equation}
is the semilinear term, and
\begin{equation}\label{C0}
\begin{aligned}
\mathcal C_0&=-\partial_y(T_{|\nabla h|^2}T_{\partial_y^2u}-T_{|\nabla h|^2\partial_y^2u}-2(T_{\nabla h}T_{\partial_y\nabla_xu}-T_{\nabla h\cdot\partial_y\nabla_xu})\\
&-(T_{\Delta h}T_{\partial_yu}-T_{\Delta h\partial_yu}))h
+2(E(\nabla h,\partial_y^2u)-E(\partial_y^2u,\nabla h))\nabla h
\\
&-T_{\partial_y^2u}H(\nabla h,\nabla h)-H(|\nabla h|^2,\partial_y^2u)
\end{aligned}
\end{equation}
is the cubic term.

First we factorize the paradifferential operator $\mathcal P$ modulo acceptable error. Let $\alpha=|\nabla h|^2$. We use the ansatz
\begin{equation}\label{factor}
\begin{aligned}
\mathcal P &\approx (T_{\sqrt{1+\alpha}}\partial_y - A + B)(T_{\sqrt{1+\alpha}}\partial_y - A - B)\\
&= T_{\sqrt{1+\alpha}}^2\partial_y^2 - (AT_{\sqrt{1+\alpha}} + T_{\sqrt{1+\alpha}}A + [T_{\sqrt{1+\alpha}}, B])\partial_y + A^2 - B^2 + [A, B].
\end{aligned}
\end{equation}
Let
\begin{equation}\label{AB-ab}
\begin{aligned}
\mathcal A&=iT_a, & \mathcal B&=T_b.
\end{aligned}
\end{equation}
Then
\begin{equation}\label{AB-eqn}
\begin{aligned}
2a\sqrt{1+\alpha}+\{\sqrt{1+\alpha},b\}&\approx2\nabla h\cdot\zeta,\\
a^2+b^2+\{a,b\}&\approx|\zeta|^2.
\end{aligned}
\end{equation}
Let further
\begin{equation}\label{ab-10}
\begin{aligned}
a&=a^{(1)}+a^{(0)}, & b&=b^{(1)}+b^{(0)},
\end{aligned}
\end{equation}
where $a^{(j)}$ and $b^{(j)}$ are symbols homogeneous of degree $j$ in $\zeta$. Then comparing homogeneous parts of the symbols yields
(see also the derivation after (B.34) in \cite{DeIoPaPu}.)
\begin{equation}\label{ab-eqn}
\begin{aligned}
2a^{(1)}\sqrt{1+\alpha}&=2\nabla h\cdot\zeta,\\
2a^{(0)}\sqrt{1+\alpha}+\{\sqrt{1+\alpha},b^{(1)}\}\varphi_{\ge0}(\zeta)
&=0,\\
(a^{(1)})^2+(b^{(1)})^2&=|\zeta|^2,\\
2a^{(1)}a^{(0)}+2b^{(1)}b^{(0)}+\{a^{(1)},b^{(1)}\}\varphi_{\ge0}(\zeta)
&=0,
\end{aligned}
\end{equation}
which can be satisfied by (see (B.35)--(B.38) of \cite{DeIoPaPu})
\begin{equation}\label{ab-soln}
\begin{aligned}
a^{(1)}&=\frac{\nabla h\cdot\zeta}{\sqrt{1+\alpha}}
=\nabla h\cdot\zeta(1+M_2(\nabla h)),\\
b^{(1)}&=\sqrt{|\zeta|^2-(a^{(1)})^2}=|\zeta|(1+M_2(\zeta,\nabla h)),\\
a^{(0)}&=-\frac{\{\sqrt{1+\alpha},b^{(1)}\}}{2\sqrt{1+\alpha}}
\varphi_{\ge0}(\zeta)
=M_2(\zeta,\nabla h,\nabla^2h)\varphi_{\ge0}(\zeta),\\
b^{(0)}&=-\frac1{2b^{(1)}}\left( -\frac{\nabla h\cdot\zeta}{1+\alpha}
\{\sqrt{1+\alpha},b^{(1)}\}+\left\{ \frac{\nabla h\cdot\zeta}{\sqrt{1+\alpha}},b^{(1)} \right\} \right)\varphi_{\ge0}(\zeta)
\\
&=-\frac{\sqrt{1+\alpha}}{2b^{(1)}}\left\{ \frac{\nabla h\cdot\zeta}{1+\alpha},b^{(1)} \right\}\varphi_{\ge0}(\zeta)\\
&=b_1^{(0)}+M_3(\zeta,\nabla h,\nabla^2h)\varphi_{\ge0}(\zeta),\\
b_1^{(0)}&=-\frac{\zeta^T(\nabla^2h)\zeta}{2|\zeta|^2}\varphi_{\ge0}(\zeta),
\end{aligned}
\end{equation}
where $M_k$ denotes a function that vanishes to order $k$ near $h=0$,
that is homogeneous of degree $0$ in $\zeta$, and that is smooth in all its arguments away from $\zeta=0$
(for example, $\sqrt{1+\alpha}=1+M_2(\nabla h)$, hence the first line.) 

The ansatz (\ref{factor}) is an approximation in the following sense.
By (B.39) of \cite{DeIoPaPu},
\begin{equation}\label{P-factor}
(T_{\sqrt{1+\alpha}}\partial_y-\mathcal{A+B})(T_{\sqrt{1+\alpha}}\partial_y-\mathcal{A-B})
=\mathcal P+\mathcal R_1+\mathcal R_2,
\end{equation}
where, according to (\ref{ab-soln}),
\begin{equation}\label{R1}
\begin{aligned}
\mathcal R_1&=-E(|\zeta|,b_1^{(0)})-E(b_1^{(0)},|\zeta|)
+i[T_{\nabla h\cdot\zeta},T_{|\zeta|}]+T_{\{\nabla h\cdot\zeta,|\zeta|\}\varphi_{\ge0}(\zeta)}\\
&=-E(|\zeta|,b_1^{(0)})-E(b_1^{(0)},|\zeta|)
+iE_1(\nabla h\cdot\zeta,|\zeta|)-T_{\{\nabla h\cdot\zeta,|\zeta|\}\varphi_{\le-1}(\zeta)}
\end{aligned}
\end{equation}
is a paradifferential operator of order 0 with coefficients of degree 1 in $h$, and
\begin{equation}\label{R2}
\begin{aligned}
\mathcal R_2&=E_1(\sqrt{1+\alpha},\sqrt{1+\alpha})\partial_y^2
-i(E_1(a,\sqrt{1+\alpha})+E_1(\sqrt{1+\alpha},a))\partial_y\\
&-([T_{\sqrt{1+\alpha}},b^{(0)}]+E_1(\sqrt{1+\alpha},b^{(1)})-E_1(b^{(1)},\sqrt{1+\alpha})+iT_{\{\sqrt{1+\alpha},b^{(1)}\}\varphi_{\le-1}(\zeta)})\partial_y\\
&-E_1(a,a)-(E_1(b,b)-E_1(|\zeta|,b_1^{(0)})-E_1(b_1^{(0)},|\zeta|))
-T_{(a^{(0)})^2+(b^{(0)})^2}\\
&+i([T_a,T_b]-[T_{\nabla h\cdot\zeta},T_{|\zeta|}]+(T_{\{a^{(1)},b^{(1)}\}\varphi_{\ge0}(\zeta)}-T_{\{\nabla h\cdot\zeta,|\zeta|\}\varphi_{\ge0}(\zeta)})
\end{aligned}
\end{equation}
is a paradifferential operators of order 0 with coefficients of degree $\ge2$ in $h$. The operator norms of $\mathcal R_1$ and $\mathcal R_2$ can be controlled in a more precise way, see (\ref{Rj-bound}) for details.

Now we paralinearize the Zakharov system (\ref{Zakharov}). Define
\begin{equation}\label{QS}
\begin{aligned}
\mathcal Q(x,y)&=\int_{-\infty}^y e^{(y'-y)|\nabla|}\mathcal R_1[e^{y'|\nabla|}\phi](x,y')dy',\\
\mathcal S(x,y)&=\int_{-\infty}^y e^{(y'-y)|\nabla|}\mathcal S_0[h,e^{y'|\nabla|}\phi](x,y')dy'\\
&=-H(h,e^{y|\nabla|}\Delta\phi)-|\nabla|H(h,e^{y|\nabla|}|\nabla|\phi).
\end{aligned}
\end{equation}
The third line follows from (\ref{S0}) after comparing the multipliers.
\begin{proposition}\label{QS-bound}
For $s\in\R$ we have (here $\|f\|_{X\cap Y}$ denotes $\|f\|_X+\|f\|_Y$)
\[
\|\mathcal Q\|_{L_y^2H_x^{s+1}\cap L_y^\infty H_x^{s+1/2}}
+\|\mathcal S\|_{L_y^2H_x^{s+1}\cap L_y^\infty H_x^{s+1/2}}
\lesssim_s \|\nabla h\|_{W^{2,\infty}}\||\nabla|^{1/2}\phi\|_{H^{s-1}}.
\]
\end{proposition}
\begin{proof}
Since $P_{\le-3}\mathcal R_1=iP_{\le-3}[T_{\nabla h\cdot\zeta},T_{|\zeta|}]$, $\mathcal R_1$ vanishes to degree 2 near $\zeta=0$,
so we can estimate in the frequency space, using Lemma \ref{K-bound} (ii), Lemma \ref{Eaf-Lp} and Lemma \ref{E1f-Lp} applied to (\ref{QS}) and (\ref{R1}), to get
\begin{align*}
\|\mathcal Q\|_{L_y^2H_x^{s+1}\cap L_y^\infty H_x^{s+1/2}}
&\lesssim_s \||\nabla|^{-1}\mathcal R_1[e^{y|\nabla|}\phi]\|_{L_y^2H_x^{s+1}}\\
&\lesssim_s \|\nabla h\|_{W^{2,\infty}}\|e^{y|\nabla|}|\nabla|\phi\|_{L_y^2H_x^{s-1}}\\
&\lesssim_s \|\nabla h\|_{W^{2,\infty}}\||\nabla|^{1/2}\phi\|_{H^{s-1}}.
\end{align*}
Since the multiplier of $\mathcal S$ is $(|\eta|-|\xi|)|\eta|$ and $(\xi-\eta,\eta)\mapsto(|\eta|-|\xi|)/|\xi-\eta|$ is a Mikhlin multipler,
Lemma \ref{PkH-Lp} (iii) still allows us to transfer 2 derivatives from $\phi$ to $h$ to get the bound for $\mathcal S$.
\end{proof}

For the evolution equation for $h$, we put
\begin{equation}\label{l-def}
\begin{aligned}
\lambda&=\lambda^{(1)}+\lambda^{(0)},\\
\lambda^{(1)}&=b^{(1)}\sqrt{1+\alpha}=|\zeta|(1+M_2(\zeta,\nabla h)),\\
\lambda^{(0)}&=b^{(0)}\sqrt{1+\alpha}+\frac12(\Delta h-\{\sqrt{1+\alpha},a^{(1)}\})\varphi_{\ge0}(\zeta)\\
&=M_1(\zeta,\nabla h,\nabla^2h)\varphi_{\ge0}(\zeta).
\end{aligned}
\end{equation}

The first equastion in (\ref{Zakharov}) can be paralinearized in the following way.
\begin{proposition}\label{ht-paralin}
If $s>4$ and $\|h\|_{H^s}<c_s$ is sufficiently small then
\begin{equation}\label{ht}
h_t=T_\lambda(w|_\Gamma)-\nabla\cdot T_Vh+\mathcal Q|_\Gamma+\mathcal S_h+\mathcal C_h,
\end{equation}
where $\mathcal Q$ is given by (\ref{QS}),
\begin{equation}\label{Sh}
\mathcal S_h=\mathcal S|_\Gamma-H(\nabla h,\nabla\phi)
\end{equation}
and for any integer $k<0$,
\[
\|\mathcal C_h\|_{H^{s+1/2}}
\lesssim_s \|\nabla h\|_{W^{3,\infty}}(C_{k,4}[|\nabla|^{1/2}\phi]\|h\|_{H^s}+\|\nabla h\|_{W^{3,\infty}}\||\nabla|^{1/2}\phi\|_{H^{s-1}}),
\]
where $C_{k,r}$ is defined in (\ref{Ckr-def}).
\end{proposition}

Now we paralinearize the second equation in (\ref{Zakharov}).
By (\ref{BV-u}) and (\ref{DN-u}),
\begin{equation}\label{phi-t}
\phi_t=-h-\frac12|V+B\nabla h|^2+\frac12(1+|\nabla h|^2)B^2
=-h+\frac12(B^2-2BV\cdot\nabla h-|V|^2).
\end{equation}
Recall the Alinhac good unknown $w|_\Gamma = \phi - T_Bh$. Then we have
\begin{proposition}\label{wt-paralin}
\begin{align}
\label{wt}
w_t|_\Gamma&=-T_ah-T_V\cdot\nabla(w|_\Gamma)+\mathcal S_w+\mathcal C_w,\\
\label{a-def}
a&=1+B_t+V\cdot\nabla B,\\
\label{Sw}
\mathcal S_w&=\frac12(H(B,B)-H(V,V)).
\end{align}
If $s>5/2$ and $\|\nabla h\|_{H^{s-1}\cap H^5}<c_s$ is sufficiently small, then
\begin{equation}\label{Cw-bound}
\|\mathcal C_w\|_{H^{s+1/2}}
\lesssim_s \||\nabla|^{1/2}\phi\|_{C_*^3}(\|\nabla h\|_{W^{2,\infty}}\||\nabla|^{1/2}\phi\|_{H^{s-1}}+\||\nabla|^{1/2}\phi\|_{C_*^4}\|h\|_{H^{s-3/2}}).
\end{equation}
\end{proposition}
\begin{remark}
Here $a$ is defined in a different way than in (\ref{AB-ab}). Hereafter $a$ will be used in the sense of (\ref{a-def}) unless noted otherwise.
\end{remark}

The proofs of both propositions can be found in Appendix \ref{ParLinEst}.

\subsection{Taylor expansion}
For future use we need to Taylor expand various quantities and estimate their remainders. We start with $a$, whose leading order is just 1.

To Taylor expand $a-1$ further, we use the fundamental theorem of calculus:
\begin{equation}\label{Gh-phi-int}
G(h)\phi=|\nabla|\phi+\int_0^1 \partial_sG(sh)\phi ds.
\end{equation}
By (2.6.8) of \cite{AlDe2},
\begin{equation}\label{ds-Gh-phi}
\partial_sG(sh)\phi=-G(sh)[hB(sh)\phi]-\nabla\cdot(hV(sh)\phi),
\end{equation}
where (see (2.0.16) of \cite{AlDe2})
\begin{align}\label{BVh-phi}
B(h)\phi&=\frac{G(h)\phi+\nabla h\cdot\nabla\phi}{1+|\nabla h|^2}, &
V(h)\phi&=\nabla\phi-(\nabla h)B(h)\phi.
\end{align}

\begin{proposition}\label{Taylor1}
(i) If $\|\nabla h\|_{H^{r+2}}<c_r$ is sufficiently small then
\begin{equation}\label{Gh-Cr2}
\|(G(h)\phi-|\nabla|\phi,B-|\nabla|\phi,V-\nabla\phi)\|_{C_*^r}
\lesssim_r \|h\|_{C_*^{r+1}}\||\nabla|^{1/2}\phi\|_{C_*^{r+3/2}}.
\end{equation}

(ii) If $\|\nabla h\|_{H^{r+2}}<c_r$ is sufficiently small then
\begin{align}
\label{a-1-Cr}
\|a-1\|_{C_*^r}&\lesssim_r \||\nabla|^{1/2}\phi\|_{C_*^{r+3/2}}^2+\|h\|_{C_*^{r+1}},\\
\label{a-1-Cr2}
\|(a-1+|\nabla|h,\sqrt a-1+|\nabla|h/2)\|_{C_*^r}&\lesssim_r \||\nabla|^{1/2}\phi\|_{C_*^{r+3/2}}^2+\|h\|_{C_*^{r+1}}\||\nabla|^{1/2}h\|_{C_*^{r+3/2}}.
\end{align}

(iii) If $\|\nabla h\|_{H^{r+3}}<c_r$ is sufficiently small then
\begin{equation}\label{at-Cr2}
\|(a_t-\Delta\phi,\partial_t\sqrt a-\Delta\phi/2)\|_{C_*^r}
\lesssim_r (\||\nabla|^{1/2}\phi\|_{C_*^{r+5/2}}^2+\|h\|_{C_*^{r+2}})\||\nabla|^{1/2}\phi\|_{C_*^{r+5/2}}.
\end{equation}
\end{proposition}
\begin{proof}
(i) This follows from (\ref{Gh-phi-int}), (\ref{ds-Gh-phi}) and (\ref{Gh-Cr}).

(ii) The first bound follows from the second one, which is shown in section \ref{Taylor}.

(iii) This is also shown in section \ref{Taylor}.
\end{proof}

We also need to Taylor expand $B$ and $G(h)\phi$ to higher orders.
By the first equation in the second display after (2.6.8) of \cite{AlDe2},
\begin{equation}\label{Bs}
\begin{aligned}
\partial_sB(sh)\phi&=\frac1{1+s^2|\nabla h|^2}(\partial_sG(sh)\phi+\nabla h\cdot\nabla\phi-2s|\nabla h|^2B(sh)\phi)
\end{aligned}
\end{equation}
with $\partial_sG(sh)\phi$ as in (\ref{ds-Gh-phi}).
By the fundamental theorem of calculus and (\ref{BVh-phi}),
\begin{equation}\label{B23}
\begin{aligned}
B&=|\nabla|\phi+\int_0^1 \partial_sB(sh)\phi ds=B_1+B_2+B_3,\quad
B_1=|\nabla|\phi,\\
B_2&=-|\nabla|(h|\nabla|\phi)-\nabla\cdot(h\nabla\phi)+\nabla h\cdot\nabla\phi=-|\nabla|(h|\nabla|\phi)-h\Delta\phi,\\
B_3&=-\int_0^1 (G(sh)-|\nabla|)[hB(sh)\phi]+|\nabla|(h(B(sh)-|\nabla|)\phi)ds\\
&+\int_0^1 \nabla\cdot(sh(\nabla h)B(sh)\phi)-2s|\nabla h|^2B(sh)\phi
-s^2|\nabla h|^2\partial_sB(sh)\phi ds.
\end{aligned}
\end{equation}
The term $B_3$ is at least cubic. Now by (\ref{BV-u}) and (\ref{DN-u}),
\begin{equation}\label{Gh-phi-int2}
\begin{aligned}
G(h)\phi&=(1+|\nabla h|^2)B-\nabla h\cdot\nabla\phi\\
&=|\nabla|\phi-|\nabla|(h|\nabla|\phi)-\nabla\cdot(h\nabla\phi)+B_3+|\nabla h|^2B.
\end{aligned}
\end{equation}

\begin{proposition}\label{Taylor2}
(i) If $r>1$ and $\|\nabla h\|_{H^{r+3/2}}<c_r$ is sufficiently small then
\begin{align}\label{Gh-Hs2}
\|(G(sh)-|\nabla|)\phi\|_{H^r}
&\lesssim_r \|h\|_{H^{r+1}}\||\nabla|^{1/2}\phi\|_{C_*^1}
+\|h\|_{L^\infty}\||\nabla|^{1/2}\phi\|_{H^{r+3/2}},\\
\label{BV-Hs2}
\|(B(sh)-|\nabla|)\phi\|_{H^r}
&\lesssim_r \|h\|_{H^{r+1}}\||\nabla|^{1/2}\phi\|_{C_*^1}+\|h\|_{W^{1,\infty}}\||\nabla|^{1/2}\phi\|_{H^{r+3/2}}.
\end{align}

(ii) If $r>1$ and $\|\nabla h\|_{H^{r+5/2}}<c_r$ is sufficiently small then
\begin{align}\label{B3-Hs}
\|B_3\|_{H^r}
&\lesssim \|h\|_{H^{r+2}}\|h\|_{C_*^{3/2}}\||\nabla|^{1/2}\phi\|_{C_*^2}
+\|h\|_{W^{1,\infty}}^2\||\nabla|^{1/2}\phi\|_{H^{r+5/2}}.
\end{align}
\end{proposition}
\begin{proof}
(i) By (\ref{Gh-Hs}), (\ref{Gh-phi-int}), (\ref{ds-Gh-phi}), the Sobolev multiplication theorem and (\ref{Gh-Cr}),
\begin{align*}
\|(G(sh)-|\nabla|)\phi\|_{H^r}
&\lesssim_r \|(hB(sh)\phi,hV(sh)\phi)\|_{H^{r+1}}\\
&\lesssim_r \|h\|_{H^{r+1}}\|(B(sh)\phi,V(sh)\phi)\|_{L^\infty}
+\|h\|_{L^\infty}\|(B(sh)\phi,V(sh)\phi)\|_{H^{r+1}}\\
&\lesssim_r \|h\|_{H^{r+1}}\||\nabla|^{1/2}\phi\|_{C_*^1}
+\|h\|_{L^\infty}\||\nabla|^{1/2}\phi\|_{H^{r+3/2}}.
\end{align*}
Combining this bound with (\ref{BVh-phi}) we get the second bound.

(ii) By (\ref{B23}), (\ref{Gh-Hs2}), (\ref{Gh-Cr}), Sobolev multiplication,
(\ref{Gh-Hs}), (\ref{BV-Hs2}) and (\ref{Gh-Cr2}),
\begin{align*}
\|B_3\|_{H^r}
&\lesssim_r \sup_{s\in[0,1]}\|h\|_{H^{r+1}}\||\nabla|^{1/2}(hB(sh)\phi)\|_{C_*^1}+\|h\|_{L^\infty}\||\nabla|^{1/2}(hB(sh)\phi)\|_{H^{r+3/2}}\\
&+\sup_{s\in[0,1]} \|(h(B(sh)-|\nabla|)\phi,h(\nabla h)B(sh)\phi)\|_{H^{r+1}}+\sup_{s\in[0,1]} \||\nabla h|^2B(sh)\phi\|_{H^r}\\
&\lesssim \|h\|_{H^{r+2}}\|h\|_{C_*^{3/2}}\||\nabla|^{1/2}\phi\|_{C_*^2}
+\|h\|_{W^{1,\infty}}^2\||\nabla|^{1/2}\phi\|_{H^{r+5/2}}.
\end{align*}
Note that the last term in the expression of $B_3$ in (\ref{B23}) equals
$-|\nabla h|^2B+\int_0^1 2s|\nabla h|^2B(sh)\phi ds$, which is dominated by the last term in the second line of (\ref{B3-Hs}).
\end{proof}

\section{Quartic energy estimates}\label{EneEst}
In this section we will obtain a quartic energy estimate of the form
\[
\mathcal E(t)=\mathcal E(0)+\int_0^t \mathcal E(s)\|U(s)\|_{C_*^6}^2ds.
\]

\subsection{Defining the quartic energy}\label{EneEst-def}
\begin{definition}
For an integer $N\ge5$ define
\begin{align*}
U&=h+i|\nabla|^{1/2}\phi, & \tilde U&=h+i|\nabla|^{1/2}(w|_\Gamma),\\
\mathcal U&=T_{\sqrt a}h+iT_{\sqrt\lambda}(w|_\Gamma), &
a&=1+B_t+V\cdot\nabla B,\\
\mathcal E&=\|P_{\ge0}\mathcal U\|_{H^N}^2, &
U_+&=U,\ U_-=\bar U,\ \mathcal U_+=\mathcal U,\ \mathcal U_-=\mathcal{\bar U}.
\end{align*}
\end{definition}

We first show that $\|U\|_{H^N}^2$, $\|\tilde U\|_{H^N}^2$ and $\mathcal E$ are close to each other.

\begin{proposition}\label{U-tld=U}
If $s>2$ and $\|h\|_{H^{s+1/2}}<c_s$ is sufficiently small then
\[
\|U\|_{H^s}\approx\|\tilde U\|_{H^s}.
\]
\end{proposition}
\begin{proof}
By (\ref{w-def}), Lemma \ref{Taf-Lp} (ii), Sobolev embedding and Proposition \ref{u-Hs} (ii),
\begin{equation}\label{U-tld-U-HN}
\begin{aligned}
\|U-\tilde U\|_{H^s}&=\||\nabla|^{1/2}T_Bh\|_{H^s}
\lesssim_s \|B\|_{L^\infty}\|h\|_{H^{s+1/2}}
\lesssim_s c_s\|B\|_{H^{s-1}}\\
&\lesssim_s c_s\||\nabla|^{1/2}\phi\|_{H^{s-1/2}}
\lesssim_s c_s\|U\|_{H^s}.
\end{aligned}
\end{equation}
When $c_s$ is sufficiently small, the result follows.
\end{proof}

\begin{proposition}\label{E-HN}
If $\|U\|_{H^5}$ is sufficiently small, then
\[
\|\mathcal U-\tilde U\|_{H^N}\lesssim \|\tilde U\|_{H^N}^2.
\]
\end{proposition}
\begin{proof}
By Lemma \ref{Taf-Lp} (ii),
\begin{equation}\label{E-U-dif}
\|\mathcal U-\tilde U\|_{H^N}
\lesssim \|\sqrt a-1\|_{L^\infty}\|h\|_{H^N}+\|\sqrt{\lambda/|\zeta|}-1\|_{\mathcal L^\infty_0}\||\nabla^{1/2}(w|_\Gamma)\|_{H^N}.
\end{equation}
By (\ref{a-1-Cr}) and the smallness of $\|\nabla h\|_{H^{7/2}}+\||\nabla|^{1/2}\phi\|_{C_3^*}\lesssim \|U\|_{H^{9/2}}$,
\begin{equation}\label{a-1-Loo}
\|a-1\|_{W^{1,\infty}}\lesssim \|a-1\|_{C_*^{3/2}}
\lesssim \|h\|_{C_*^{5/2}}+\||\nabla|^{1/2}\phi\|_{C_*^3}^2
\lesssim \|U\|_{C_*^3}.
\end{equation}
By (\ref{l-def}), $\lambda$ vanishes to degree 1 near $\zeta = 0$, so
$\|\sqrt{\lambda/|\zeta|}-1\|_{\mathcal L^\infty_0} \lesssim \|\sqrt\lambda-|\zeta|^{1/2}\|_{\mathcal L^\infty_{1/2}}$, and since
$\lambda = |\zeta|(1+M_2(\zeta,\nabla h)) + M_1(\zeta,\nabla h,\nabla^2h)$,
\begin{equation}\label{l-|z|-Loo}
\|\sqrt\lambda-|\zeta|^{1/2}\|_{\mathcal L^\infty_{1/2}}
+\|\nabla_x\sqrt\lambda\|_{\mathcal L^\infty_{1/2}}
\lesssim \|h\|_{W^{3,\infty}}\lesssim \|U\|_{C_*^{7/2}}.
\end{equation}
Putting (\ref{a-1-Loo}) and (\ref{l-|z|-Loo}) into (\ref{E-U-dif}) we get
\begin{equation}\label{cal-U-tld-U-HN}
\|\mathcal U-\tilde U\|_{H^N}\lesssim \|U\|_{C_*^{7/2}}\|\tilde U\|_{H^N}.
\end{equation}
By Proposition \ref{U-tld=U}, $\|U\|_{C_*^{7/2}} \lesssim \|U\|_{H^{9/2}} \lesssim \|\tilde U\|_{H^5}$, where we have used $N \ge 5$.
\end{proof}

The rest of this section is devoted to estimating $d\mathcal E/dt$.
By (\ref{ht}) and (\ref{wt}), the evolution equation for $\mathcal U$ is
\begin{align*}
\partial_t\mathcal U&=T_{\sqrt a}T_\lambda(w|_\Gamma)
-T_{\sqrt a}\nabla\cdot T_Vh
+T_{\sqrt a}(\mathcal Q|_\Gamma+\mathcal S_h+\mathcal C_h)\\
&-iT_{\sqrt\lambda}T_ah-iT_{\sqrt\lambda}T_V\cdot\nabla(w|_\Gamma)
+iT_{\sqrt\lambda}(\mathcal S_w+\mathcal C_w)+T_{\partial_t\sqrt a}h
+iT_{\partial_t\sqrt\lambda}(w|_\Gamma).
\end{align*}
Hence
\begin{align}
\nonumber
(\partial_t+iT_{\sqrt{a\lambda}})\mathcal U
&=(-T_{\sqrt a}E(\sqrt\lambda,\sqrt\lambda)+E(\sqrt a,\sqrt\lambda)T_{\sqrt\lambda})(w|_\Gamma)\\
\nonumber
&+i(T_{\sqrt\lambda}E(\sqrt a,\sqrt a)-E(\sqrt\lambda,\sqrt a)T_{\sqrt a})h\\
\nonumber
&-T_{\sqrt a}\nabla\cdot T_Vh-iT_{\sqrt\lambda}T_V\cdot\nabla(w|_\Gamma)
+T_{\sqrt a}(\mathcal Q|_\Gamma+\mathcal S_h+\mathcal C_h)\\
&+iT_{\sqrt\lambda}(\mathcal S_w+\mathcal C_w)+T_{\partial_t\sqrt a}h+iT_{\partial_t\sqrt\lambda}(w|_\Gamma).
\label{cal-Ut0}
\end{align}
From (\ref{ab-soln}) and (\ref{l-def}) it follows that
\[
\lambda=|\zeta|(1+M_2(\zeta,\nabla h,\nabla^2h))+b_1^{(0)}+\left( \frac12\Delta h+M_3(\zeta,\nabla h,\nabla^2h) \right)\varphi_{\ge0}(\zeta).
\]
Hence
\begin{equation}\label{l-l1}
\sqrt\lambda=[\sqrt\lambda]_{\le1}+|\zeta|^{1/2}M_2(\zeta,\nabla h,\nabla^2h)+|\zeta|^{-1/2}M_3(\zeta,\nabla h,\nabla^2h)\varphi_{\ge0}(\zeta),
\end{equation}
where
\begin{align*}
[\sqrt\lambda]_{\le1}&=\sqrt{|\zeta|}+[\sqrt\lambda]_1, &
[\sqrt\lambda]_1&=\left( \frac{b_1^{(0)}}{2\sqrt{|\zeta|}}+\frac{\Delta h}{4\sqrt{|\zeta|}} \right)\varphi_{\ge0}(\zeta)
\end{align*}
is the first order approximation of $\sqrt\lambda$, i.e.,
if $\|h\|_{C_*^4}$ is sufficiently small then
\begin{equation}\label{l-l1-Loo}
\|\sqrt\lambda-[\sqrt\lambda]_{\le1}\|_{\mathcal L^\infty_{1/2}}+
\|\nabla_x(\sqrt\lambda-[\sqrt\lambda]_{\le1})\|_{\mathcal L^\infty_{1/2}}
\lesssim \|h\|_{W^{3,\infty}}^2\lesssim \|h\|_{C_*^4}^2.
\end{equation}
The first order approximation of $\sqrt a$, $\partial_t\sqrt a$ and $\partial_t\sqrt\lambda$ are $1-|\nabla|h/2$, $\Delta\phi/2$ and
\[
[\partial_t\sqrt\lambda]_{\le1}=\partial_t[\sqrt\lambda]_1
=\left( \frac{\zeta^T(\nabla^2|\nabla|\phi)\zeta}{2|\zeta|^{5/2}}
+\frac{\Delta|\nabla|\phi}{4\sqrt{|\zeta|}} \right)\varphi_{\ge0}(\zeta)
\]
respectively: by (\ref{a-1-Cr2}), (\ref{at-Cr2}), (\ref{ht=Gh-phi}), (\ref{Gh-Cr}) and (\ref{Gh-Cr2}), if $\|U\|_{H^6}$ is sufficiently small,
\begin{equation}\label{Loo2}
\|\sqrt a-1+|\nabla|h/2\|_{W^{1,\infty}}+\|\partial_t\sqrt a-\Delta\phi/2\|_{L^\infty}+\|\partial_t\sqrt\lambda-[\partial_t\sqrt\lambda]_{\le1}\|_{\mathcal L^\infty_{1/2}}\lesssim \|U\|_{C_*^4}^2.
\end{equation}
Now by (\ref{cal-Ut0}),
\begin{equation}\label{cal-Ut}
(\partial_t+iT_{\sqrt{a\lambda}}+iT_{V\cdot\zeta})\mathcal U
=\mathcal Q|_\Gamma+\mathcal{\tilde Q+\tilde S+C},
\end{equation}
where $\mathcal Q$ is given by (\ref{QS}),
\begin{align*}
\mathcal{\tilde Q}&=-(E(\sqrt{|\zeta|},[\sqrt\lambda]_1)
+E([\sqrt\lambda]_1,\sqrt{|\zeta|})
+E(|\nabla h|/2,\sqrt{|\zeta|})|\nabla|^{1/2})(w|_\Gamma)\\
&+iE(\sqrt{|\zeta|},|\nabla h|/2)h-iE(\zeta,\nabla\phi)h+E(\sqrt{|\zeta|},\nabla\phi,\zeta/\sqrt{|\zeta|})|\nabla|^{1/2}(w|_\Gamma)\\
\displaybreak[0]
&+T_{\Delta\phi/2}h+iT_{[\partial_t\sqrt\lambda]_{\le1}}(w|_\Gamma),\\
\displaybreak[0]
\mathcal{\tilde S}&=\mathcal S_h+i|\nabla|^{1/2}\mathcal S_w
=(\ref{Sh})+i|\nabla|^{1/2}(\ref{Sw}),\\
\mathcal C&=\mathcal{\tilde C}_1+\mathcal{\tilde C}_2
+T_{\sqrt a-1}(\mathcal Q|_\Gamma+\mathcal S_h)+iT_{\sqrt\lambda-|\zeta|^{1/2}}\mathcal S_w+T_{\sqrt a}\mathcal C_h+iT_{\sqrt\lambda}\mathcal C_w\\
\displaybreak[0]
&+T_{\partial_t\sqrt a-\Delta\phi/2}h+iT_{\partial_t\sqrt\lambda-[\partial_t\sqrt\lambda]_{\le1}}(w|_\Gamma),\\
\mathcal{\tilde C}_1&=[-T_{\sqrt a-1}E(\sqrt\lambda,\sqrt\lambda)
-E(\sqrt{|\zeta|},\sqrt\lambda-[\sqrt\lambda]_{\le1})
-E(\sqrt\lambda-[\sqrt\lambda]_{\le1},\sqrt{|\zeta|})\\
&-E(\sqrt\lambda-\sqrt{|\zeta|},\sqrt\lambda-\sqrt{|\zeta|})
+E(\sqrt a-1,\sqrt\lambda)T_{\sqrt\lambda-|\zeta|^{1/2}}\\
&+E(\sqrt a-1,\sqrt\lambda-\sqrt{|\zeta|})|\nabla|^{1/2}
+E(\sqrt a-1+|\nabla h|/2,\sqrt{|\zeta|})|\nabla|^{1/2}](w|_\Gamma)\\
&+i[T_{\sqrt\lambda}E(\sqrt a-1,\sqrt a-1)-E(\sqrt\lambda,\sqrt a-1)T_{\sqrt a-1}\\
\displaybreak[0]
&-E(\sqrt\lambda-\sqrt{|\zeta|},\sqrt a-1)-E(\sqrt{|\zeta|},\sqrt a-1+|\nabla h|/2)]h,\\
\mathcal{\tilde C}_2&=-iT_{\sqrt a-1}E(\zeta,V)h-iE(\zeta,V-\nabla\phi)h
+i[T_{V\cdot\zeta},T_{\sqrt a-1}]h\\
&+T_{\sqrt\lambda-|\zeta|^{1/2}}E(V,\zeta)(w|_\Gamma)+[T_{\sqrt\lambda-|\zeta|^{1/2}},T_{V\cdot\zeta}](w|_\Gamma).
\end{align*}

By Lemma \ref{para2diff} (i), $T_{\sqrt{a\lambda}+V\cdot\zeta}$ is self-adjoint,
so $\langle T_{\sqrt{a\lambda}+V\cdot\zeta}f,f \rangle\in\R$.
Now we can decompose $d\mathcal E/dt$ accordingly:
\[
\frac{d}{dt}\mathcal E
=2\Re\langle (\partial_t+iT_{\sqrt{a\lambda}}+iT_{V\cdot\zeta})P_{\ge0}\langle\nabla\rangle^N\mathcal U,
P_{\ge0}\langle\nabla\rangle^N\mathcal U \rangle
=2(\mathcal E_Q+\mathcal E_{\tilde Q}+\mathcal E_S+\mathcal E_4),
\]
where (note that $[|\nabla|^{1/2},P_{\ge0}\langle\nabla\rangle^N]=0$)
\begin{align*}
\mathcal E_Q&=\Re\langle P_{\ge0}\langle\nabla\rangle^N\mathcal Q|_\Gamma,P_{\ge0}\langle\nabla\rangle^NU \rangle,\\
\mathcal E_{\tilde Q}&=\Re\langle[iT_{-|\zeta|^{1/2}|\nabla|h/2+[\sqrt\lambda]_1+\zeta\cdot\nabla\phi},P_{\ge0}\langle\nabla\rangle^N]\mathcal U
+P_{\ge0}\langle\nabla\rangle^N\mathcal{\tilde Q},P_{\ge0}\langle\nabla\rangle^N\mathcal U\rangle,\\
\mathcal E_S&=\Re\langle P_{\ge0}\langle\nabla\rangle^N\mathcal{\tilde S},P_{\ge0}\langle\nabla\rangle^NU \rangle,\\
\mathcal E_4&=\Re\langle P_{\ge0}\langle\nabla\rangle^N(\mathcal Q|_\Gamma+\mathcal{\tilde S}),P_{\ge0}\langle\nabla\rangle^N(\mathcal U-U) \rangle
+\Re\langle P_{\ge0}\langle\nabla\rangle^N\mathcal C,P_{\ge0}\langle\nabla\rangle^N\mathcal U\rangle\\
&+\Re\langle[iT_{(\sqrt a-1)(\sqrt\lambda-|\zeta|^{1/2})+(\sqrt a-1+|\nabla|h/2)|\zeta|^{1/2}+\sqrt\lambda-[\sqrt\lambda]_{\le1}},P_{\ge0}\langle\nabla\rangle^N]\mathcal U,P_{\ge0}\langle\nabla\rangle^N\mathcal U\rangle.
\end{align*}

\subsection{Bounding the quartic energy}
\begin{proposition}\label{E4}
If $\|h\|_{H^N}+\|U\|_{H^6}$ is sufficiently small then
\begin{align*}
|\mathcal E_4|&\lesssim \inf_{k<0}C_{k,4}[U]\|U\|_{C_*^5}\|\tilde U\|_{H^N}^2,
\end{align*}
where $C_{k,4}[\cdot]$ is defined in (\ref{Ckr-def}).
\end{proposition}
\begin{proof}
By Proposition \ref{U-tld=U}, $\|U\|_{H^{N-1/2}}$ and $\|\tilde U\|_{H^{N-1/2}}$ are interchangeable.

For the first term in $\mathcal E_4$, by (\ref{U-tld-U-HN}) (with $s=N-1/2$), (\ref{cal-U-tld-U-HN}) and (\ref{Gh-Cr}) it suffices to show
\begin{equation}\label{QS-HN}
\|P_{\ge0}(\mathcal Q|_\Gamma+\mathcal{\tilde S})\|_{H^{N+1/2}}\lesssim \|U\|_{C_*^4}\|U\|_{H^{N-1}}.
\end{equation}
By Proposition \ref{QS-bound} with $s=N$,
this bound holds for $(\mathcal{Q+S})|_\Gamma$.
By Lemma \ref{PkH-Lp} (iii) with $m=2$, it also holds for $H(\nabla h,\nabla\phi)$, and hence for $\mathcal S_h$ by (\ref{Sh}).
By Lemma \ref{PkH-Lp} (iii), (\ref{Sw}), (\ref{Gh-Cr}) and (\ref{Gh-Hs}),
\begin{equation}\label{Sw-HN}
\begin{aligned}
\|\mathcal S_w\|_{H^{N+1}}
&\lesssim \|(B,V)\|_{W^{3,\infty}}\|(B,V)\|_{H^{N-2}}\\
&\lesssim \||\nabla|^{1/2}\phi\|_{C_*^4}\||\nabla|^{1/2}\phi\|_{H^{N-3/2}}
\lesssim \|U\|_{C_*^4}\|U\|_{H^{N-1}}
\end{aligned}
\end{equation}
and the desired bound for $\mathcal S_w$ follows.

For the second term in $\mathcal E_4$, by (\ref{cal-U-tld-U-HN}) it suffices to show
\begin{equation}\label{C-HN}
\|P_{\ge0}\mathcal C\|_{H^N}\lesssim \inf_{k<0}C_{k,4}[U]\|U\|_{C_*^5}\|\tilde U\|_{H^N}.
\end{equation}
The desired bound for $\mathcal{\tilde C}_1$ follows from Lemma \ref{Taf-Lp} (ii), Lemma \ref{Eaf-Lp} (ii), (\ref{a-1-Loo}), (\ref{l-|z|-Loo}), (\ref{l-l1-Loo}) and (\ref{Loo2}).
To get the desired bound for $\mathcal{\tilde C}_2$, we also need (\ref{Gh-Cr}) and (\ref{Gh-Cr2}) to control the $W^{1,\infty}$ norms of $V$ and $V-\nabla\phi$.

The desired bound for $T_{\sqrt a-1}(\mathcal{Q+S})|_\Gamma$ follows from Lemma \ref{Taf-Lp} (ii), (\ref{a-1-Loo}) and the real part of (\ref{QS-HN}).
The desired bound for $T_{\sqrt\lambda-|\zeta|^{1/2}}\mathcal S_w$ follows from Lemma \ref{Taf-Lp} (ii), (\ref{l-|z|-Loo}) and (\ref{Sw-HN}).
The desired bound for $T_{\sqrt a}\mathcal C_h$ follows from Lemma \ref{Taf-Lp} (ii), (\ref{a-1-Loo}) and Proposition \ref{ht-paralin}.
The desired bound for $T_{\sqrt\lambda}\mathcal C_w$ follows from Lemma \ref{Taf-Lp} (ii), (\ref{l-|z|-Loo}) and (\ref{Cw-bound}).
The desired bounds for $T_{\partial_t\sqrt a-\Delta\phi/2}h$ and
$T_{\partial_t\sqrt\lambda-[\partial_t\sqrt\lambda]_{\le1}}(w|_\Gamma)$ follow from Lemma \ref{Taf-Lp} (ii) and (\ref{Loo2}).
Now all the terms in $\mathcal C$ have been controlled as desired.

Finally, the bound for the third term in $\mathcal E_4$ follows from Lemma \ref{Eaf-Lp} (ii), (\ref{a-1-Loo}), (\ref{l-|z|-Loo}), (\ref{l-l1-Loo}) and (\ref{Loo2}).
\end{proof}

\subsection{Bounding the semilinear energy}
\begin{proposition}\label{Es}
If $\|h\|_{H^N}$ is suffciently small then
\begin{align*}
\left| \int_0^t \mathcal E_S(s)ds \right|
&\lesssim \|\tilde U\|_{L^\infty([0,t])H^N}^3
+\|U\|_{L^2([0,t])C_*^4}\|\inf_{k<0}C_{k,4}[U]\|_{L^2([0,t])}\|\tilde U\|_{L^\infty([0,t])H^N}^2,
\end{align*}
where $C_{k,4}[\cdot]$ is defined in (\ref{Ckr-def}).
\end{proposition}
\begin{proof}
Again, $\|U\|_{H^{N-1/2}}$ and $\|\tilde U\|_{H^{N-1/2}}$ are interchangeable.

By (\ref{Gh-Cr}), $\|(B,V)\|_{W^{3,\infty}}\lesssim\|(B,V)\|_{C_*^{7/2}}\lesssim\|U\|_{C_*^4}$. By Proposition \ref{u-Hs} (ii), Proposition \ref{u-Hs2} (ii) and (\ref{BV-u}), for any integer $k<0$,
\begin{align}
\nonumber
\|(B-|\nabla|\phi,V-\phi)\|_{H^{N-3/2}}
&\lesssim C_{k,2}[U](\||\nabla|^{1/2}\phi\|_{H^{N-1}}+\|\nabla h\|_{H^{N-1}})\\
&\lesssim C_{k,2}[U]\|\tilde U\|_{H^N}
\label{BV-Hs2-sharp}
\end{align}
so by Lemma \ref{PkH-Lp} (iii) (with $m=3$),
\[
\||\nabla|^{1/2}(\mathcal S_w-\mathcal S_w[|\nabla|\phi,\nabla\phi])\|_{H^{N+1}}
\lesssim \|U\|_{C_*^4}\inf_{k<0}C_{k,2}[U]\|\tilde U\|_{H^N}.
\]
Since we also have $\|U\|_{H^{N-1}}\lesssim\|\tilde U\|_{H^N}$,
we can replace $\mathcal S_w$ by $\mathcal S_w[|\nabla|\phi,\nabla\phi]$ in the expression of $\mathcal{\tilde S}$ with acceptable error.

Recall $U_+=U$ and $U_-=\bar U$. From the expression of $\mathcal S$
(see (\ref{S0}), (\ref{QS}), (\ref{Sh}) and (\ref{Sw})),
it follows that $\mathcal E_S$ is a linear combination of terms of the form $\Re \mathcal E_S^{\mu\nu}$, where
\begin{align*}
\mathcal E_S^{\mu\nu}
&=C^2\iint s(\xi_1,\xi_2)\hat U_\mu(\xi_1)\hat U_\nu(\xi_2)\langle\xi_1+\xi_2\rangle^{2N}\overline{\hat U(\xi_1+\xi_2)}d\xi_1 d\xi_2, & \mu,\nu&\in\{+,-\},
\end{align*}
\begin{align}
\label{s-mult}
s(\xi_1,\xi_2)&\in\left( 1-\varphi_{\le-10}\left( \frac{|\xi_1|}{|\xi_1+2\xi_2|} \right)-\varphi_{\le-10}\left( \frac{|\xi_2|}{|\xi_2+2\xi_1|} \right) \right)n_1(\xi_1)n_2(\xi_2)n_3(\xi_1+\xi_2)\\
\nonumber
&\times\{1,|\xi_1|p(\xi_1,\xi_2),|\xi_2|p(\xi_1,\xi_2)\},\\
\label{p-mult}
p(\xi_1,\xi_2)&=(|\xi_2|+|\xi_1+\xi_2|)^{-1},
\end{align}
$n_j\in S^{m_j}_{1,0}$ ($1\le j\le 3$), $\sum_{j=1}^3 m_j=3/2$ and
$\supp n_3\subset\supp\varphi_{\ge0}$. Let
\begin{align*}
\Phi_{\mu\nu}(\xi_1,\xi_2)
&=\sqrt{|\xi_1+\xi_2|}-\mu\sqrt{|\xi_1|}-\nu\sqrt{|\xi_2|},\\
I_S^{\mu\nu}[f_1,f_2,f_3]
&=C^2\iint \frac{s(\xi_1,\xi_2)}{\Phi_{\mu\nu}(\xi_1,\xi_2)}\hat f_1(\xi_1)\hat f_2(\xi_2)\langle\xi_1+\xi_2\rangle^{2N}\overline{\hat f_3(\xi_1+\xi_2)}d\xi_1 d\xi_2,\\
I_S^{\mu\nu}&=I_S^{\mu\nu}[U_\mu,U_\nu,U].
\end{align*}

By (\ref{ht=Gh-phi}) and (\ref{phi-t}), the evolution equation for $U$ is
\begin{equation}\label{Ut}
U_t=-i|\nabla|^{1/2}U+N,\ 
N=(G(h)-|\nabla|)\phi+\frac i2|\nabla|^{1/2}((1+|\nabla h|^2)B^2-|\nabla\phi|^2).
\end{equation}
Let $N_+=N$ and $N_-=\bar N$. Then
\begin{align*}
\frac{dI_S^{\mu\nu}}{dt}&=I_S^{\mu\nu}[(U_\mu)_t,U_\nu,U]+I_S^{\mu\nu}[U_\mu,(U_\nu)_t,U]+I_S^{\mu\nu}[U_\mu,U_\nu,U_t]\\
&=\mathcal E_S^{\mu\nu}+I_S^{\mu\nu}[N_\mu,U_\nu,U]+I_S^{\mu\nu}[U_\mu,N_\nu,U]+I_S^{\mu\nu}[U_\mu,U_\nu,N].
\end{align*}
Integration by parts in time gives
\begin{align}
\label{Es-IBP1}
\int_0^t \mathcal E_S^{\mu\nu}(s)ds&=I_S^{\mu\nu}(t)-I_S^{\mu\nu}(0)\\
\label{Es-IBP2}
&-\int_0^t (I_S^{\mu\nu}[N_\mu,U_\nu,U]+I_S^{\mu\nu}[U_\mu,N_\nu,U])(s)ds\\
\label{Es-IBP3}
&-\int_0^t I_S^{\mu\nu}[U_\mu,U_\nu,N](s)ds.
\end{align}

To bound $I_S^{\mu\nu}$, we need to control the norms of the symbols $p$, $s$ and $\Phi_{\mu\nu}^{-1}$.
\begin{lemma}\label{ps-1/F-Soo}
(i) For $k_1$, $k_2$, $k_3\in\Z$ we have
\begin{align*}
\|p\|_{S^\infty_{k_1,k_2;k_3}}&\lesssim 2^{-\max(k_1,k_2)}, &
\|s\|_{S^\infty_{k_1,k_2;k_3}}&\lesssim 2^{3\max(k_1,k_2)/2}.
\end{align*}

(ii) For $L\ge0$ we have
\[
|\nabla^L\Phi_{\mu\nu}^{-1}|\lesssim_L \min(|\xi_1|,|\xi_2|,|\xi_1+\xi_2|)^{-L-1/2}.
\]

(iii) For $k_1$, $k_2$, $k_3\in\Z$ we have
\[
\|\Phi_{\mu\nu}^{-1}\|_{S^\infty_{k_1,k_2;k_3}}\lesssim 2^{-\min(k_1,k_2,k_3)/2}.
\]

(iii') If in addition to (iii) we have $k_1\le k_2-3$ and $\nu=-$ then
\[
\|\Phi_{\mu\nu}^{-1}\|_{S^\infty_{k_1,k_2;k_3}}\lesssim 2^{-k_2/2}.
\]
\end{lemma}
\begin{proof}
See section \ref{MulEst}.
\end{proof}

Now we bound (\ref{Es-IBP1}), (\ref{Es-IBP2}) and (\ref{Es-IBP3}).
By Lemma \ref{paraprod}, Lemma \ref{ps-1/F-Soo} (i) and (iii),
\[
|I_S^{\mu\nu}[P_{k_1}f_1,P_{k_2}f_2,P_{k_3}f_3]|
\lesssim 2^{2Nk_3^++3k_2/2-k_3/2}\|P_{k_1}f_1\|_{L^\infty}\|P_{k_2}f_2\|_{L^2}\|P_{k_3}f_3\|_{L^2}.
\]
Thanks to the factors $n_3$ and $\varphi_{\le-10}$,
this term vanishes unless $k_3\ge-1$ and $k_1, k_2\ge k_3-20$,
in which case, using Bernstein's inequality, it can be bounded by
\begin{align*}
&2^{-(N-1/2)|k_2-k_3|}\|P_{k_1}f_1\|_{C_*^{2+m}}\|P_{k_2}f_2\|_{H^{N-1/2}}\|P_{k_3}f_3\|_{H^{N-1/2-m}}\\
\lesssim &2^{-(N-1/2)|k_2-k_3|}\|f_1\|_{C_*^{2+m}}\|P_{k_2}f_2\|_{H^{N-1/2}}\|P_{k_3}f_3\|_{H^{N-1/2-m}}.
\end{align*}
A similar bound with $f_1$ and $f_2$ swapped holds.
The additive restriction of frequencies implies $|k_1-k_2|=O(1)$,
so summing over $k_1$, $k_2$ and $k_3$ using the Cauchy--Schwarz inequality gives
\begin{equation}\label{Is}
\begin{aligned}
|I_S^{\mu\nu}[f_1,f_2,f_3]|
&\lesssim \|f_1\|_{C_*^{2+m}} \sum_{k,l\in\Z} 2^{-(N-1/2)|l|}\|P_{k+l}f_2\|_{H^{N-1/2}}\|P_kf_3\|_{H^{N-1/2-m}}\\
&\lesssim \|f_1\|_{C_*^{2+m}} \|f_2\|_{H^{N-1/2}}\|f_3\|_{H^{N-1/2-m}}.
\end{aligned}
\end{equation}

Recall that $\|U\|_{H^{N-1/2}}\approx\|\tilde U\|_{H^{N-1/2}}$.
By (\ref{Is}) with $f_1=U_\mu$, $f_2=U_\nu$, $f_3=U$ and $m=0$,
\[
|(\ref{Es-IBP1})|\lesssim \|U(t)\|_{H^{N-1/2}}^3+\|U(0)\|_{H^{N-1/2}}^3
\lesssim \|\tilde U\|_{L^\infty([0,t])H^N}^3.
\]
By (\ref{Ut}), (\ref{Gh-Cr2}) and (\ref{Gh-Cr}), $\|N\|_{C_*^2}\lesssim \|U\|_{C_*^4}^2$, so by (\ref{Is}) with $f_1=N_\mu$, $f_2=U_\nu$, $f_3=U$ and $m=0$ (and its symmetric version),
\[
|(\ref{Es-IBP2})|\lesssim \|U\|_{L^2([0,t])C_*^4}^2\|\tilde U\|_{L^\infty([0,t])H^N}^2.
\]
By Sobolev multiplication (applied to (\ref{Ut}) and (\ref{Gh-phi})), (\ref{BV-Hs2-sharp}), (\ref{Gh-Hs}) and (\ref{Gh-Cr}),
\begin{equation}\label{N-HN}
\|N\|_{H^{N-2}}\lesssim \inf_{k<0}C_{k,2}[U]\|\tilde U\|_{H^N}.
\end{equation}
Then by (\ref{Is}) with $f_1=U_\mu$, $f_2=U_\nu$, $f_3=N$ and $m=3/2$,
\[
|(\ref{Es-IBP3})|\lesssim \|U\|_{L^2([0,t])C_*^4}\|\inf_{k<0}C_{k,2}[U]\|_{L^2([0,t])}\|\tilde U\|_{H^N}^2.
\]

Combining the three bounds shows the claim.
\end{proof}

\subsection{Bounding the quasilinear energy $\mathcal E_Q$}
\begin{proposition}\label{Eq}
If $N\ge6$ and $\|h\|_{H^N\cap H^7}$ is suffciently small then
\[
\left| \int_0^t \mathcal E_Q(s)ds \right|
\lesssim \|\tilde U\|_{L^\infty([0,t])H^N}^3
+\|\inf_{k<0}C_{k,6}[U]\|_{L^2([0,t])}^2\|\tilde U\|_{L^\infty([0,t])H^N}^2,\\
\]
where $C_{k,6}[\cdot]$ is defined in (\ref{Ckr-def}).
\end{proposition}
\begin{proof}
$\mathcal E_Q$ is a linear combination of terms of the form $\Re\mathcal E_Q^{\mu\nu}$, where
\begin{align*}
\mathcal E_Q^{\mu\nu}
&=C^2\iint (pq)(\xi_1,\xi_2)\hat U_\mu(\xi_1)\hat U_\nu(\xi_2)\varphi_{\ge0}(\xi_1+\xi_2)\langle\xi_1+\xi_2\rangle^{2N-1/2}\overline{\hat U(\xi_1+\xi_2)}d\xi_1 d\xi_2,
\end{align*}
$\mu,\nu\in\{+,-\}$ and
\begin{align}
\nonumber
p(\xi_1,\xi_2)&=(|\xi_2|+|\xi_1+\xi_2|)^{-1},\\
\nonumber
q(\xi_1,\xi_2)&=\varphi_{\le-10} \left( \frac{|\xi_1|}{|\xi_1+2\xi_2|} \right)\frac{(|\xi_1+\xi_2|-|\xi_2|)^2(|\xi_1+\xi_2|+|\xi_2|)}2\\
&\times\left( \frac{2((\xi_1+\xi_2)\cdot\xi_2-|\xi_1+\xi_2||\xi_2|)}{|\xi_1+2\xi_2|^2}\varphi_{\ge0}\left( \xi_1+\frac{\xi_2}2 \right)+\varphi_{\le-1}\left( \xi_1+\frac{\xi_2}2 \right) \right),
\label{q-mult}
\end{align}
where the expression of $q$ comes from (B.42) of \cite{DeIoPaPu}. Let
\begin{align*}
I_Q^{\mu\nu}[f_1,f_2,f_3]&=C^2\iint \frac{(pq)(\xi_1,\xi_2)}{\Phi_{\mu\nu}(\xi_1,\xi_2)}\hat f_1(\xi_1)\hat f_2(\xi_2)\varphi_{\ge0}(\xi_1+\xi_2)\langle\xi_1+\xi_2\rangle^{2N-1/2}\overline{\hat f_3(\xi_1+\xi_2)}d\xi_1 d\xi_2,\\
I_Q^{\mu\nu}&=I_Q^{\mu\nu}[U_\mu,U_\nu,U].
\end{align*}
Similarly integration by parts in time gives
\begin{align}
\label{EQ-IBP1}
\int_0^t \mathcal E_Q^{\mu\nu}(s)ds&=I_Q^{\mu\nu}(t)-I_Q^{\mu\nu}(0)\\
\label{EQ-IBP2}
&-\int_0^t I_Q^{\mu\nu}[N_\mu,U_\nu,U](s)ds\\
\label{EQ-IBP3}
&-\int_0^t (I_Q^{\mu\nu}[U_\mu,N_\nu,U]+I_Q^{\mu\nu}[U_\mu,U_\nu,N])(s)ds.
\end{align}

To bound $I_S^{\mu\nu}$, we need a bound of the $S^\infty$ norm of $q$.
\begin{lemma}\label{q-Soo}
For $k_1$, $k_2$, $k_3\in\Z$ we have
\[
\|q\|_{S^\infty_{k_1,k_2;k_3}}\lesssim 2^{2k_1+k_2}(2^{2(k_1-k_2)}+1_{k_2\le2})1_{k_1\le k_2-6}
\lesssim 2^{2(k_1+k_1^+)-|k_2|}1_{k_1\le k_2-6}.
\]
\end{lemma}
\begin{proof}
See section \ref{MulEst}.
\end{proof}

Now we bound (\ref{EQ-IBP1}), (\ref{EQ-IBP2}) and (\ref{EQ-IBP3}).
By Lemma \ref{paraprod}, Lemma \ref{ps-1/F-Soo} (i) and (iii), and Lemma \ref{q-Soo}, when $k_2\ge-2$ and $k_1\le k_2-6$ we have
\[
|I_Q^{\mu\nu}[P_{k_1}f_1,P_{k_2}f_2,P_{k_3}f_3]|\lesssim 2^{(2N-1/2)k_3^++3k_1/2+2k_1^+-2k_2}\|P_{k_1}f_1\|_{L^\infty}\|P_{k_2}f_2\|_{L^2}\|P_{k_3}f_3\|_{L^2}.
\]
By the additive restriction of frequencies, $k_3=k_2+O(1)\ge-O(1)$, so
\[
|I_Q^{\mu\nu}[P_{k_1}f_1,P_{k_2}f_2,P_{k_3}f_3]|\lesssim 2^{(2N-5/2)k_3+3k_1/2+2k_1^+}\|P_{k_1}f_1\|_{L^\infty}\|P_{k_2}f_2\|_{L^2}\|P_{k_3}f_3\|_{L^2}.
\]
Summing over $k_3=k_2+O(1)\ge-O(1)$ and $k_1\in\Z$ gives
\begin{equation}\label{IQ}
|I_Q^{\mu\nu}[f_1,f_2,f_3]|\lesssim \|f_1\|_{C_*^4}\|f_2\|_{H^{N-2}}\|f_3\|_{H^{N-1/2}}.
\end{equation}
A similar bound with $f_2$ and $f_3$ swapped holds.

By (\ref{IQ}) with $f_1=U_\mu$, $f_2=U_\nu$ and $f_3=U$,
\[
|(\ref{EQ-IBP1})|\lesssim \|U(t)\|_{H^{N-1/2}}^3+\|U(0)\|_{H^{N-1/2}}^3
\lesssim \|\tilde U\|_{L^\infty([0,t])H^N}^3.
\]
By (\ref{Ut}), (\ref{Gh-Cr2}) and (\ref{Gh-Cr}),
\begin{equation}\label{N-Cr}
\|N\|_{C_*^4}\lesssim \|U\|_{C_*^6}^2.
\end{equation}
so by (\ref{IQ}) with $f_1=N_\mu$, $f_2=U_\nu$ and $f_3=U$,
\[
|(\ref{EQ-IBP2})|\lesssim \|U\|_{L^2([0,t])C_*^6}^2\|\tilde U\|_{L^\infty([0,t])H^N}^2.
\]
By (\ref{IQ}) with $f_1=U_\mu$, $f_2=N_\nu$ and $f_3=U$ (and its symmetric version) and (\ref{N-HN}), (\ref{EQ-IBP3}) satisfies the same bound as (\ref{Es-IBP3}).

Combining the three bounds shows the claim.
\end{proof}

\subsection{Bounding the quasilinear energy $\mathcal E_{\tilde Q}$}
\begin{proposition}\label{Eq-tld}
If $\|h\|_{H^N}+\|U\|_{H^6}$ is sufficiently small then
\[
\left| \int_0^t \mathcal E_{\tilde Q}(s)ds \right|
\lesssim \|\tilde U\|_{L^\infty([0,t])H^N}^3
+\|U\|_{L^2([0,t])C_*^3}\|U\|_{L^2([0,t])C_*^5}\|\tilde U\|_{L^\infty([0,t])H^N}^2.
\]
\end{proposition}
\begin{proof}
By Proposition \ref{E-HN}, $\|\mathcal U\|_{H^N}$ and $\|\tilde U\|_{H^N}$ are interchangeable.

Up to a quartic error that can be bounded using Lemma \ref{Taf-Lp} (ii), Lemma \ref{Eaf-Lp} (ii) and (\ref{cal-U-tld-U-HN}), $\mathcal E_{\tilde Q}$ is a sum of the terms of the form $\mathcal E_{\tilde Q}^{\mu\nu}$, where
\begin{equation}\label{til-q-mult}
\begin{aligned}
\mathcal E_{\tilde Q}^{\mu\nu}&=\Re C^2\iint
\tilde q(\xi_1,\xi_2)\hat U_\mu(\xi_1)\mathcal{\hat U}_\nu(\xi_2)\overline{\mathcal{\hat U}(\xi_1+\xi_2)}d\xi_1 d\xi_2,\\
\tilde q(\xi_1,\xi_2)&=\varphi_{\le-10}\left( \frac{|\xi_1|}{|\xi_1+2\xi_2|} \right)n_1(\xi_1)n_2(\xi_2)n_3(\xi_1+\xi_2)\\
&\times\left[ n_4(\xi_1+\xi_2)n_5(\xi_2)-n_4\left(\frac{\xi_1+2\xi_2}2 \right)n_5\left(\frac{\xi_1+2\xi_2}2 \right) \right],
\end{aligned}
\end{equation}
$n_j\in S^{m_j}_{1,0}$ ($1\le j\le 5$), $\sum_{j=1}^5 {m_j}=2N+3/2$,
$m_1\ge1/2$ and $\supp n_2\cup\supp n_3\subset\supp\varphi_{\ge0}$. Let
\begin{align*}
I_{\tilde Q}^{\mu\nu}[f_1,f_2,f_3]
&=\Re C^2\iint
\frac{\tilde q(\xi_1,\xi_2)}{\Phi_{\mu\nu}(\xi_1,\xi_2)}\hat f_1(\xi_1)\hat f_2(\xi_2)\overline{\hat f_3(\xi_1+\xi_2)}d\xi_1 d\xi_2,\\
I_{\tilde Q}^{\mu\nu}&=I_{\tilde Q}^{\mu\nu}[U_\mu,\mathcal U_\nu,\mathcal U].
\end{align*}
Similarly integration by parts in time gives
\begin{align}
\label{Eq-IBP1}
\int_0^t \mathcal E_{\tilde Q}^{\mu\nu}(s)ds
&=I_{\tilde Q}^{\mu\nu}(t)-I_{\tilde Q}^{\mu\nu}(0)\\
\label{Eq-IBP2}
&-\int_0^t I_{\tilde Q}^{\mu\nu}[N_\mu,\mathcal U_\nu,\mathcal U](s)ds\\
\label{Eq-IBP3}
&-\int_0^t (I_{\tilde Q}^{\mu\nu}[U_\mu,(\mathcal U_\nu)_t+i\Lambda\mathcal U_\nu,\mathcal U]+I_{\tilde Q}^{\mu\nu}[U_\mu,\mathcal U_\nu,\mathcal U_t+i\Lambda\mathcal U])(s)ds.
\end{align}
The bound then follows from the corresponding bounds for (\ref{Eq-IBP1}), (\ref{Eq-IBP2}) and (\ref{Eq-IBP3}), to be shown in Proposition \ref{Eq-IBP12} and Proposition \ref{Eq-IBP33} below. 
\end{proof}

To estimate $I_{\tilde Q}^{\mu\nu}$, we need to bound the $S^\infty$ norm of the $\tilde q$ multiplier.
\begin{lemma}\label{tld-q-Soo}
For $k_1$, $k_2$, $k_3\in\Z$ we have
\begin{align*}
\|\tilde q\|_{S^\infty_{k_1,k_2;k_3}}
&\lesssim 2^{2Nk_3^++3k_1/2}1_{k_1\le k_2-6},\\
\|\nabla_{\xi_2}\tilde q\|_{S^\infty_{k_1,k_2;k_3}}
&\lesssim 2^{(2N-1)k_3^++3k_1/2}1_{k_1\le k_2-6}.
\end{align*}
\end{lemma}
\begin{proof}
The factor $1_{k_1\le k_2-6}$ comes from the $\varphi_{\le10}$ factor and is assumed to be nonzero thereafter. The bound itself follows from the identity
\[
n_4(\xi_1+\xi_2)n_5(\xi_2)-(n_4n_5)\left(\frac{\xi_1+2\xi_2}2 \right)\\
=\frac12\int_0^1 \xi_1\cdot\nabla(n_4(\xi_t)n_5(\eta_t))dt,
\]
where $\xi_t=((1+t)\xi_1+2\xi_2)/2$ and $\eta_t=((1-t)\xi_1+2\xi_2)/2$,
and the fact that $n_1\ge1/2$. Then we use Lemma \ref{Soo-Cn} (i) and (ii) to bound the $S^\infty$ norm of the integrand.
The bound on $\nabla_{\xi_2}q$ follows in a similar way.
\end{proof}

\begin{proposition}\label{Eq-IBP12}
If $\|U\|_{H^5}$ is sufficiently small then
\begin{align*}
|(\ref{Eq-IBP1})|&\lesssim \|\tilde U\|_{L^\infty([0,t])H^N}^3, &
|(\ref{Eq-IBP2})|&\lesssim \|U\|_{L^2([0,t])C_*^3}^2\|\tilde U\|_{L^\infty([0,t])H^N}^2.
\end{align*}
\end{proposition}
\begin{proof}
By Lemma \ref{ps-1/F-Soo} (iii) and Lemma \ref{tld-q-Soo}, $\|\tilde q/\Phi_{\mu\nu}\|_{S^\infty_{k_1,k_2;k_3}}\lesssim 2^{2Nk_3^++k_1}$.
By the additive restriction of the frequencies, $|k_2-k_3|\le2$,
so by Lemma \ref{paraprod},
\[
|I_{\tilde Q}^{\mu\nu}[P_{k_1}f_1,P_{k_2}f_2,P_{k_3}f_3]|
\lesssim 2^{2Nk_3^++k_1}1_{|k_2-k_3|\le2}
\|P_{k_1}f_1\|_{L^\infty}\|P_{k_2}f_2\|_{L^2}\|P_{k_3}f_3\|_{L^2}.
\]
Summing over $k_1\in\Z$ and $|k_2-k_3|\le2$ using the Cauchy--Schwarz inequality gives
\begin{equation}\label{Iq}
|I_{\tilde Q}^{\mu\nu}[f_1,f_2,f_3]|\lesssim \|f_1\|_{C_*^{3/2}}\|f_2\|_{H^N}\|f_3\|_{H^N}.
\end{equation}
By (\ref{Iq}) with $f_1=U_\mu$, $f_2=\mathcal U_\nu$, $f_3=\mathcal U$,
Proposition \ref{U-tld=U} and Proposition \ref{E-HN},
\[
|(\ref{Eq-IBP1})|\lesssim \|\tilde U(t)\|_{H^N}^3+\|\tilde U(0)\|_{H^N}^3
\lesssim \|\tilde U\|_{L^\infty([0,t])H^N}^3.
\]
For the same reason as (\ref{N-Cr}), $\|N\|_{C_*^{3/2}}\lesssim \|U\|_{C_*^3}^2$, so by (\ref{Iq}) with $f_1=N_\mu$, $f_2=\mathcal U_\nu$, $f_3=\mathcal U$ and Proposition \ref{E-HN}, the bound for (\ref{Eq-IBP2}) follows as well.
\end{proof}

\begin{proposition}\label{Eq-IBP33}
If $\|h\|_{H^N}+\|U\|_{H^6}$ is sufficiently small then
\[
|(\ref{Eq-IBP3})|\lesssim \|U\|_{L^2([0,t])C_*^3}\|U\|_{L^2([0,t])C_*^5}
\|\tilde U\|_{L^\infty([0,t])H^N}^2.
\]
\end{proposition}
\begin{proof}
Recall that $\|\mathcal U\|_{H^N}$ and $\|\tilde U\|_{H^N}$ are interchangeable.

By (\ref{cal-Ut}), the integrand of (\ref{Eq-IBP3}) becomes
\begin{align}
\label{Iq1}
&I_{\tilde Q}^{\mu\nu}[U_\mu,(\mathcal Q|_\Gamma+\mathcal{\tilde Q+\tilde S+C})_\nu,\mathcal U]
+I_{\tilde Q}^{\mu\nu}[U_\mu,\mathcal U_\nu,\mathcal Q|_\Gamma+\mathcal{\tilde Q+\tilde S+C}]\\
\label{Iq2}
-&I_{\tilde Q}^{\mu\nu}[U_\mu,(iT_{V\cdot\zeta}\mathcal U)_\nu,\mathcal U]
-I_{\tilde Q}^{\mu\nu}[U_\mu,\mathcal U_\nu,iT_{V\cdot\zeta}\mathcal U]\\
\label{Iq3}
-&I_{\tilde Q}^{\mu\nu}[U_\mu,(iT_{\sqrt{a\lambda}-|\zeta|^{1/2}}\mathcal U)_\nu,\mathcal U]
-I_{\tilde Q}^{\mu\nu}[U_\mu,\mathcal U_\nu,iT_{\sqrt{a\lambda}-|\zeta|^{1/2}}\mathcal U].
\end{align}
By (\ref{Iq}), (\ref{QS-HN}), (\ref{C-HN}), Lemma \ref{Sm-Lm}, Lemma \ref{Lmq=Lq}, Lemma \ref{prod-Lm}, Lemma \ref{Taf-Lp} (ii), Lemma \ref{Eaf-Lp} (ii) and the fact that $C_{-1,4}[U]\lesssim \|U\|_{H^5}$ is sufficiently small,
\begin{align*}
|(\ref{Iq1})|&\lesssim \|U\|_{C_*^1}\|\tilde U\|_{H^N}\|\mathcal Q|_\Gamma+\mathcal{\tilde Q+\tilde S+C}\|_{H^N}\\
&\lesssim \|U\|_{C_*^1}\|U\|_{C_*^5}(1+C_{-1,4}[U])\|\tilde U\|_{H^N}^2
\lesssim \|U\|_{C_*^1}\|U\|_{C_*^5}\|\tilde U\|_{H^N}^2.
\end{align*}

For (\ref{Iq2}), by Lemma \ref{para2diff} (ii), the operator $iT_{V\cdot\zeta}$ maps real valued functions to real valued functions,
so $(iT_{V\cdot\zeta}\mathcal U)_\nu=iT_{V\cdot\zeta}\mathcal U_\nu$.
Taking the complex conjugation on the third slot of $I_{\tilde Q}^{\mu\nu}$ into account we have
\begin{equation}\label{Iq2-4}
-(\ref{Iq2})=\Re C^3\iiint ir_{\mu\nu,j}(\xi,\eta,\theta)
\hat U_\mu(\xi-\eta-\theta)\mathcal{\hat U_\nu}(\eta)\overline{\mathcal{\hat U(\xi)}}\hat V_j(\theta)d\xi d\eta d\zeta,
\end{equation}
where
\begin{align*}
r_{\mu\nu,j}(\xi,\eta,\theta)&=\frac{2\eta_j+\theta_j}2\times\frac{\tilde q(\xi-\eta-\theta,\eta+\theta)}{\Phi_{\mu\nu}(\xi-\eta-\theta,\eta+\theta)}\varphi_{\le-10}\left( \frac{|\theta|}{|2\eta+\theta|} \right)\\
&-\frac{2\xi_j-\theta_j}2\times\frac{\tilde q(\xi-\eta-\theta,\eta)}{\Phi_{\mu\nu}(\xi-\eta-\theta,\eta)}\varphi_{\le-10}\left( \frac{|\theta|}{|2\xi-\theta|} \right).
\end{align*}

We distinguish two cases to bound the $S^\infty$ norm of
\[
\left( \partial_{\xi_2}\Phi_{\mu\nu}^{-1} \right)(\xi-\eta-\theta,\eta+t\theta)
=\frac{\nabla\Lambda(\xi-\theta+t\theta)-\nu\nabla\Lambda(\eta+t\theta)}
{\Phi_{\mu\nu}(\xi-\eta-\theta,\eta+t\theta)^2}.
\]

{\bf Case 1:} $\nu=-$. By Lemma \ref{ps-1/F-Soo} (iii'), if $\tilde q\neq0$ we have
$\|\Phi_{\mu\nu}^{-1}\|_{S^\infty_{k_1,k_2;k_3}}\lesssim 2^{-k_3/2}$.
Since $\Lambda$ is homogeneous of degree 1/2, by Lemma \ref{Soo-Cn} (i),
\begin{equation}\label{gr-1/F-Soo}
\|\partial_{\xi_2}\left( \Phi_{\mu\nu}^{-1} \right)(\xi-\eta-\theta,\eta+t\theta)\|_{S^\infty_{k_1,k_2;k_3}}\lesssim 2^{-3k_3/2}.
\end{equation}

{\bf Case 2:} $\nu=+$. By Lemma \ref{ps-1/F-Soo} (iii), if $\tilde q\neq0$ we have
$\|\Phi_{\mu\nu}^{-1}\|_{S^\infty_{k_1,k_2;k_3}}\lesssim 2^{-k_1/2}$.
For the numerator we have
\[
\|\text{numerator}\|_{S^\infty_{k_1,k_2;k_3}}
\lesssim 2^{k_1}\sup_{s\in[0,1]}\|\nabla^2\Lambda(\eta+s(\xi-\eta-\theta)+t\theta)\|_{S^\infty_{k_1,k_2;k_3}}
\lesssim 2^{k_1-3k_3/2}
\]
so by Lemma \ref{Soo-Cn} (i), (\ref{gr-1/F-Soo}) also holds.

Now we use the fundamental theorem of calculus as in Lemma \ref{tld-q-Soo} to obtain, for $k_4\in\N\cup\{\le0\}$ (we use the convention $2^{\le0}=1$),
\[
\|r_{\mu\nu,j}(\xi,\eta,\theta)\varphi_{k_4}(\theta)\|_{S^\infty_{k_1,k_2;k_3}}
\lesssim 2^{2Nk_3^+}(2^{2k_1}+2^{k_1+k_4})1_{k_1,k_4\le k_3-5}.
\]
Using Lemma \ref{paraprod} and (\ref{Gh-Cr}) and summing over $k_1,k_4\le k_3-5$ and $k_2=k_3+O(1)$ give
\[
|(\ref{Iq2})|\lesssim \|U\|_{C_*^2}\|U\|_{C_*^3}\|\tilde U\|_{H^N}^2.
\]

For (\ref{Iq3}) we distinguish two cases.

{\bf Case 1:} $\nu=+$. Then we can omit the subscript $\nu$ and get the same cancellation as in (\ref{Iq2}) from the complex conjugation in the third slot of $I_{\tilde Q}^{\mu+}$. By (\ref{a-1-Cr}), $\sum_{k\in\N\cup\{\le0\}} 2^k\|P_k(\sqrt{a\lambda}-\sqrt{|\zeta|})\|_{\mathcal L^\infty_{1/2}}\lesssim\|a-1\|_{C_*^{3/2}}+\|h\|_{C_*^4}\lesssim\|U\|_{C_*^4}$, so by Lemma \ref{paraprod},
\[
|(\ref{Iq3})|\lesssim \|U\|_{C_*^3}\|U\|_{C_*^4}\|\tilde U\|_{H^N}^2.
\]

{\bf Case 2:} $\nu=-$. By Lemma \ref{ps-1/F-Soo} (iii'), if $\tilde q\neq0$ we have
$\|\Phi_{\mu\nu}^{-1}\|_{S^\infty_{k_1,k_2;k_3}}\lesssim 2^{-k_3/2}$, so
\[
\|\tilde q/\Phi_{\mu\nu}\|_{S^\infty_{k_1,k_2;k_3}}\lesssim 2^{(2N-1/2)k_3+3k_1/2}.
\]
This can be used to obtain the desired bound by recovering the loss of derivative in $T_{\sqrt{a\lambda}-|\zeta|^{1/2}}\mathcal U$.
By Lemma \ref{Taf-Lp} (ii), (\ref{a-1-Loo}) and (\ref{l-|z|-Loo}),
\[
\|T_{\sqrt{a\lambda}-|\zeta|^{1/2}}\mathcal U\|_{H^{N-1/2}}\lesssim \|U\|_{C_*^4}\|\tilde U\|_{H^N}.
\]
Then Lemma \ref{paraprod} gives
\[
|(\ref{Iq3})|\lesssim \|U\|_{C_*^2}\|U\|_{C_*^4}\|\tilde U\|_{H^N}^2.
\]

Combining the three bounds and integrating in time show the claim.
\end{proof}

\subsection{Quartic energy estimates}
Recall from (\ref{E-def}) that the energy
\[
E=\int_{\R^2} \frac12(h^2+\phi G(h)\phi)
\]
is conserved. Now we can show quartic energy estimates (\ref{growth-Em-X}) and (\ref{growth-Em-Z}).
\begin{proof}[Proof of (\ref{growth-Em-X})]
Recall $U = h + i|\nabla|^{1/2}\phi$. By the Cauchy--Schwarz inequality,
\[
|2E-\|U\|_{L^2}^2|=|\langle\phi,(G(h)-|\nabla|)\phi\rangle|
\le\|U\|_{L^2}\|(G(h)-|\nabla|)\phi\|_{\dot H^{-1/2}}.
\]
By (1.1.15) of \cite{AlDe2}, $G(h):\dot H^{1/2}\to\dot H^{-1/2}$ is bounded, locally uniformly in $h$.
Then by (\ref{Gh-phi-int}), (\ref{ds-Gh-phi}), Sobolev multiplication and (\ref{Gh-Hs}),
\begin{align*}
\|(G(h)-|\nabla|)\phi\|_{\dot H^{-1/2}}
&\le\sup_{s\in[0,1]} (\|G(sh)[hB(sh)\phi]\|_{\dot H^{-1/2}}+\|hV(sh)\phi\|_{H^{1/2}})\\
&\lesssim \sup_{s\in[0,1]} (\|hB(sh)\phi\|_{H^{1/2}}+\|hV(sh)\phi\|_{H^{1/2}})\\
&\lesssim \|h\|_{H^2}\||\nabla|^{1/2}\phi\|_{H^2}\le\|U\|_{H^2}^2.
\end{align*}
Hence
\begin{equation}\label{E-U}
|2E-\|U\|_{L^2}^2|\lesssim\|U\|_{H^2}^3.
\end{equation}
By (\ref{HN0}) we know that $\|U(0)\|_{H^2}\lesssim\ep$.
By (\ref{E-U}) and conservation of energy, $E(t)=E(0)\lesssim\ep^2$.
By (\ref{growthX1}) and Proposition \ref{U-tld=U},
$\|U(t)\|_{H^2}\lesssim\ep_1$, so by (\ref{E-U}) again,
$\|U(t)\|_{L^2}^2\lesssim\ep^2+\ep_1^3$. Then by Sobolev multiplication and (\ref{Gh-Hs}),
\[
\|\tilde U(t)\|_{H^{-1/2}}^2\lesssim \|U(t)\|_{L^2}^2+\|B(t)\|_{L^\infty}^2\|h(t)\|_{L^2}^2\lesssim(1+\|U(t)\|_{H^2}^2)\|U(t)\|_{L^2}^2
\lesssim\ep^2+\ep_1^3.
\]

By (\ref{HN0}), Proposition \ref{U-tld=U} and Proposition \ref{E-HN},
$\mathcal E(0)\lesssim\ep^2$ and $\|P_{\ge0}\tilde U(t)\|_{H^N}^2=\mathcal E(t)+O(\ep_1^3)$. By Proposition \ref{E4}, Proposition \ref{Es}, Proposition \ref{Eq} and Proposition \ref{Eq-tld} we have
\[
\mathcal E(t)=\mathcal E(0)+O(\ep_1^3+\ep_1^2\|U\|_{L^2([0,t])C_*^6}\|C_{-[\log(1+s)]-1,6}[U(s)]\|_{L^2([0,t])}).
\]
Recalling (\ref{Ckr-def}) we have
\begin{align*}
\|C_{-[\log(1+s)]-1,6}[U(s)]\|_{L^2([0,t])}&\lesssim\mathcal L\|U\|_{L^2([0,t])C_*^6}+\|(1+s)^{-1}\|_{L^2([0,t])}\|U\|_{L^\infty([0,t])L_x^2}\\
&\lesssim\ep_1+\mathcal L\ep_2,
\end{align*}
Since $\|U\|_{L^2([0,t])C_*^6}\lesssim\ep_2$ is sufficiently small,
\[
\mathcal E(t)\lesssim\ep^2+\ep_1^3+\ep_1^3\ep_2+\mathcal L\ep_1^2\ep_2^2
\lesssim\ep^2+\ep_1^3+\mathcal L\ep_1^2\ep_2^2
\]
and the same bound holds for $\|P_{\ge0}\tilde U(t)\|_{H^N}^2$. Then
\[
\|\tilde U(t)\|_{H^N}^2\lesssim \|P_{\ge0}\tilde U(t)\|_{H^N}^2+\|\tilde U(t)\|_{H^{-1/2}}^2\lesssim \ep^2+\ep_1^3+\mathcal L\ep_1^2\ep_2^2.
\]
Taking the square root gives (\ref{growth-Em-X}).
\end{proof}

\begin{proof}[Proof of (\ref{growth-Em-Z})]
The bound for $\|\tilde U(t)\|_{H^{-1/2}}$ remains the same as above.
For $\mathcal E(t)$, note that
\begin{align*}
\|C_{-[\log(1+s)]-1,6}[U(s)]\|_{L^2([0,t])}&\lesssim\|\log(2+2s)U(s)\|_{L^2([0,t])C_*^6}\\
&+\|(1+s)^{-1}\|_{L^2([0,t])}\|U\|_{L^\infty([0,t])L_x^2}.
\end{align*}
Summing (\ref{growthZ}) in $k \in \Z$ we get $\|(1+s)^{(\alpha-\delta)/2}U(s)\|_{L^2([0,t])C_*^6}\lesssim\ep_1$, so the left-hand side of the above $\lesssim \ep_1$. Then $\mathcal E(t) \lesssim \ep^2 + \ep_1^3$, and the same argument as above shows (\ref{growth-Em-Z}).
\end{proof}

\section{Strichartz estimates}\label{StrEst}
This section is devoted to the proof of the Strichartz estimate (\ref{growthX2}), which is contained in Proposition \ref{U3-C5}, Proposition \ref{W2-C6} and Proposition \ref{H2-C5} below.

\subsection{Definition of the profile}
Recall the evolution equation $U_t + i|\nabla|^{1/2}U = N$, where (see (\ref{Ut}) and (\ref{Gh-phi-int2}))
\begin{equation}\label{N23}
\begin{aligned}
N&=(G(h)-|\nabla|)\phi+\frac i2|\nabla|^{1/2}((1+|\nabla h|^2)B^2-|\nabla\phi|^2)=N_2+N_3,\\
N_2&=-|\nabla|(h|\nabla|\phi)-\nabla\cdot(h\nabla\phi)+\frac i2|\nabla|^{1/2}((|\nabla|\phi)^2-|\nabla\phi|^2),\\
N_3&=B_3+|\nabla h|^2B+\frac i2|\nabla|^{1/2}(B^2-(|\nabla|\phi)^2+|\nabla h|^2B^2).
\end{aligned}
\end{equation}
$N_2$ can be expressed as sums of the terms $N_{\mu\nu}=N_{\mu\nu}[U_\mu,U_\nu]$, where $\mu,\nu=\pm$,
\[
\mathcal FN_{\mu\nu}(\xi) = \int_{\xi_1+\xi_2=\xi} m_{\mu\nu}(\xi_1,\xi_2)\hat U_\mu(\xi_1)\hat U_\nu(\xi_2)d\xi_1
\]
and $m_{\mu\nu}(\xi_1,\xi_2)$ are linear combinations of multipliers in the set
\begin{equation}\label{m-prod}
\left\{\frac{|\xi_1+\xi_2||\xi_2|-(\xi_1+\xi_2)\cdot\xi_2}{\sqrt{|\xi_2|}},
\sqrt{|\xi_1+\xi_2|}\frac{|\xi_1||\xi_2|+\xi_1\cdot\xi_2}{\sqrt{|\xi_1||\xi_2|}}\right\}.
\end{equation}

By Duhamel's formula,
\begin{equation}\label{Duhamel}
U(t)=e^{-it\Lambda}U(0)+U_2(t)+U_3(t),\ 
U_j(t)=\int_0^t e^{-i(t-s)\Lambda}N_j(s)ds,\ j\in\{2,3\}.
\end{equation}
Define the profile
\[
\Upsilon_\pm(t) = e^{\pm it\Lambda}U_\pm(t)\text{ and }
\Upsilon_j(t) = e^{it\Lambda}U_j(t),\text{ where }
\Lambda = |\nabla|^{1/2}\text{ and }j \in \{2, 3\}.
\]
Then the evolution equation for $\Upsilon:=\Upsilon_+$ is
\[
\Upsilon_t=e^{it\Lambda}(i\Lambda U+U_t)=e^{it\Lambda}N.
\]

By Lemma \ref{dispersiveTT*} (ii) and (\ref{growthX1}),
\[
\|e^{-is\Lambda}U(0)\|_{L^2([0,t])C_*^6}\lesssim \sqrt{\mathcal L}\cdot\ep_1.
\]
To bound $U_2$, we observe that $\Phi_{\mu\nu}(\xi_1,\xi_2)\neq0$ unless $\xi_1$ or $\xi_2$ or $\xi_1+\xi_2=0$, in which case $m_{\mu\nu}(\xi_1,\xi_2)=0$, so we can integrate $\Upsilon_2$ by parts in $s$ to get
\begin{align}
\label{F=W+H}
\Upsilon_2(t)&=\sum_{\mu,\nu=\pm} \left( W_{\mu\nu}(t)-W_{\mu\nu}(0)
-\int_0^t H_{\mu\nu}(s)ds \right),\\
\label{W2-F}
\mathcal FW_{\mu\nu}(\xi,t)&=C\int_{\xi_1+\xi_2=\xi}
e^{it\Phi_{\mu\nu}(\xi_1,\xi_2)}\frac{m_{\mu\nu}(\xi_1,\xi_2)}{i\Phi_{\mu\nu}(\xi_1,\xi_2)}\hat\Upsilon_\mu(\xi_1,t)\hat\Upsilon_\nu(\xi_2,t)d\xi_1,\\
\nonumber
\mathcal FH_{\mu\nu}(\xi,t)&=C\int_{\xi_1+\xi_2=\xi}
e^{it\Phi_{\mu\nu}(\xi_1,\xi_2)}\frac{m_{\mu\nu}(\xi_1,\xi_2)}{i\Phi_{\mu\nu}(\xi_1,\xi_2)}\\
&\times(\hat\Upsilon_\mu(\xi_1,t)e^{it\nu\Lambda(\xi_2)}\hat N_\nu(\xi_2,t)+e^{it\mu\Lambda(\xi_1)}\hat N_\mu(\xi_1,t)\hat\Upsilon_\nu(\xi_2,t))d\xi_1.
\label{H2-FN-NF}
\end{align}

We need to bound the $S^\infty$ norms of the multipliers in the equations above.
\begin{lemma}\label{m-Soo-lem}
For $k_1$, $k_2$, $k_3\in\Z$ we have
\begin{align}
\label{m-Soo}
\|m_{\mu\nu}\|_{S^\infty_{k_1,k_2;k_3}}&\lesssim 2^{(k_3+\max k_j+\min k_j)/2},\\
\|m_{\mu\nu}/\Phi_{\mu\nu}\|_{S^\infty_{k_1,k_2;k_3}}&\lesssim 2^{(k_3+\max k_j)/2}.
\label{m/F-Soo}
\end{align}
\end{lemma}
\begin{proof}
This is shown in section \ref{MulEst}.
\end{proof}

Let $m_{\mu\nu\rho}=m_{\mu\sigma}m_{\nu\rho}/\Phi_{\mu\sigma}$.
Then by Lemma \ref{m-Soo-lem} and Lemma \ref{Soo-Cn} (i), for $k$, $k_j$, $l\in\Z$ we have
\begin{equation}\label{mk3-Soo}
\|\varphi_l(\xi_2+\xi_3)m_{\mu\nu\rho}(\xi_1,\xi_2,\xi_3)\|_{S^\infty_{k_1,k_2,k_3;k}}\lesssim 2^{(k+l+3\max k_j)/2},
\end{equation}
Similar bounds hold if $m_{\mu\nu\rho}=m_{\sigma\rho}m_{\mu\nu}/\Phi_{\sigma\rho}$.

Using these bounds, we can estimate the $L^2([0,t])C_*^6$ norms of $U_2$ (see Proposition \ref{W2-C6} and Proposition \ref{H2-C5} below) and $U_3$ (see Proposition \ref{U3-C5} below). Putting them together we get a proof of (\ref{growthX2}).

To set up the notation for later proofs, we rewrite (\ref{W2-F}) as
\begin{equation}\label{W2-U}
e^{-it\Lambda(\xi)}\hat W_{\mu\nu}(\xi,t)=C\int_{\xi_1+\xi_2=\xi}
\frac{m_{\mu\nu}(\xi_1,\xi_2)}{i\Phi_{\mu\nu}(\xi_1,\xi_2)}\hat U_\mu(\xi_1,t)\hat U_\nu(\xi_2,t)d\xi_1.
\end{equation}
We view $N_{\mu\nu}$ and $W_{\mu\nu}$ as bilinear forms of $\Upsilon_\mu$ and $\Upsilon_\nu$, and decompose
\begin{align}\label{W2k}
W_{\mu\nu}&=\sum_{k_1,k_2\in\Z} W^{\mu\nu}_{k_1,k_2}, &
W^{\mu\nu}_{k_1,k_2}&=W_{\mu\nu}[P_{k_1}\Upsilon_\mu,P_{k_2}\Upsilon_\nu].
\end{align}
By symmetry we can assume $k_1\le k_2$. Since by (\ref{m/F-Soo}) and $S^\infty\subset L^\infty$,
\begin{equation}\label{m-C0}
|(m_{\mu\nu}/\Phi_{\mu\nu})(\xi_1,\xi_2)|\lesssim \sqrt{|\xi_1+\xi_2|(|\xi_1|+|\xi_2|)},
\end{equation}
we have
\begin{align}
\nonumber
\|W^{\mu\nu}_{k_1,k_2}\|_{L^2}
&\lesssim 2^{k_2}\|\varphi_{k_1}\mathcal FU\|_{L^1}\|\varphi_{k_2}\mathcal FU\|_{L^2}\\
\label{WL2}
&\lesssim 2^{-(N-1/2)k_2^++k_2+k_1}\|P_{k_1}U\|_{L^2}\|P_{k_2}U\|_{H^{N-1/2}},\\
\nonumber
\|P_kW^{\mu\nu}_{k_1,k_2}\|_{L^2}
&\lesssim 2^k\|\mathcal FW^{\mu\nu}_{k_1,k_2}\|_{L^\infty}
\lesssim 2^{(k_2+3k)/2}\|P_{k_1}U\|_{L^2}\|P_{k_2}U\|_{L^2}\\
&\lesssim 2^{-(N-1/2)k_2^++(k_2+3k)/2}\|P_{k_1}U\|_{L^2}\|P_{k_2}U\|_{H^{N-1/2}}.
\label{WL22}
\end{align}

\subsection{Bounding the quadratic boundary terms $W_{\mu\nu}$}
\begin{proposition}\label{W2-C6}
Assume $N\ge9$ and (\ref{growthX1}). Then
\[
\|e^{-is\Lambda}(W_{\mu\nu}(s)-W_{\mu\nu}(0))\|_{L^2([0,t])C_*^6}
\lesssim \sqrt{\mathcal L}\cdot\ep_1^2+\ep_1\ep_2.
\]
\end{proposition}
\begin{proof}
Since $W_{\mu\nu}$ loses only one derivative, by Theorem C.1 of \cite{GeMaSh2}, Sobolev embedding, Proposition \ref{U-tld=U}, (\ref{growthX1}) and the fact that $N\ge9$,
\begin{equation}\label{W0-H7}
\|W_{\mu\nu}(0)\|_{H^7}\lesssim \|U(0)\|_{W^{8,4}}\|U(0)\|_{L^4}
\lesssim \|U(0)\|_{H^{8.5}}^2\lesssim\ep_1^2.
\end{equation}
Then by Lemma \ref{dispersiveTT*} (ii),
\begin{equation}\label{W0-L2C6}
\|e^{-is\Lambda}W_{\mu\nu}(0)\|_{L^2([0,t])C_*^6}
\lesssim \sqrt{\mathcal L}\cdot\ep_1^2.
\end{equation}
Similarly to (\ref{W0-H7}),
\begin{equation}\label{Ws-H7}
\|e^{-is\Lambda}W_{\mu\nu}(s)\|_{C_*^6}
\lesssim \|W_{\mu\nu}(s)\|_{H^7}
\lesssim \|U(s)\|_{W^{8,2.4}}\|U(s)\|_{L^{12}}.
\end{equation}
Since $W^{8,2.4}=F^8_{2.4,2}\supset B^8_{2.4,2}$ (see Section 2.3.2 and 2.3.3 of \cite{Tr2}), which interpolates between $B^6_{\infty,\infty}=C_*^6$ and $B^{8.4}_{2,5/3}\supset H^{8.5}$, and $L^{12}$ interpolates in the same way between $L^2$ and $L^\infty$, we have
\begin{equation}\label{Ws-Wkp-interpol}
\|e^{-is\Lambda}W_{\mu\nu}(s)\|_{C_*^6}
\lesssim \|U(s)\|_{C_*^6}\|U(s)\|_{H^{8.5}}.
\end{equation}
Since $N\ge9$, by (\ref{growthX1}) and Proposition \ref{U-tld=U} we have
$\|U(s)\|_{H^{8.5}}\lesssim\ep_1$. Then
\begin{equation}\label{Ws-L2C6}
\|e^{-is\Lambda}W_{\mu\nu}(s)\|_{L^2([0,t])C_*^6}
\lesssim \|U(s)\|_{L^2([0,t])C_*^6}\ep_1\lesssim \ep_1\ep_2.
\end{equation}
Combining (\ref{W0-L2C6}) and (\ref{Ws-L2C6}) shows the claim.
\end{proof}

\subsection{Bounding the cubic bulk term $U_3$}
\begin{proposition}\label{U3-C5}
Assume $N\ge11$ and (\ref{growthX1}). Then
\[
\|U_3\|_{L^2([0,t])C_*^6}\lesssim \sqrt{\mathcal L}\ep_1\ep_2^2.
\]
\end{proposition}
\begin{proof}
By Sobolev multiplication, (\ref{Gh-Hs}), (\ref{Gh-Cr2}), (\ref{Gh-Cr}), (\ref{BV-Hs2}) and Proposition \ref{U-tld=U}, if $\|h\|_{H^{r+3}}<c_r$ is sufficiently small then
\begin{align*}
&\|B^2-(|\nabla|\phi)^2\|_{H^{r+1/2}}\\
\lesssim_r&\|B+|\nabla|\phi\|_{H^{r+1/2}}\|B-|\nabla|\phi\|_{L^\infty}
+\|B+|\nabla|\phi\|_{L^\infty}\|B-|\nabla|\phi\|_{H^{r+1/2}}\\
\lesssim_r&\||\nabla|^{1/2}\phi\|_{H^{r+2}}\|h\|_{C_*^{3/2}}\||\nabla|^{1/2}\phi\|_{C_*^2}+\||\nabla|^{1/2}\phi\|_{C_*^1}^2\|h\|_{H^{r+3/2}}
\lesssim \|U\|_{C_*^2}^2\|\tilde U\|_{H^{r+5/2}}.
\end{align*}
Since $\|\nabla h\|_{L^\infty}\lesssim \|h\|_{H^3}\lesssim 1$,
the same bound applies to $|\nabla h|^2B$ and $|\nabla h|^2B^2$.
Putting this and (\ref{B3-Hs}) into (\ref{N23}) we know that,
if $r>1$ and $\|h\|_{H^{r+7/2}}<c_r$ is sufficiently small then
\begin{equation}\label{N3-Hs}
\|N_3\|_{H^r}\lesssim_r \|U\|_{C_*^2}^2\|U\|_{H^{r+5/2}}
\lesssim_r \|U\|_{C_*^2}^2\|\tilde U\|_{H^{r+3}}.
\end{equation}
By Lemma \ref{dispersiveTT*} (ii), Bernstein's inequality and $N \ge 11$ we then have
\begin{equation}\label{U3-L2Loo}
\begin{aligned}
\|U_3\|_{L^2([0,t])C_*^6}
&\le\int_0^t \|e^{-i(s-s')\Lambda}N_3(s')\|_{L^2_s([s',t])C_*^6}ds'\\
&\lesssim \sqrt{\mathcal L}\|e^{is'\Lambda}N_3(s')\|_{L^1_{s'}([0,t])H^7}\\
&\lesssim \sqrt{\mathcal L}\|U\|_{L^2([0,t])C_*^2}^2\|\tilde U\|_{L^\infty([0,t])H^{10}}
\end{aligned}
\end{equation}
from which the result follows.
\end{proof}

\subsection{Bounding the cubic bulk term $H_{\mu\nu}$}
\begin{proposition}\label{H2-C5}
Assume $N\ge11$ and (\ref{growthX1}). Then
\[
\left\| e^{-is\Lambda}\int_0^s H_{\mu\nu}(s')ds' \right\|_{L^2_s([0,t])C_*^6}\lesssim \sqrt{\mathcal L}\ep_1\ep_2^2.
\]
\end{proposition}
\begin{proof}
A similar computation to that in (\ref{U3-L2Loo}) shows that it suffices to show
\begin{equation}\label{H-L1H7}
\|H_{\mu\nu}(s)\|_{L_s^1([0,t])H^7} \lesssim \ep_1\ep_2^2.
\end{equation}

Since $W_{\mu\nu}$ loses only one derivative, by Theorem C.1 of \cite{GeMaSh2},
\begin{equation}\label{H-H7}
\|H_{\mu\nu}\|_{H^7}\lesssim \|N\|_{L^{22}}\|U\|_{W^{8,2.2}}+\|U\|_{L^{22}}\|N\|_{W^{8,2.2}}.
\end{equation}
We have
\begin{equation}\label{N-Lp-interpol}
\|N\|_{L^{22}}\le\|N\|_{L^2}^{1/11}\|N\|_{L^\infty}^{10/11}.
\end{equation}
By (\ref{Gh-Hs2}) and Proposition \ref{U-tld=U},
if $r>1$ and $\|h\|_{H^{r+5/2}}<c_r$ is small enough then
\begin{equation}\label{N-Hr}
\|(G(h)-|\nabla|)\phi\|_{H^r}\lesssim \|U\|_{C_*^1}\|U\|_{H^{r+3/2}}
\lesssim \|U\|_{C_*^1}\|\tilde U\|_{H^{r+2}}.
\end{equation}
By (\ref{Gh-Cr}), (\ref{Gh-Hs}) and Sobolev multiplication,
the same bound holds for the second term of (\ref{Ut}), and hence $N$. Thus
$\|N\|_{L^2}\lesssim \ep_1\|U\|_{C_*^1}$.
For the same reason as (\ref{N-Cr}), $\|N\|_{L^\infty}\lesssim\|U\|_{C_*^2}^2$. Putting these two bounds into (\ref{N-Lp-interpol}) we get
\begin{equation}\label{N-Lp}
\|N\|_{L^{22}}\lesssim \ep_1^{1/11}\|U\|_{C_*^2}^{21/11}.
\end{equation}
By the same interpolation argument as used in (\ref{Ws-Wkp-interpol}),
\begin{equation}\label{U-W8q}
\|U\|_{W^{8,2.2}}\lesssim \|U\|_{C_*^6}^{1/11}\|U\|_{H^{8.2+}}^{10/11}
\lesssim \|U\|_{C_*^6}^{1/11}\ep_1^{10/11}.
\end{equation}
We have
\begin{equation}\label{U-Lp}
\|U\|_{L^{22}}\le\|U\|_{L^2}^{1/11}\|U\|_{L^\infty}^{10/11}
\lesssim\ep_1^{1/11}\|U\|_{L^\infty}^{10/11}.
\end{equation}
Similarly to (\ref{U-W8q}) we have
\[
\|N\|_{W^{8,2.2}}\lesssim \|N\|_{C_*^4}^{1/11}\|N\|_{H^{8.4+}}^{10/11}
\lesssim \|N\|_{C_*^4}^{1/11}\|N\|_{H^{8.5}}^{10/11}.
\]
By (\ref{N-Cr}), $\|N\|_{C_*^4}\lesssim \|U\|_{C_*^6}^2$. By (\ref{N-Hr}) and $N\ge11$, $\|N\|_{H^{8.5}}\lesssim \|U\|_{C_*^1}\ep_1$, so
\begin{equation}\label{N-W8q}
\|N\|_{W^{8,2.2}}\lesssim \|U\|_{C_*^6}^{12/11}\ep_1^{10/11}.
\end{equation}

Now putting (\ref{N-Lp}), (\ref{U-W8q}), (\ref{U-Lp}) and (\ref{N-W8q}) into (\ref{H-H7}) we get
\[
\|H_{\mu\nu}\|_{H^7}\lesssim \ep_1\|U\|_{C_*^6}^2.
\]
Integrating in $t$ gives the result.
\end{proof}

\section{$L^2L^\infty$ and $Z$-norm estimates}\label{ZEst}
This section is devoted to the proof of the $L^2L^\infty$ norm estimate (\ref{growth-L2Loo}) and the $Z$ norm estimate (\ref{growthZ2}).

\subsection{Integration by parts in phase space}
To bound the $Z$ norm we need a lemma to integrate by parts in phase space.
\begin{lemma}[Lemma 6.1 of \cite{Zh}]\label{int-part}
Let $0<\ep\le1/\ep\le K$. Suppose $f$, $g:\R^d\to\R$ satisfies
\[
|\nabla f|\ge 1_{\supp g},\text{ and for all }L\ge2,\ 
|\nabla^L f|\lesssim_L \ep^{1-L}\text{ on }\supp g.
\]
Then
\[
\left| \int e^{iKf}g \right|\lesssim_{d,L} (K\ep)^{-L} \sum_{l=0}^L \ep^l\|\nabla^lg\|_{L^1}.
\]
\end{lemma}

\begin{lemma}\label{int-part-cor}
Let $r\ge2$ and $0<\ep\le1/\ep\le r/2$. Suppose $f$, $g:\R^d\to\R$ satisfies $|\nabla f|\le r/2$ on $\supp g$, for all $L\ge2$, $|\nabla^Lf|\lesssim_L r\ep^{1-L}$ on $\supp g$, and for all $L\ge0$, $\|\nabla^Lg\|_{L^1}\lesssim_L A\ep^{-L}$ on $\R^d$. Let
\[
\mathcal K(x)=\int e^{i(x\cdot\xi+f(\xi))} g(\xi)d\xi.
\]
Then for all $L\ge 1$ we have
\[
\|\mathcal K\|_{L^1(\R^d\backslash B(0,r))}\lesssim_L Ar^{-L} \ep^{-L-d}.
\]
\end{lemma}
\begin{proof}
We have
\[
x\cdot\xi+f(\xi)=\frac{|x|}2\left( \frac{2x}{|x|}\cdot\xi+\frac{2f(\xi)}{|x|} \right).
\]
Suppose $|x|\ge r$. Let $K=|x|/2\ge r/2\ge1/\ep$, and
\[
F(\xi)=\frac{2x}{|x|}\cdot\xi+\frac{2f(\xi)}{|x|}.
\]
Then $|\nabla F|\ge1$ and for all $L\ge2$, $|\nabla^LF|\lesssim_L \ep^{1-L}$.
By Lemma \ref{int-part}, for $|x|\ge r$ we have
\[
|\mathcal K(x)|\lesssim_L A(|x|\ep)^{-L}.
\]
The result follows from integrating this bound with $L+d$ in place of $L$.
\end{proof}

\subsection{Bounding the quadratic boundary terms}
\begin{proposition}\label{W-Z}
Assume $N\ge\max(9/(2-2\alpha),30)$, $\alpha<1/2$ and (\ref{growthZ}). Then
\[
\|W_{\mu\nu}(t)\|_Z\lesssim \ep_1^2.
\]
\end{proposition}
\begin{proof}
We use the decomposition (\ref{W2k}) and assume by symmetry $k_1\le k_2$.
We distinguish several cases to estimate
\[
\|W_{\mu\nu}\|_Z\approx\left\| 2^{\alpha j}\|Q_j\langle\nabla\rangle^8W_{\mu\nu}\|_{L^2} \right\|_{\ell^2_{j\ge0}}
\]

{\bf Case 1:} $k_2\ge1.5\alpha j/N$. We sum (\ref{WL2}) over $k_2$ in this range and $k_1\in\Z$ and use $(9.5-N)k_2\le-2Nk_2/3-k_2/2\le-(1+0.5/N)\alpha j$ to get
\[
\left\| \sum_{k_1\in\Z\atop k_2\ge1.5\alpha j/N} 2^{\alpha j}\|Q_j\langle\nabla\rangle^8W^{\mu\nu}_{k_1,k_2}(t)\|_{L^2} \right\|_{\ell^2_{j\ge0}}
\lesssim \sum_{j\ge0} 2^{-0.5\alpha j/N}\ep_1^2\lesssim \ep_1^2.
\]

{\bf Case 2:} $k_1\le-(1+\alpha)\max(j,\mathcal L)/2$. We sum (\ref{WL2}) with $k_1$ in this range and $k_2\in\Z$ and use $N\ge10$ to get
\begin{align*}
\left\| \sum_{k_1\le-(\alpha+1/3)\max(j,\mathcal L)} 2^{\alpha j}\|Q_j\langle\nabla\rangle^8W^{\mu\nu}_{k_1,k_2}(t)\|_{L^2} \right\|_{\ell^2_{j\ge0}}
&\lesssim \sum_{j\ge0} 2^{\alpha j-(1+\alpha)\max(j,\mathcal L)/2}\ep_1^2\\
&\lesssim \ep_1^2.
\end{align*}

{\bf Case 3:} $P_{k_3}W^{\mu\nu}_{k_1,k_2}$ for $k_3\le-(1+\alpha)\max(j,\mathcal L)/2<k_1$. We sum (\ref{WL22}) with $k_3$ in place of $k$,
$k_1\in[k_3,k_2]$ and $k_2\in\Z$ to get the same bound as Case 2.

{\bf Case 4:} $\kappa:=\min k_i>-(1+\alpha)\max(j,\mathcal L)/2$ and $k_1$, $k_2<1.5\alpha j/N$.

{\bf Case 4.1:} $j\le-\kappa/2+\mathcal L+5$. In this case either
$\kappa>-\mathcal L$ or $\kappa>(1+\alpha)(\kappa/2-\mathcal L-5)/2$.
In the latter case $\kappa>-\frac{2+2\alpha}{3-\alpha}(\mathcal L+5)$,
which thus holds either way. Then $j<4(\mathcal L+5)/(3-\alpha)$ and
$k_i<3\alpha\mathcal L/N+O(1)$.

We decompose
\begin{align*}
W^{\mu\nu}_{k_1,k_2}&=\sum_{j_1,j_2\ge0} W^{\mu\nu}_{j_1,k_1,j_2,k_2},\\
W^{\mu\nu}_{j_1,k_1,j_2,k_2}&=W_{\mu\nu}[P_{[k_1-1,k_1+1]}Q_{j_1}P_{k_1}\Upsilon_\mu,P_{[k_2-1,k_2+1]}Q_{j_2}P_{k_2}\Upsilon_\nu],
\end{align*}
where we have used $P_{[k-1,k+1]}P_k=P_k$.

{\bf Case 4.1.1:} $j_2>\mathcal L$. By Lemma \ref{paraprod}, (\ref{m/F-Soo}), (\ref{dispersive3}), unitarity of $e^{it\Lambda}$, H\"older's inequality and Lemma \ref{Z-bound} (ii) we have, for $\beta\in[\alpha-1,\alpha]$,
\begin{align*}
&\|\langle\nabla\rangle^8P_{k_3}W^{\mu\nu}_{j_1,k_1,j_2,k_2}(t)\|_{L^2}\\
\lesssim&2^{8k_2^++(k_3+k_2)/2}\|e^{-it\mu\Lambda}P_{[k_1-1,k_1+1]}Q_{j_1}P_{k_1}\Upsilon_\mu(t)\|_{L^\infty}\|e^{-it\nu\Lambda}Q_{j_2}P_{k_2}\Upsilon_\nu(t)\|_{L^2}\\
\lesssim&2^{8k_2^++(k_3+k_2+3k_1^++k_1)/2}\frac{1}{(1+t)^{\alpha-\beta}}
\|Q_{j_1}P_{k_1}\Upsilon(t)\|_{L^{2/(1+\alpha-\beta)}}\|Q_{j_2}P_{k_2}\Upsilon(t)\|_{L^2}\\
\lesssim&2^{(k_3+k_2-|k_1|)/2-\beta j_1-\alpha j_2}\frac{\ep_1^2}{(1+t)^{\alpha-\beta}}.
\end{align*}
Let $\beta=0+$ and we can sum over $j_1\in\Z$ and $j_2>\mathcal L$ to get
\[
\sum_{j_1\in\Z\atop j_2>\mathcal L} \|\langle\nabla\rangle^8P_{k_3}W^{\mu\nu}_{j_1,k_1,j_2,k_2}(t)\|_{L^2}
\lesssim 2^{(k_3+k_2-|k_1|)/2}\frac{\ep_1^2}{(1+t)^{2\alpha-}}.
\]


{\bf Case 4.1.2:} $j_1\le j_2-\kappa/2$ and $j_2\le\mathcal L$.
Taking $\beta=\alpha-1$ in Case 4.1.1 we get
\[
\|\langle\nabla\rangle^8P_{k_3}W^{\mu\nu}_{j_1,k_1,j_2,k_2}(t)\|_{L^2}
\lesssim 2^{\sum k_i/2+(1-\alpha)j_1-\alpha j_2}\frac{\ep_1^2}{1+t}.
\]
Since $(1-\alpha)j_1-\alpha j_2\le(1-2\alpha)j_2-(1-\alpha)\kappa/2
\le(1-2\alpha)j_2-\kappa/2+\alpha(\mathcal L+5-j)$,
we sum over $j_1\le j_2-\kappa/2$ and $j_2\in\mathcal L$ and use $\alpha<1/2$ to get
\[
\sum_{j_1\le j_2-\kappa/2,j_2\le\mathcal L\atop\kappa/2\le\mathcal L+5-j}
\|\langle\nabla\rangle^8P_{k_3}W^{\mu\nu}_{j_1,k_1,j_2,k_2}(t)\|_{L^2}
\lesssim 2^{(\sum k_i-\kappa)/2-\alpha j}\frac{\ep_1^2}{(1+t)^\alpha}.
\]

{\bf Case 4.1.3:} $j_1>j_2-\kappa/2$ and $j_2\le\mathcal L$.
Similarly to Case 4.1.2 we put $L^2$ norm on the first factor and $L^\infty$ norm on the second to get
\[
\|\langle\nabla\rangle^8P_{k_3}W^{\mu\nu}_{j_1,k_1,j_2,k_2}(t)\|_{L^2}
\lesssim 2^{(k_3+3k_2^++2k_2)/2+(1-\alpha)j_2-\alpha j_1}\frac{\ep_1^2}{1+t}.
\]
We sum over $j_1>j_2-\kappa/2\ge j_2+j-\mathcal L-5$, $j_2\le\mathcal L$ and use $\alpha<1/2$ to get
\[
\sum_{j_1>j_2-\kappa/2,j_2\le\mathcal L\atop\kappa/2\le\mathcal L+5-j}
\|\langle\nabla\rangle^8P_{k_3}W^{\mu\nu}_{j_1,k_1,j_2,k_2}(t)\|_{L^2}
\lesssim 2^{(k_3+3k_2^++2k_2)/2-\alpha j}\frac{\ep_1^2}{(1+t)^\alpha}.
\]

Combining the 3 cases above and summing over $-O(\mathcal L)\le k_i\le 3\alpha\mathcal L/N+O(1)$ we get
\[
\sum_{\text{Case 4.1}}
\|\langle\nabla\rangle^8P_{k_3}W^{\mu\nu}_{k_1,k_2}(t)\|_{L^2}
\lesssim \frac{\ep_1^2}{(1+t)^{(2\alpha-3\alpha/N)-}}
+\frac{2^{-\alpha j}\mathcal L\ep_1^2}{(1+t)^{\alpha-9\alpha/N}}.
\]
Then we sum over $0\le j<4(\mathcal L+5)/(3-\alpha)$ and use $N\ge\max(9/(2-2\alpha),30)$ to get
\begin{align*}
&\left\| \sum_{\text{Case 4.1}}
2^{\alpha j}\|Q_j\langle\nabla\rangle^8P_{k_3}W^{\mu\nu}_{k_1,k_2}(t)\|_{L^2} \right\|_{\ell^2_{j<4(\mathcal L+5)/(3-\alpha)}}
\lesssim \frac{\ep_1^2}{(1+t)^{\frac{2\alpha^2(1-\alpha)}{3(3-\alpha)}-}}.
\end{align*}

{\bf Case 4.2:} $j>-\kappa/2+\mathcal L+5$. In this case $t<2^{\kappa/2+j-5}$. Since $j>-\kappa/2+\mathcal L>-j/N+\mathcal L$,
we have $\mathcal L<(1+1/N)j$.

We decompose
\begin{equation}\label{W2k-AB}
\begin{aligned}
W^{\mu\nu}_{k_1,k_2}&=A^{\mu\nu}_{k_1,k_2}+B^{\mu\nu}_{k_1,k_2},\\
A^{\mu\nu}_{k_1,k_2}&=W_{\mu\nu}[P_{[k_1-1,k_1+1]}Q_{\ge j-4}P_{k_1}\Upsilon_\mu,P_{k_2}\Upsilon_\nu],\\
B^{\mu\nu}_{k_1,k_2}&=W_{\mu\nu}[P_{[k_1-1,k_1+1]}Q_{\le j-5}P_{k_1}\Upsilon_\mu,P_{k_2}\Upsilon_\nu].
\end{aligned}
\end{equation}

For $A$ we have, by (\ref{WL2}) and unitarity of $e^{it\Lambda}$,
\[
\|\langle\nabla\rangle^8A^{\mu\nu}_{k_1,k_2}\|_{L^2}
\lesssim 2^{(8.5-N)k_2^++k_2+k_1}\|Q_{\ge j-4}P_{k_1}\Upsilon\|_{L^2}\|U\|_{H^{N-1/2}}.
\]
We sum over $j\ge0$, $k_1$, $k_2\in\Z$ and use $N\ge10$ and Lemma \ref{Z-bound} (ii) to get
\[
\sum_{k_1,k_2\in\Z} \|2^{\alpha j}\|Q_j\langle\nabla\rangle^8A^{\mu\nu}_{k_1,k_2}\|_{L^2}\|_{\ell^2_{j\ge0}}
\lesssim \|\Upsilon\|_Z\|U\|_{H^{N-1/2}}\lesssim \ep_1^2.
\]

For $B$ we have
\begin{equation}\label{Bk2-K}
\begin{aligned}
\langle\nabla\rangle^8P_{k_3}B^{\mu\nu}_{k_1,k_2}(x,t)
&=\int G(x,y,z,t)Q_{\le j-5}P_{k_1}\Upsilon_\mu(y,t)P_{k_2}\Upsilon_\nu(z,t)dydz,\\
G(x,y,z,t)&=\iint e^{i\phi_{\mu\nu}(\xi_1,\xi_2)}\langle\xi_1+\xi_2\rangle^8m^{\mu\nu}_{k_1,k_2;k_3}(\xi_1,\xi_2)d\xi_1d\xi_2,\\
\phi_{\mu\nu}(\xi_1,\xi_2)&=(x-y)\cdot\xi_1+(x-z)\cdot\xi_2+t\Phi_{\mu\nu}(\xi_1,\xi_2),\\
m^{\mu\nu}_{k_1,k_2;k_3}(\xi_1,\xi_2)&=\varphi_{[k_1-1,k_1+1]}(\xi_1)\varphi_{[k_2-1,k_2+1]}(\xi_2)\frac{m_{\mu\nu}(\xi_1,\xi_2)}{i\Phi_{\mu\nu}(\xi_1,\xi_2)}\varphi_{k_3}(\xi_1+\xi_2).
\end{aligned}
\end{equation}

Since $m_{\mu\nu}$ are linear combinations of multipliers in the set (\ref{m-prod}), for $L\ge0$ we have
\begin{equation}\label{m-CL}
|\nabla^L m_{\mu\nu}(\xi_1,\xi_2)|\lesssim_L (|\xi_1|+|\xi_2|)^{3/2}\min(|\xi_1|,|\xi_2|,|\xi_1+\xi_2|)^{-L}.
\end{equation}
By (\ref{m-CL}) and Lemma \ref{ps-1/F-Soo} (ii) then
\begin{equation}\label{mk3-CL}
|\nabla^Lm^{\mu\nu}_{k_1,k_2;k_3}(\xi_1,\xi_2)|\lesssim_L 2^{3\max k_i/2-(L+1/2)\min k_i}\lesssim 2^{3k_2/2-(L+1/2)\kappa}.
\end{equation}
Using $\kappa>-\frac{1+\alpha}{2}\max(j,\mathcal L)>-(\frac{1+\alpha}{2}+\frac{1}{N})j>-\beta j$, where $\beta=\frac{3+\alpha}{4}$, we have
\[
\left\| \langle\xi_1+\xi_2\rangle^8m^{\mu\nu}_{k_1,k_2;k_3}(\xi_1,\xi_2) \right\|_{W^{L,1}}\lesssim_L 2^{8k_2^++2k_1+3k_2/2+2k_3+\beta Lj+j/2}.
\]
Since $|\nabla\Lambda(\xi)|=|\xi|^{-1/2}/2$, we have $|t\nabla\Phi_{\mu\nu}(\xi_1,\xi_2)|<2\cdot2^{-\kappa/2}t<2^{j-3}$.
Also, for $L\ge2$ we have $|t\nabla^L\Phi_{\mu\nu}(\xi_1,\xi_2)|\lesssim_L 2^{(1/2-L)\kappa}t\lesssim 2^{j+(1-L)\beta j}$.
Hence by Lemma \ref{int-part-cor} (with $A=2^{8k_2^++2k_1+3k_2/2+2k_3+j/2}$, $r=2^{j-2}$, $\ep\approx2^{\beta j}$, $d=4$ and $L$ depending on $N$ and $\beta=\frac{3+\alpha}4$),
\begin{equation}\label{K2-L1}
\|1_{|x-y|>2^{j-2}}G(x,y,z,t)\|_{L^1_{y,z}}\lesssim 2^{8k_2^++2k_1+3k_2/2+2k_3-2j}.
\end{equation}
When $\varphi_j(x)\varphi_{\le j-5}(y)>0$, we have $|x-y|>2^{j-2}$, so
\[
\|\varphi_j(x)\varphi_{\le j-5}(y)G(x,y,z,t)\|_{L^1_{y,z}}
\lesssim\text{right-hand side of (\ref{K2-L1})}.
\]
Combining this with Bernstein's inequality $\|P_k\Upsilon\|_{L^\infty}\lesssim 2^k\|P_k\Upsilon\|_{L^2}=2^k\|P_kU\|_{L^2}$ we get
\begin{align*}
\|Q_j\langle\nabla\rangle^8P_{k_3}B^{\mu\nu}_{k_1,k_2}\|_{L^2}
&\lesssim 2^j\|Q_j\langle\nabla\rangle^8P_{k_3}B^{\mu\nu}_{k_1,k_2}\|_{L^\infty}\\
&\lesssim 2^{8k_2^++3k_1+5k_2/2+2k_3-j}\|P_{k_1}U\|_{L^2}\|P_{k_2}U\|_{L^2}.
\end{align*}
We sum over $k_i$, $j\in\Z$ and use $N\ge14$ to get
\[
\left\| \sum_{\text{Case 4.2}}2^{\alpha j}\|Q_j\langle\nabla\rangle^8P_{k_3}B^{\mu\nu}_{k_1,k_2}(t)\|_{L^2} \right\|_{\ell^2_{j>N\mathcal L/(N+1)}}
\lesssim \ep_1^2.
\]

Combining Case 1 through Case 4 above shows the claim.
\end{proof}

\subsection{Bounding the cubic bulk terms $H_{\mu\nu,3}$}
Recall (\ref{N23}) and (\ref{H2-FN-NF}).
Decompose $H_{\mu\nu}=H_{\mu\nu,3}+H_{\mu\nu,4}$, where for $j\in\{2,3\}$,
\begin{equation}\label{H34}
H_{\mu\nu,j+1}(t)=W_{\mu\nu}[\Upsilon_\mu(t),e^{it\nu\Lambda}(N_j)_\nu(t)]
+W_{\mu\nu}[e^{it\mu\Lambda}(N_j)_\mu(t),\Upsilon_\nu(t)].
\end{equation}
The $Z$ norm bounds of $H_{\mu\nu,3}$ and $H_{\mu\nu,4}$ are shown in Proposition \ref{H3-Z} and Proposition \ref{H4-Z} respectively.

\begin{proposition}\label{H3-Z}
Assume $N \ge \max(33/(\alpha - \alpha^2),8/\alpha^2) \ge 132$ and (\ref{growthZ}). Then
\[
\int_0^t \|H_{\mu\nu,3}(s)\|_Zds\lesssim\ep_1^3.
\]
\end{proposition}
\begin{proof}
Decompose $H_{\mu\nu,3}$ into linear combinations of the terms $H_{\mu\nu\rho}$, where
\begin{equation}\label{H3k}
\begin{aligned}
H_{\mu\nu\rho}&=\sum_{l\in\Z,\sigma=\pm} H^{\mu\nu\rho\sigma}_l,\\
H^{\mu\nu\rho\sigma}_l[\Upsilon_\mu,\Upsilon_\nu,\Upsilon_\rho](t)
&=W_{\mu\sigma}[\Upsilon_\mu,P_le^{it\sigma\Lambda}N_{\nu\rho,2}^\sigma[\Upsilon_\nu,\Upsilon_\rho]]\\
&+W_{\sigma\rho}[P_le^{it\sigma\Lambda}N_{\mu\nu,2}^\sigma[\Upsilon_\mu,\Upsilon_\nu],\Upsilon_\rho]],\\
H^{\mu\nu\rho\sigma}_l&=\sum_{k_1,k_2,k_3\in\Z} H^{\mu\nu\rho\sigma}_{k_1,k_2,k_3,l},\\
H^{\mu\nu\rho\sigma}_{k_1,k_2,k_3,l}
&=H^{\mu\nu\rho\sigma}_l[P_{k_1}\Upsilon_\mu,P_{k_2}\Upsilon_\nu,P_{k_3}\Upsilon_\rho].
\end{aligned}
\end{equation}
We assume by symmetry $k_1\le k_2\le k_3$. From (\ref{mk3-Soo}) and $S^\infty\subset L^\infty$ it follows in the same way as (\ref{WL2}) and (\ref{WL22}) that
\begin{align}
\label{H3k-L2}
\|\langle\nabla\rangle^8H^{\mu\nu\rho}_{k_1,k_2,k_3}\|_{L^2}
&\lesssim 2^{k_1+k_2+8k_3^++5k_3/2}\|P_{k_1}U\|_{L^2}\|P_{k_2}U\|_{L^2}\|P_{k_3}U\|_{L^2},\\
\label{H3k-L2l}
\|\langle\nabla\rangle^8H^{\mu\nu\rho}_{k_1,k_2,k_3,l}\|_{L^2}
&\lesssim 2^{2l+8k_3^++2k_3+l/2}\|P_{k_1}U\|_{L^2}\|P_{k_2}U\|_{L^2}\|P_{k_3}U\|_{L^2},\\
\label{H3k-L2l2}
\|\langle\nabla\rangle^8H^{\mu\nu\rho}_{k_1,k_2,k_3,l}\|_{L^2}
&\lesssim 2^{k_1+k_2+8k_3^++2k_3+l/2}\|P_{k_1}U\|_{L^2}\|P_{k_2}U\|_{L^2}\|P_{k_3}U\|_{L^2}.
\end{align}

We distinguish several cases.

{\bf Case 1:} $k_3\ge(j+6\mathcal L/5)/N$. We sum (\ref{H3k-L2}) over $k_1,k_2\in\Z$, $k_3$ in this range, and use $(11-N)k_3\le-(N-11)(j+6\mathcal L/5)/N\le-(1+\alpha)j/2-1.1\mathcal L$ to get
\begin{align*}
\left\| \sum_{k_1,k_2\in\Z\atop k_3\ge(j+6\mathcal L/5)/N} 2^{\alpha j}\|Q_j\langle\nabla\rangle^8H^{\mu\nu\rho}_{k_1,k_2,k_3}(t)\|_{L^2} \right\|_{\ell^2_{j\ge0}}
&\lesssim \sum_{j\ge0} \frac{2^{(\alpha-(1+\alpha)/2)j}}{(1+t)^{1.1}}\|U(t)\|_{H^{N-1/2}}^3\\
&\lesssim \frac{\ep_1^3}{(1+t)^{1.1}}.
\end{align*}

{\bf Case 2:} $k_1\le-2\max(j,{\mathcal L})/(2+\alpha)$. By unitarity of $e^{it\Lambda}$, Bernstein's inequality, Lemma \ref{Z-bound} (i) and (\ref{growthX1}),
\begin{equation}\label{PkU-L2-Z}
\|P_kU\|_{L^2}=\|P_k\Upsilon\|_{L^2}
\lesssim 2^{(\alpha-)k-8k^+}\|P_k\Upsilon\|_{W^{8,\frac2{1+\alpha}+}}
\lesssim 2^{(\alpha-)k-8k^+}\ep_1.
\end{equation}
By Lemma \ref{paraprod}, (\ref{mk3-Soo}), Lemma \ref{dispersiveZ} (i) and (\ref{PkU-L2-Z}) we get
\begin{align*}
\|\langle\nabla\rangle^8P_{k_4}H^{\mu\nu\rho\sigma}_{k_1,k_2,k_3,l}(t)\|_{L^2}
&\lesssim 2^{8k_3^++(k_4+l+3k_3)/2}\|P_{k_1}U(t)\|_{L^\infty}\|P_{k_2}U(t)\|_{L^\infty}\|P_{k_3}U(t)\|_{L^2}\\
&\lesssim 2^{(8.5-N)k_3^++(k_4+l-|k_2|+3k_3)/2+(1+\alpha-)k_1}\frac{\ep_1^3}{(1+t)^{\alpha-}}.
\end{align*}
We sum over $k_4$, $l\le k_3+O(1)$, $k_2$, $k_3\in\Z$, $k_1$ in the range above and $j\ge0$, and use $N\ge12$ to get
\begin{align*}
&\left\| \sum_{k_1\le-2\max(j,{\mathcal L})/(2+\alpha)} \sum_{k_2,k_3\in\Z}
\sum_{l\le k_3+O(1)} 2^{\alpha j}\|Q_j\langle\nabla\rangle^8H^{\mu\nu\rho\sigma}_{k_1,k_2,k_3,l}(t)\|_{L^2} \right\|_{\ell^2_{j\ge0}}\\
\lesssim &\left\| \sum_{k_1\le-2\max(j,{\mathcal L})/(2+\alpha)} 2^{\alpha j+(1+\alpha-)k_1} \right\|_{\ell^2_{j\ge0}}\frac{\ep_1^3}{(1+t)^{\alpha-}}\\
\lesssim &\sum_{j\ge0} 2^{\alpha j-(2+2\alpha-)\max(j,\mathcal L)/(2+\alpha)}\frac{\ep_1^3}{(1+t)^{\alpha-}}
\lesssim \frac{\ep_1^3}{(1+t)^{(2+2\alpha-)/(2+\alpha)}}.
\end{align*}

{\bf Case 3:} $l\le-4\max(j,\mathcal L)/5$. We sum (\ref{H3k-L2l}) over $k_1$, $k_2$, $k_3\in\Z$ and $l$ in that range and use (\ref{PkU-L2-Z}) and $N\ge11$ to get
\begin{align*}
&\left\| \sum_{l\le-4\max(j,\mathcal L)/5}\sum_{k_1,k_2,k_3\in\Z}
2^{\alpha j}\|Q_j\langle\nabla\rangle^8H^{\mu\nu\rho\sigma}_{k_1,k_2,k_3,l}(t)\|_{L^2} \right\|_{\ell^2_{j\ge0}}\\
\lesssim&\sum_{j\ge0}\sum_{k_1,k_2,k_3\in\Z} 2^{\alpha j-2\max(j,\mathcal L)-(\alpha-)(|k_1|+|k_2|)-|k_3|}\ep_1^3
\lesssim \frac{\ep_1^3}{(1+t)^{2-\alpha}}.
\end{align*}

{\bf Case 4:} $k_4\le-2\max(j,{\mathcal L})/(2+\alpha)$.
By the Bernstein inequality, Lemma \ref{paraprod}, (\ref{mk3-Soo}), Lemma \ref{dispersiveZ} (i) and (\ref{PkU-L2-Z}) we get
\begin{align*}
\|\langle\nabla\rangle^8P_{k_4}H^{\mu\nu\rho\sigma}_{k_1,k_2,k_3,l}(t)\|_{L^2}
&\lesssim 2^{k_4}\|P_{k_4}H^{\mu\nu\rho\sigma}_{k_1,k_2,k_3,l}(t)\|_{L^1}\\
&\lesssim 2^{(3k_4+l+3k_3)/2}\|P_{k_1}U(t)\|_{L^\infty}\|P_{k_2}U(t)\|_{L^2}\|P_{k_3}U(t)\|_{L^2}\\
&\lesssim 2^{(3k_4+l+3k_3-|k_1|)/2+(\alpha-)(k_2+k_3)-8(k_2^++k_3^+)}\frac{\ep_1^3}{(1+t)^{\alpha-}}.
\end{align*}
We sum over $l\le k_3+O(1)$, $k_1$, $k_2$, $k_3\in\Z$ and $k_4$ in that range to get
\begin{align*}
&\left\| \sum_{k_4\le-2\max(j,{\mathcal L})/(2+\alpha)} \sum_{k_1,k_2,k_3\in\Z} \sum_{l\le k_3+O(1)} 2^{\alpha j}\|Q_j\langle\nabla\rangle^8P_{k_4}H_{k_1,k_2,k_3,l}^{\mu\nu\rho\sigma}(t)\|_{L^2} \right\|_{\ell_{j\ge0}^2}\\
\lesssim &\sum_{j\ge0} \sum_{k_1,k_2,k_3\in\Z} 2^{\alpha j-3\max(j,\mathcal L)/(2+\alpha)-|k_1|/2-(\alpha-)(|k_2|+|k_3|)}\frac{\ep_1^3}{(1+t)^{\alpha-}}
\lesssim \frac{\ep_1^3}{(1+t)^{\frac{3}{2+\alpha}-}}.
\end{align*}

{\bf Case 5:} $-\alpha'\max(j,{\mathcal L})<k_i$ ,$l\le(j+6\mathcal L/5)/N+O(1)$, $1\le i\le4$, where $\alpha'=\max(2/(2+\alpha),4/5)\in(0,1)$.

{\bf Case 5.1:} $j\le-\min k_i/2+\mathcal L+5$. In this case either $l$, $\min k_i>-\mathcal L$ or $l$, $\min k_i>-j\ge\min k_i/2-\mathcal L-5$.
In the latter case $l,\min k_i>-2(\mathcal L+5)$, which thus holds either way. Then $j<2(\mathcal L+5)$, and $k_i$, $l\le3.2\mathcal L/N+O(1)$.

We decompose
\begin{align*}
H^{\mu\nu\rho\sigma}_{k_1,k_2,k_3,l}&=A^{\mu\nu\rho\sigma}_{k_1,k_2,k_3,l}
+B^{\mu\nu\rho\sigma}_{k_1,k_2,k_3,l}\\
A^{\mu\nu\rho\sigma}_{k_1,k_2,k_3,l}&=H_{\mu\nu\rho\sigma}[P_{k_1}\Upsilon_\mu,P_{k_2}\Upsilon_\nu,P_{[k_3-1,k_3+1]}Q_{>\beta j}P_{k_3}\Upsilon_\rho],\\
B^{\mu\nu\rho\sigma}_{k_1,k_2,k_3,l}&=H_{\mu\nu\rho\sigma}[P_{k_1}\Upsilon_\mu,P_{k_2}\Upsilon_\nu,P_{[k_3-1,k_3+1]}Q_{\le\beta j}P_{k_3}\Upsilon_\rho].
\end{align*}
Note the {\it absence} of spatial cutoffs in the first two arguments,
which enables us to use the $L^2L^\infty$ norm assumption in (\ref{growthZ}).

To bound $A$, we use Lemma \ref{paraprod}, (\ref{mk3-Soo}), unitarity of $e^{it\Lambda}$, Lemma \ref{Z-bound} (ii) and $k_i$, $l\le3.2\mathcal L/N+O(1)$ to get, for $\beta:=(1-\alpha)/2$,
\begin{align*}
\|\langle\nabla\rangle^8P_{k_4}A^{\mu\nu\rho\sigma}_{k_1,k_2,k_3,l}(t)\|_{L^2}
&\lesssim 2^{8k_3^++\frac{k_4+l+3k_3}{2}}\|P_{k_1}U(t)\|_{L^\infty}\|P_{k_2}U(t)\|_{L^\infty}\|Q_{>\beta j}P_{k_3}\Upsilon(t)\|_{L^2}\\
&\lesssim 2^{\frac{k_4^-}{2}-\frac{k_4^++|l|+|k_3|}{6}-\alpha\beta j}(1+t)^{\frac{48N}{5}}\|P_{k_1}U(t)\|_{L^\infty}\|P_{k_2}U(t)\|_{L^\infty}\ep_1.
\end{align*}
We integrate over $t\ge t_j=(2^{j+\min k_i/2-5}-1)^+$ and use the $L^2L^\infty$ norm assumption in (\ref{growthZ}) to get
\begin{align*}
\int_{t_j}^t 2^{\alpha j}\|Q_j\langle\nabla\rangle^8P_{k_4}A^{\mu\nu\rho\sigma}_{k_1,k_2,k_3,l}(s)\|_{L^2}ds
&\lesssim (1+t_j)^{\frac{10}{N}-\alpha}2^{(\alpha-\alpha\beta)j-\frac{k_4^++|l|+|k_3|}{6}-\frac{|k_1|+|k_2|-k_4^-}{2}}\ep_1^3\\
&\lesssim 2^{(\frac{10}{N}-\alpha\beta)j}2^{-\frac{k_4^++|l|+|k_3|}{6}-\frac{|k_1|+|k_2|-k_4^-+(\alpha-10/N)\min k_i}{2}}\ep_1^3.
\end{align*}
Using $\alpha\beta-10/N$ and $\alpha-10/N\in(0,1)$ we sum over $k_i$, $l\le3.2\mathcal L/N+O(1)$ and $j\ge0$ to get
\begin{align*}
&\int_0^t \left\| \sum_{\text{Case 5.1}}
2^{\alpha j}\|Q_j\langle\nabla\rangle^8P_{k_4}A^{\mu\nu\rho\sigma}_{k_1,k_2,k_3,l}(s)\|_{L^2} \right\|_{\ell_{j\ge0}^2}ds\\
\lesssim&\sum_{j\ge0} \sum_{k_i,l\le3.2\mathcal L/N+O(1)} \int_{t_j}^t 2^{\alpha j}\|Q_j\langle\nabla\rangle^8P_{k_4}A^{\mu\nu\rho\sigma}_{k_1,k_2,k_3,l}(s)\|_{L^2}ds
\lesssim \sum_{j\ge0} 2^{(\frac{10}{N}-\alpha\beta)j}\ep_1^3\lesssim \ep_1^3.
\end{align*}

To bound $B$, we put $L^\infty$ norms on the first and third factors and use (\ref{dispersive2}) to get
\begin{align*}
\|\langle\nabla\rangle^8P_{k_4}B^{\mu\nu\rho\sigma}_{k_1,k_2,k_3,l}(t)\|_{L^2}
&\lesssim 2^{8k_3^++(k_4+l+3k_3)/2}\|P_{k_1}U(t)\|_{L^\infty}\|P_{k_2}U(t)\|_{L^2}\\
&\times\|e^{-it\rho\Lambda}P_{[k_3-1,k_3+1]}Q_{\le\beta j}P_{k_3}\Upsilon_\rho(t)\|_{L^\infty}\\
&\lesssim 2^{10k_3^++(k_4+l+3k_3)/2}(1+t)^{-1}\|P_{k_1}U(t)\|_{L^\infty}\|P_{k_2}U(t)\|_{L^2}\\
&\times\|Q_{\le\beta j}P_{k_3}\Upsilon(t)\|_{L^1}\\
&\lesssim 2^{-\frac{|k_2|+|k_4|}{2}-\frac{|l|+|k_3|}{8}+(1-\alpha)\beta j}(1+t)^{\frac{84}{5N}-1}\|P_{k_1}U(t)\|_{L^\infty}\ep_1^2.
\end{align*}
We integrate over $t\ge t_j$ and use the $L^2L^\infty$ norm assumption in (\ref{growthZ}) to get
\begin{align*}
&\int_{t_j}^t 2^{\alpha j}\|Q_j\langle\nabla\rangle^8P_{k_4}A^{\mu\nu\rho\sigma}_{k_1,k_2,k_3,l}(s)\|_{L^2}ds\\
\lesssim&(1+t_j)^{\frac{84}{5N}-\frac{1+\alpha}{2}}2^{(\alpha+(1-\alpha)\beta)j-\frac{|l|+|k_3|}{8}-\frac{|k_1|+|k_2|+|k_4|}{2}}\ep_1^3\\
\lesssim&2^{(\frac{84}{5N}-(1-\alpha)(\frac{1}{2}-\beta))j}2^{-\frac{|l|+|k_3|}{8}-\frac{|k_1|+|k_2|+|k_4|+((1+\alpha)/2-16.8/N)\min k_i}{2}}\ep_1^3.
\end{align*}
Since $N\ge33/(\alpha-\alpha^2)$, $(1+\alpha)/2-16.8/N$ and $(1-\alpha)(1/2-\beta)-16.8/N\in(0,1)$, so we sum over $k_i$, $l\le3.2\mathcal L/N+O(1)$ and $j\ge0$ to get similar bounds to $A$. Combining the bounds for $A$ and $B$ we conclude Case 5.1.

{\bf Case 5.2:} $j>-\min k_i/2+\mathcal L+5$. In this case $j>-(j+2\mathcal L)/(2N)+\mathcal L$, which implies $j>(1-1.5/N)\mathcal L$ and $\mathcal L<(1+2/N)j<13j/12$. Then $k_i$, $l>-(1+2/N)\alpha'j$ and $k_i$, $l\le(j+6\mathcal L/5)/N+O(1)\le2.3j/N+O(1)$.

We further distinguish two cases.

{\bf Case 5.2.1:} $\alpha j>1.1\mathcal L$. We decompose
\begin{align*}
H^{\mu\nu\rho\sigma}_{k_1,k_2,k_3,l}
&=A^{\mu\nu\rho\sigma}_{k_1,k_2,k_3,l}+B^{\mu\nu\rho\sigma}_{k_1,k_2,k_3,l},\\
A^{\mu\nu\rho\sigma}_{k_1,k_2,k_3,l}
&=B^{\mu\nu\rho\sigma}_{\ge j-4,k_1,\ge j-4,k_2,\ge j-4,k_3,l},\\
B^{\mu\nu\rho\sigma}_{k_1,k_2,k_3,l}
&=\sum_{I_1,I_2,I_3\in\{\ge j-4,\le j-5\}\atop\exists I_i=``\le j-5"}
B^{\mu\nu\rho\sigma}_{I_1,k_1,I_2,k_2,I_3,k_3,l},\\
B^{\mu\nu\rho\sigma}_{I_1,k_1,I_2,k_2,I_3,k_3,l}
&=H^{\mu\nu\rho\sigma}_l[P_{[k_1-1,k_1+1]}Q_{I_1}P_{k_1}\Upsilon_\mu,
P_{[k_2-1,k_2+1]}Q_{I_2}P_{k_2}\Upsilon_\nu,P_{[k_3-1,k_3+1]}Q_{I_3}P_{k_3}\Upsilon_\rho].
\end{align*}

To bound $A$, by (\ref{H3k-L2l2}) and (\ref{growthZ}) we have
\begin{align*}
\|\langle\nabla\rangle^8A^{\mu\nu\rho\sigma}_{k_1,k_2,k_3,l}\|_{L^2}
&\lesssim 2^{k_1+k_2+8k_3^++2k_3+l/2}\prod_{i=1}^3\|Q_{\ge{j-4}}P_{k_i}\Upsilon\|_{L^2}\\
&\lesssim 2^{-|k_1|-|k_2|+(8.5-N)k_3^++2k_3+l/2-2\alpha j}\ep_1^3.
\end{align*}
We sum over $k_1$, $k_2\in\Z$, $l\le k_3+O(1)$, $k_3\in\Z$ and use $N\ge12$ to get
\[
\sum_{k_1,k_2,k_3,l} 2^{\alpha j}\|Q_j\langle\nabla\rangle^8A^{\mu\nu\rho\sigma}_{k_1,k_2,k_3,l}(t)\|_{L^2}\lesssim 2^{-\alpha j}\ep_1^3.
\]
Then we sum over $\alpha j>1.1\mathcal L$ to get
\[
\left\| \sum_{k_1,k_2,k_3,l} 2^{\alpha j}\|Q_j\langle\nabla\rangle^8A^{\mu\nu\rho\sigma}_{k_1,k_2,k_3,l}(t)\|_{L^2} \right\|_{\ell_{\alpha j>1.1\mathcal L}^2}
\lesssim \frac{\ep_1^3}{(1+t)^{1.1}}.
\]

To bound $B$, note that $\sum_{\mu,\nu,\rho,\sigma=\pm}
\langle\nabla\rangle^8P_{k_4}B^{\mu\nu\rho\sigma}_{I_1,k_1,I_2,k_2,I_3,k_3,l}(x,t)$ is a linear combination of
\begin{align*}
&\int G(x,y,z,w,t)Q_{I_1}P_{k_1}\Upsilon_\mu(y,t)
Q_{I_2}P_{k_2}\Upsilon_\nu(z,t)Q_{I_3}P_{k_3}\Upsilon_\rho(w,t)dydzdw,\\
&G(x,y,z,w,t)=\int e^{i\phi_{\mu\nu\rho}(\xi_1,\xi_2,\xi_3)}\langle\xi_1+\xi_2+\xi_3\rangle^8m^{\mu\nu\rho}_{k_1,k_2,k_3,l;k_4}(\xi_1,\xi_2,\xi_3)d\xi_1d\xi_2d\xi_3,\\
&\Phi_{\mu\nu\rho}(\xi_1,\xi_2,\xi_3)=\Lambda(\xi_1+\xi_2+\xi_3)-\mu\Lambda(\xi_1)-\nu\Lambda(\xi_2)-\rho\Lambda(\xi_3),\\
&\phi_{\mu\nu\rho}(\xi_1,\xi_2,\xi_3)=(x-y)\cdot\xi_1+(x-z)\cdot\xi_2+(x-w)\cdot\xi_3+t\Phi_{\mu\nu\rho}(\xi_1,\xi_2,\xi_3)
\end{align*}
and $m^{\mu\nu\rho}_{k_1,k_2,k_3,l;k_4}$ is a cutoff of $m_{\mu\nu\rho}(\xi_1,\xi_2,\xi_3)$ for $|\xi_j|\approx 2^{k_j}$, $j=1$, 2, 3, $|\xi_1+\xi_2+\xi_3|\approx 2^{k_4}$ and $|\xi_2+\xi_3|$ (or $|\xi_1+\xi_2|)\approx 2^l$.

From (\ref{m-CL}), (\ref{mk3-CL}), $\nabla^L(\varphi_k(\xi)\langle\xi\rangle^8)\lesssim_L 2^{8k^+-Lk}$ and Lemma \ref{ps-1/F-Soo} (ii) it follows that
\begin{equation}\label{mk3l-CL}
\left| \nabla^L \left( \langle\xi_1+\xi_2+\xi_3\rangle^8m^{\mu\nu\rho}_{k_1,k_2,k_3,l;k_4}(\xi_1,\xi_2,\xi_3) \right) \right|
\lesssim_L 2^{8k_3^++3k_3-(L+1/2)\min(k_i,l)}.
\end{equation}
Note that $k_i$, $l>-\beta j$, with $\beta=(1+2/N)\alpha'\in(0,1)$
(since $N>4/\alpha$ and $N>8$), so
\[
\|\langle\xi_1+\xi_2+\xi_3\rangle^8m^{\mu\nu\rho}_{k_1,k_2,k_3,l;k_4}(\xi_1,\xi_2,\xi_3)\|_{W^{L,1}}\lesssim_L 2^{8k_3^++5k_3+2k_2+2k_1+\beta Lj+j/2}.
\]
Again $|t\nabla\Phi_{\mu\nu\rho}(\xi_1,\xi_2,\xi_3)|<2^{j-3}$,
and $|t\nabla^L\Phi_{\mu\nu\rho}(\xi_1,\xi_2,\xi_3)|\lesssim 2^{j+\beta(1-L)j}$ for $L\ge2$, so by Lemma \ref{int-part-cor} (with $A=2^{8k_3^++5k_3+2k_2+2k_1+j/2}$, $r=2^{j-2}$, $\ep\approx2^{-\beta j}$, $d=6$, and $L=L(N,\alpha)$),
\[
\|1_{\max(|x-y|,|x-z|,|x-w|)>2^{j-2}}G(x,y,z,w,t)\|_{L^1_{y,z,w}}
\lesssim 2^{8k_3^++5k_3+2k_2+2k_1-3j}.
\]
Using $N\ge15$ and $N\ge6/(1-\alpha)$ we argue as in Case 4.2 of Proposition \ref{W-Z} to get
\begin{align*}
&\left\| \sum_{\text{Case 5.2.1}}
2^{\alpha j}\|Q_j\langle\nabla\rangle^8B^{\mu\nu\rho\sigma}_{k_1,k_2,k_3,l}(t)\|_{L^2} \right\|_{\ell^2_{j>(1-1.5/N)\mathcal L}}\\
\lesssim&(1+t)^{-(2-\alpha)(1-1.5/N)}\ep_1^3
\le(1+t)^{-(2-\alpha-3/N)}\ep_1^3\le(1+t)^{-(3-\alpha)/2}\ep_1^3.
\end{align*}

{\bf Case 5.2.2:} $\alpha j\le1.1\mathcal L$. We decompose
\begin{align*}
H^{\mu\nu\rho\sigma}_{k_1,k_2,k_3,l}&=A^{\mu\nu\rho\sigma}_{k_1,k_2,k_3,l}
+B^{\mu\nu\rho\sigma}_{k_1,k_2,k_3,l}\\
A^{\mu\nu\rho\sigma}_{k_1,k_2,k_3,l}&=H^{\mu\nu\rho\sigma}_l[P_{k_1}\Upsilon_\mu,P_{k_2}\Upsilon_\nu,P_{[k_3-1,k_3+1]}Q_{\ge j-4}P_{k_3}\Upsilon_\rho],\\
B^{\mu\nu\rho\sigma}_{k_1,k_2,k_3,l}&=H^{\mu\nu\rho\sigma}_l[P_{k_1}\Upsilon_\mu,P_{k_2}\Upsilon_\nu,P_{[k_3-1,k_3+1]}Q_{\le j-5}P_{k_3}\Upsilon_\rho]
\end{align*}
as in Case 5.1 in this proof.

To bound $A$, by Lemma \ref{paraprod}, (\ref{mk3-Soo}), unitarity of $e^{it\Lambda}$ and $k_i$, $l\le2.3j/N+O(1)$ we have
\begin{align*}
\|\langle\nabla\rangle^8P_{k_4}A^{\mu\nu\rho\sigma}_{k_1,k_2,k_3,l}(t)\|_{L^2}
&\lesssim 2^{8k_3^++(k_4+l+3k_3)/2}\|P_{[k_1-1,k_1+1]}U(t)\|_{L^\infty}\\
&\times\|P_{[k_2-1,k_2+1]}U(t)\|_{L^\infty}\|Q_{\ge j-4}P_{k_3}\Upsilon(t)\|_{L^2}\\
&\lesssim 2^{-(|k_4|+|l|+|k_3|)/6-(\alpha-7/N)j}\|P_{[k_1-1,k_1+1]}U(t)\|_{L^\infty}\\
&\times\|P_{[k_2-1,k_2+1]}U(t)\|_{L^\infty}\ep_1.
\end{align*}
We integrate over $t\ge t_j:=(2^{\alpha j/1.1}-1)^+$, use the $L^2L^\infty$ norm assumption in (\ref{growthZ}), sum over $k_i$, $l\le2.3\mathcal L/N+O(1)$ and $j\ge0$ and use $N \ge 8/\alpha^2$ to get
\begin{align*}
&\int_0^t \left\| \sum_{k_i,l\le2.3\mathcal L/N+O(1)} 2^{\alpha j}\|\langle\nabla\rangle^8P_{k_4}A^{\mu\nu\rho\sigma}_{k_1,k_2,k_3,l}(s)\|_{L^2} \right\|_{\ell_{\alpha j\le1.1\mathcal L}^2}ds\\
&\lesssim \sum_{j\ge0} \sum_{k_i,l\in\Z} \int_{t_j}^t 2^{\alpha j}\|\langle\nabla\rangle^8P_{k_4}A^{\mu\nu\rho\sigma}_{k_1,k_2,k_3,l}(s)\|_{L^2}ds\\
&\lesssim \sum_{j\ge0} 2^{\alpha(10/N-\alpha)j/1.1}\ep_1^3\lesssim \ep_1^3.
\end{align*}
The bound for $B$ is similar to that in Case 5.2.1, and in fact easier,
because in the kernel $G(x,y,z,w,t)$ we already have $|x-w|>2^{j-2}$.

Combining Case 1 through Case 5 above (with integration in $t$ when necessary) shows the result.
\end{proof}

\subsection{Bounding the cubic bulk term $N_3^\circ$}
Recall from (\ref{Duhamel}) that
\[
\Upsilon_3(t)=\int_0^t e^{is\Lambda}N_3(s)ds.
\]
We now bound the $Z$ norm of $\Upsilon_3$ by expanding $N_3$ to one order higher.

From (\ref{B23}) it follows that
\begin{equation}\label{B34}
\begin{aligned}
B_3&=B_3^\circ+B_4,\\
B_3^\circ&=|\nabla|(h|\nabla|(h|\nabla|\phi))+\frac{1}{2}[\nabla\cdot(h\nabla\cdot(h|\nabla|\phi))+|\nabla|(h\nabla\cdot(h\nabla\phi))\\
&-|\nabla|(h\nabla h\cdot\nabla\phi)+\nabla\cdot(h(\nabla h)|\nabla|\phi)]
-|\nabla h|^2|\nabla|\phi\\
&=|\nabla|(h|\nabla|(h|\nabla|\phi))+\frac{1}{2}(\Delta(h^2|\nabla|\phi)+|\nabla|(h^2\Delta\phi))-|\nabla h|^2|\nabla|\phi
\end{aligned}
\end{equation}
and $B_4$ is an integral of quartic terms in $h$ and $\phi$.
Then from (\ref{N23}) it follows that
\begin{equation}\label{N34}
\begin{aligned}
N_3&=N_3^\circ+N_4,\\
N_3^\circ&=|\nabla|(h|\nabla|(h|\nabla|\phi))+\frac{1}{2}(\Delta(h^2|\nabla|\phi)+|\nabla|(h^2\Delta\phi))\\
&-i|\nabla|^{1/2}(|\nabla|\phi)(|\nabla|(h|\nabla|\phi)+h\Delta\phi),\\
N_4&=B_4+|\nabla h|^2(B_2+B_3)+\frac{i}{2}|\nabla|^{1/2}(2BB_3-B_3^2+B_2^2+|\nabla h|^2B^2).
\end{aligned}
\end{equation}
The $Z$ norm bounds for $N_3^\circ$ and $N_4$ are shown in Proposition \ref{N30-Z} and Proposition \ref{N4-Z} respectively.

\begin{proposition}\label{N30-Z}
Assume $N \ge \max(33/(\alpha - \alpha^2), 8/\alpha^2)$ and (\ref{growthZ}). Then
\[
\int_0^t \|e^{is\Lambda}N_3^\circ(s)\|_Zds\lesssim \ep_1^3.
\]
\end{proposition}
\begin{proof}
The multiplier of $N_3^\circ$ is a linear combination of the multipliers
\begin{align*}
|\xi_1+\xi_2+\xi_3||\xi_2+\xi_3|\sqrt{|\xi_3|},
|\xi_1+\xi_2+\xi_3|^2\sqrt{|\xi_3|}, |\xi_1+\xi_2+\xi_3||\xi_3|^{3/2},\\
\sqrt{|\xi_1+\xi_2+\xi_3||\xi_1||\xi_3|}\cdot|\xi_2+\xi_3|\text{ and }
\sqrt{|\xi_1+\xi_2+\xi_3||\xi_1||\xi_3|^3}.
\end{align*}
The first one and the fourth one satisfy the bounds (\ref{mk3-Soo}) (and its consequences (\ref{H3k-L2}), (\ref{H3k-L2l}) and (\ref{H3k-L2l2}))
and (\ref{mk3l-CL}), so their bounds can be shown in the same way as Proposition \ref{H3-Z}.

Now we turn to the remaining three multipliers. They differ from the previous two in that they do not contain a factor that is a power of the ``intermediate frequency" $|\xi_2 + \xi_3|$. Thus they are smooth near $\xi_2 + \xi_3 = 0$, and instead of (\ref{mk3-Soo}) they satisfy the bound
\[
\|m_{\mu\nu\rho}(\xi_1, \xi_2, \xi_3)\|_{S_{k_1,k_2,k_3;k}^\infty}
\lesssim 2^{k/2+2\max k_j}
\]
which still implies the bound (\ref{H3k-L2}). We don't need the other two.

Now we go over the proof of Proposition \ref{H3-Z} to show that it applies to the remaining three multipliers. Case 1 remains unchanged.
In Case 2, we omit dyadic decomposition according to $|\xi_2 + \xi_3|$,
and replace $l$ with $\max k_j$ in all the relevant bounds.
Case 3 is not needed. Case 4 can be adapted in the same way as Case 2.
For Case 5 we again omit dyadic decomposition in $|\xi_2 + \xi_3|$ and replace $l$ with $\max k_j$ in all the relevant bounds.
The only necessary change is in bounding $A$ in Case 5.2.1,
where instead of (\ref{H3k-L2l2}) we use (\ref{H3k-L2}) to get
\[
\|\langle\nabla\rangle^8A^{\mu\nu\rho\sigma}_{k_1,k_2,k_3,l}\|_{L^2}
\lesssim 2^{-|k_1|-|k_2|+(8.5-N)k_3^++5k_3/2-2\alpha j}\ep_1^3.
\]
Then we sum over $k_1$, $k_2$ and $k_3 \in \Z$ to obtain the same bound as before.
\end{proof}

\subsection{Bounding the quartic bulk term $N_4$}
To bound the $Z$ norm of
\[
\int_0^t e^{is\Lambda}N_4(s)ds
\]
we first need to relate the $Z$-norm of the profile to that of the solution, as done in Section 8 of \cite{GeMaSh2}.

\begin{lemma}\label{UZ-VZ}
Fix $\alpha\in(0,1)$. Then for $t \ge 0$,
\[
\|e^{it\Lambda}u\|_Z\lesssim t^\alpha\|u\|_{W^{8,4/(2+\alpha)}}+\|u\|_Z.
\]
\end{lemma}
\begin{proof}
By Lemma \ref{Z-bound} (iv) and the square function estimates (see Section I.6.4 of \cite{St}),
\begin{align*}
\|e^{it\Lambda}u\|_Z&\approx \|\|e^{it\Lambda}P_ku\|_Z\|_{\ell_k^2},\\
\|\|P_ku\|_{W^{8,4/(2+\alpha)}}\|_{\ell_k^2}
&\le\|\|\langle\nabla\rangle^8P_ku\|_{\ell_k^2}\|_{L^{4/(2+\alpha)}}
\approx \|u\|_{W^{8,4/(2+\alpha)}},\\
\|\|P_ku\|_Z\|_{\ell_k^2}&\approx \|u\|_Z.
\end{align*}
Hence it suffices to show the result with $P_ku$ in place of $u$.
Since $\|u\|_{W^{8,4/(2+\alpha)}}\approx\|\langle\nabla\rangle^8u\|_{L^{4/(2+\alpha)}}$ (see Theorem 2.3.3 of \cite{Tr2}),
after replacing $\langle\nabla\rangle^8u$ with $u$ it suffices to show
\[
\|(1+|x|)^\alpha e^{it\Lambda}P_ku\|_{L^2}
\lesssim t^\alpha \|P_ku\|_{L^{4/(2+\alpha)}}+\|(1+|x|)^\alpha P_ku\|_{L^2}.
\]
We distinguish two cases.

{\bf Case 1:} $2^{k/2}t \le 1$. Note that for any integer $L \ge 0$,
$\nabla^L(e^{it\Lambda})$ is a linear combination of
\[
e^{it\Lambda}\nabla^{n_1}(t\Lambda)\cdots\nabla^{n_l}(t\Lambda)
\]
where $1 \le l \le L$, $\sum_{j=1}^l n_j = 1$ and $n_j \ge 1$.
Since $|\nabla^n(t\Lambda)(\xi)| \lesssim_n t|\xi|^{1/2-n}$ for $n \ge 0$,
when $|\xi| \approx 2^k$ we have
\[
|\nabla^L(e^{it\Lambda}\varphi_k)(\xi)|
\lesssim_L \max_{l=1}^L t^l|\xi|^{l/2-L}
\lesssim_L |\xi|^{-L}
\]
so $e^{it\Lambda}$ is a Calderon--Zygmund operator whose norm is uniformly bounded in $t$, so it is uniformly bounded on weighted $L^2$ spaces with $A_2$ weights, and the second term on the right-hand side suffices.

{\bf Case 2:} $2^{k/2}t>1$. Note that
\[
\|(1+|x|)^\alpha e^{it\Lambda}u\|_{L^2}
\approx \|2^{\alpha j}\|Q_je^{it\Lambda}u\|_{L^2}\|_{\ell_{j\ge0}^2}.
\]
We divide the sum into several parts and estimate them separately.

{\bf Part 1:} $2^j\le 2^{2-k/2}t$. By unitarity of $e^{it\Lambda}$ and the Bernstein inequality,
\begin{align*}
\|2^{\alpha j}\|Q_je^{it\Lambda}u\|_{L^2}\|_{\ell_{2^j\le 2^{2-k/2}t}^2}
&\le\|2^{\alpha j}\|u\|_{L^2}\|_{\ell_{2^j\le 2^{2-k/2}t}^2}
\lesssim t^\alpha 2^{-\alpha k/2}\|u\|_{L^2}\\
&\lesssim t^\alpha \|u\|_{L^{4/(2+\alpha)}}.
\end{align*}

{\bf Part 2:} $2^j>2^{2-k/2}t$. We further decompose
\[
\|Q_je^{it\Lambda}u\|_{L^2}\le
\|Q_je^{it\Lambda}Q_{\ge j-3}u\|_{L^2}+\|Q_je^{it\Lambda}Q_{\le j-4}u\|_{L^2}.
\]

{\bf Part 2.1:} $Q_je^{it\Lambda}Q_{\ge j-3}u$.
By unitarity of $e^{it\Lambda}$,
\begin{align*}
\|2^{\alpha j}\|Q_je^{it\Lambda}Q_{\ge j-3}u\|_{L^2}\|_{\ell_{j\ge0}^2}
&\le\|2^{\alpha j}\|Q_{\ge j-3}u\|_{L^2}\|_{\ell_{j\ge0}^2}\\
&\lesssim \|\|2^{\alpha j}\|Q_{j'}u\|_{L^2}\|_{\ell_{j\le j'+3}^2}\|_{\ell_{j'\ge0}^2}\\
&\lesssim \|2^{\alpha j'}\|Q_{j'}u\|_{L^2}\|_{\ell_{j'\ge0}^2}\\
&\approx \|(1+|x|)^\alpha u\|_{L^2}.
\end{align*}

{\bf Part 2.2:} $Q_je^{it\Lambda}Q_{\le j-4}u$. We have
\[
Q_je^{it\Lambda}Q_{\le j-4}u(x)
=\varphi_j(x)\int K(y)\varphi_{\le j-4}(x-y)u(x-y)dy,
\]
where
\[
K(x)=\int e^{i(t|\xi|^{1/2}+x\cdot\xi)}\varphi_k(\xi)d\xi.
\]
By Lemma \ref{int-part-cor}, with $A\approx 2^{2k}$, $r=3\cdot2^{j-3}$,
$\ep\approx 2^k$, $L=1$ and $d=2$,
\[
\|K\|_{L^1(\R^2\backslash B(0,r))}\lesssim 2^{2k-j-3k}=2^{-j-k}.
\]
Since when $\varphi_j(x)\varphi_{\le j-4}(x-y)>0$ we have $|x|>2^{j-1}$ and $|y|<2^{j-3}$ so $|x-y|>r$, by Young's inequality and Bernstein's inequality we then have
\[
\|Q_je^{it\Lambda}Q_{\le j-4}u\|_{L^2}\lesssim 2^{-j-k}\|Q_{\le j-4}u\|_{L^2}
\le 2^{-j-k}\|u\|_{L^2}\lesssim 2^{-j-k+\alpha k/2}\|u\|_{L^{4/(2+\alpha)}}.
\]
Hence
\begin{align*}
\|2^{\alpha j}\|Q_je^{it\Lambda}Q_{\le j-4}u\|_{L^2}\|_{\ell_{2^j>2^{2-k/2}t}^2}
&\lesssim \|2^{\alpha j-j-k+\alpha k/2}\|_{\ell_{2^j>2^{2-k/2}t}^2}\|u\|_{L^{4/(2+\alpha)}}\\
&\lesssim 2^{(\alpha-1)(\log t-k/2)-k+\alpha k/2}\|u\|_{L^{4/(2+\alpha)}}\\
&=t^{\alpha-1}2^{-k/2}\|u\|_{L^{4/(2+\alpha)}}
<t^\alpha \|u\|_{L^{4/(2+\alpha)}},
\end{align*}
where we have used the assumption $2^{k/2}t>1$.
\end{proof}

Now we are ready to bound the $Z$ norm of $N_4$.
\begin{proposition}\label{N4-Z}
Assume $N \ge 14$ and (\ref{growthZ}). Then
\[
\int_0^t \|e^{is\Lambda}N_4(s)\|_Zds\lesssim \ep_1^4.
\]
\end{proposition}
\begin{proof}
From Lemma \ref{UZ-VZ} it follows that
\begin{equation}\label{N4-UZ-VZ}
\|e^{it\Lambda}N_4(t)\|_Z
\lesssim t^\alpha\|N_4(t)\|_{W^{8,4/(2+\alpha)}}+\|N_4(t)\|_Z.
\end{equation}
We bound the two terms on the right-hand side separately.
By Proposition \ref{BN-ZW},
\begin{equation}\label{N4-U4}
\begin{aligned}
\|N_4\|_{W^{8,4/(2+\alpha)}}&\lesssim \|U\|_{W^{2,\infty}}^2\|U\|_{W^{11,4/\alpha}}\|U\|_{W^{9,2}},\\
\|N_4\|_Z&\lesssim \|U\|_{W^{2,\infty}}^2\|U\|_{W^{11,\infty}}\|U\|_Z.
\end{aligned}
\end{equation}
By Sobolev embedding and Proposition \ref{U-tld=U}, $\|U\|_{W^{N-2,\infty}}\lesssim\ep_1$. Then by Sobolev interpolation and Lemma \ref{dispersiveZ} (ii) we have, for $\beta := 5/(N - 8) \in (0, 1)$,
\begin{equation}\label{U-Wkp-interpol}
\begin{aligned}
\|U(t)\|_{W^{11,\infty}}
&\lesssim \|U(t)\|_{W^{6,\infty}}^{1-\beta}\|U(t)\|_{W^{N-2,\infty}}^\beta
\lesssim (1+t)^{(-(1-\beta)\alpha)+}\ep_1,\\
\|U(t)\|_{W^{11,4/\alpha}}
&\lesssim \|U(t)\|_{H^{11}}^{\alpha/2}\|U(t)\|_{W^{11,\infty}}^{1-\alpha/2}
\lesssim (1+t)^{(-(1-\alpha/2)(1-\beta)\alpha)+}\ep_1.
\end{aligned}
\end{equation}
Also by Lemma \ref{UZ-VZ} and the embedding $H^9\subset W^{8,4/(2+\alpha)}$,
\begin{equation}\label{U-Z}
\|U(t)\|_Z\lesssim t^\alpha\|U(t)\|_{W^{8,4/(2+\alpha)}}+\|\Upsilon(t)\|_Z
\lesssim (1+t)^\alpha\ep_1.
\end{equation}
Combining (\ref{N4-U4}), (\ref{U-Wkp-interpol}), (\ref{U-Z}) and the $L^2L^\infty$ norm assumption in (\ref{growthZ}) we get the desired bound for the two terms on the right-hand side of (\ref{N4-UZ-VZ}).
\end{proof}

\subsection{Bounding the quartic bulk terms $H_{\mu\nu,4}$}
Recall from (\ref{H34}) that
\[
H_{\mu\nu,4}(t)=W_{\mu\nu}[e^{it\mu\Lambda}(N_3)_\mu(t),\Upsilon_\nu(t)]
+W_{\mu\nu}[\Upsilon_\mu(t),e^{it\nu\Lambda}(N_3)_\nu(t)].
\]
To bound $H_{\mu\nu,4}$, we need a weighted analog of Theorem C.1 of \cite{GeMaSh2}. Recall the notation there:
\[
\mathcal FB_{m(\xi,\eta)}[f,g](\xi)
=\int m(\xi,\eta)f(\eta)g(\xi-\eta)d\eta.
\]
\begin{lemma}\label{paraprod-Z}
Let $\alpha \in (0, 1)$ and $2\le p$, $q$, $r < \infty$ with $1/p+1/q=1/r$.

(i) There is $c>0$ such that if $m \in \mathcal B_0$ and $m$ is supported on $\{|\eta|<c|\xi|\}$, then
\begin{align*}
\|(1+|x|)^\alpha B_m[f,g]\|_{L^r}&\lesssim_{m,p,q} \|(1+|x|)^\alpha f\|_{L^p}\|g\|_{L^q}\\
\|(1+|x|)^\alpha B_m[f,g]\|_{L^r}&\lesssim_{m,p,q} \|f\|_{L^p}\|(1+|x|)^\alpha g\|_{L^q}.
\end{align*}

(ii) Let $m_{\mu\nu}$ be one of the multipliers in (\ref{m-prod}). Let $\Phi_{\mu\nu}=\sqrt{|\xi|}-\mu\sqrt{|\eta|}-\nu\sqrt{|\xi-\eta|}$. Then for any integer $k \ge 0$,
\begin{align*}
\|(1+|x|)^\alpha\nabla^kB_{m_{\mu\nu}/\Phi_{\mu\nu}}[f,g]\|_{L^r}
&\lesssim_{k,p,q} \|\nabla^{k+1}f\|_{L^p}\|(1+|x|)^\alpha g\|_{L^q}\\
&+\|(1+|x|)^\alpha f\|_{L^q}\|\nabla^{k+1}g\|_{L^p}.
\end{align*}
\end{lemma}
\begin{proof}
(i) We follow the proof of Theorem C.1 in \cite{GeMaSh2}. First we write
\begin{equation}\label{m-Taylor}
m(\xi,\eta)=\sum_{k=0}^M \frac{|\eta|^{k/2}}{|\xi|^{k/2}} m_k\left( \frac{\eta}{|\eta|},\frac{\xi}{|\xi|} \right)+m_{M+1},
\end{equation}
where $m_{M+1} \in \mathcal B_0$ is a remainder which is smooth enough.

We claim that $B_{m_{M+1}}$ is a bilinear Calderon--Zygmund operator
as defined in Section 2 of \cite{GrTo}. Indeed, the boundedness of $B_{M+1}: L^p \times L^q \to L^r$ is shown in Theorem C.1 (i) of \cite{GeMaSh2}.
To show the necessary bounds (and derivative bounds) for the kernel,
we argue as in the proof of Proposition 2 of Section VI.4 of \cite{St},
noting that only finitely many derivatives of the symbol is needed there.

Now we insert the weights. Since $\alpha \in (0, 1)$,
the weights 1 and $(1+|x|)^{s\alpha} \in A_s$ for $s \in \{p, q\}$
(see for example, the discussion after Definition 1.4 of \cite{Ku}). Then by Corollary 3.9 (i) of \cite{LOPTT} we get the bound for $B_{m_{M+1}}$.

For the first $M + 1$ terms, as in \cite{GeMaSh2} we can assume without loss of generality that $m_k=1$ and bound instead
\[
\sum_j |\nabla|^{-k/2}[P_{<j-C}|\nabla|^{k/2}fP_jg]
\]
for some constant $C$. Again since for $s \in \{p, q, r\}$ we have
$(1 + |x|)^{s\alpha} \in A_s$, by the square function estimates (see Section I.6.4 of \cite{St}) we get
\begin{align*}
&\left\| (1+|x|)^\alpha\sum_j |\nabla|^{-k/2}[P_{<j-C}|\nabla|^{k/2}fP_jg] \right\|_{L^r}\\
\approx&\left\| (1+|x|)^\alpha\|2^{-jk/2}P_{<j-C}|\nabla|^{k/2}fP_jg\|_{\ell_j^2} \right\|_{L^r},\\
&\|(1+|x|)^\alpha Mf\|_{L^p}\approx\|(1+|x|)^\alpha f\|_{L^p},\\
&\|(1+|x|)^\alpha\|P_jg\|_{\ell_j^2}\|_{L^q}\approx\|(1+|x|)^\alpha g\|_{L^q}.
\end{align*}
Then we can insert the weights in (C.1) of \cite{GeMaSh2} to get
\begin{align*}
&\left\| (1+|x|)^\alpha\sum_j |\nabla|^{-k/2}[P_{<j-C}|\nabla|^{k/2}fP_jg] \right\|_{L^r}\\
\lesssim&\min(\|(1+|x|)^\alpha Mf\|_{L^p}\|\|P_jg\|_{\ell_j^2}\|_{L^q},
\|Mf\|_{L^p}\|(1+|x|)^\alpha\|P_jg\|_{\ell_j^2}\|_{L^q})
\end{align*}
from which the bound in (i) follows.

(ii) We decompose the multiplier $m_{\mu\nu}/\Phi_{\mu\nu}$ as
\[
\frac{m_{\mu\nu}}{\Phi_{\mu\nu}}=m_{LH}+m_{HL}+\sum_{j=-\infty}^0 m_{HH,j},
\]
where $m_{LH}$ is supported on $\{|\eta|<c|\xi|\}$, $m_{HL}$ is supported on $\{|\xi-\eta|<c|\xi|\}$, and $m_{HH,j}$ is supported on $\{2^{j-1}<|\xi|/(|\xi-\eta|+|\eta|)<2^{j+1}\}$. Here $c$ is the constant in (i).

By the definition of $m_{LH}$, we know that the symbol $|\xi|^km_{LH}/|\xi-\eta|^{k+1} \in \mathcal B_0$, so by (i),
\[
\|(1+|x|)^\alpha\nabla^kB_{m_{LH}}[f,g]\|_{L^r}
\lesssim \|(1+|x|)^\alpha f\|_{L^q}\|\nabla^{k+1}g\|_{L^p}.
\]
Switching the roles of $f$ and $g$ we get a similar bound for $m_{HL}$.
For the remaining terms, by (\ref{m-prod}), for any integer $L \ge 0$,
$|\nabla^Lm_{HH,j}(\xi,\eta)|\lesssim |\xi|^{1/2-L}(|\xi-\eta|+|\eta|)$,
so $B_{|\xi|^km_{HH,j}/|\eta|^{k+1}}$ is a bilinear Calderon--Zygmund operator whose norm is $O_k(2^{(k+1/2)j})$. Then by Corollary 3.9 (i) of \cite{LOPTT}, $\|(1+|x|)^\alpha\nabla^kB_{m_{HH,j}}[f,g]\|_{L^r}$ can be bounded by $O_k(2^{(k+1/2)j}) \times$ the right-hand side. Since $\sum_{j\le0} O_k(2^{(k+1/2)j})=O_k(1)$, the desired bound for the remaining terms also follows.
\end{proof}

\begin{proposition}\label{H4-Z}
Assume $N \ge 17$ and (\ref{growthZ}). Then
\[
\int_0^t \|H_{\mu\nu,4}(s)\|_Zds\lesssim \ep_1^4.
\]
\end{proposition}
\begin{proof}
From Lemma \ref{UZ-VZ} it follows that
\begin{equation}\label{H4-UZ-VZ}
\|H_{\mu\nu,4}(t)\|_Z
\lesssim t^\alpha\|e^{-it\Lambda}H_{\mu\nu,4}(t)\|_{W^{8,4/(2+\alpha)}}
+\|e^{-it\Lambda}H_{\mu\nu,4}(t)\|_Z.
\end{equation}
We bound the two terms on the right-hand side separately.

In the first term, $e^{-it\Lambda}H_{\mu\nu,4}$ is a bilinear operator whose multiplier $m \in \mathcal B_1$, as defined in Appendix C of \cite{GeMaSh2}, acting on $N_3$ and $U$. After decomposing $m = m_1 + m_2$ where $m_1$, $m_2 \in \mathcal{\tilde B}_1$, by Theorem C.1 (ii) in \cite{GeMaSh2}, (\ref{U-Wkp-interpol}), (\ref{N3-Hs}) and $N \ge 13$ we get
\begin{equation}\label{H4-Wkp}
\begin{aligned}
\|e^{-it\Lambda}H_{\mu\nu,4}(t)\|_{W^{8,4/(2+\alpha)}}
&\lesssim \|N_3(t)\|_{W^{9,2}}\|U(t)\|_{W^{9,4/\alpha}}\\
&\lesssim \|N_3(t)\|_{H^9}(1+t)^{(-(1-\alpha/2)(1-\beta)\alpha)+}\ep_1\\
&\lesssim \|U(t)\|_{C_*^2}^2(1+t)^{(-(1-\alpha/2)(1-\beta)\alpha)+}\ep_1^2.
\end{aligned}
\end{equation}
The bound for the first term follows from the $L^2L^\infty$ norm assumption in (\ref{growthZ}).

For the second term, by Lemma \ref{paraprod-Z} (ii) we have, for $p \in (2, \infty)$ to be determined later and $q = 2p/(p - 2)$,
\begin{equation}\label{H4-UZ-NZ}
\|e^{-it\Lambda}H_{\mu\nu,4}(t)\|_Z
\lesssim \|N_3(t)\|_{W^{9,p}}\|(1+|x|)^\alpha U(t)\|_{L^q}
+\|(1+|x|)^\alpha N_3(t)\|_{L^q}\|U(t)\|_{W^{9,p}}.
\end{equation}
By Proposition \ref{BN-ZW}, (\ref{U-Wkp-interpol}) and the Gagliardo--Nirenberg interpolation (see \cite{Ni}), with $\beta := 5/(N - 8)$ and $\beta' := (4 - 2/p)/(N - 8)$, if $p > (N - 7)/2$ then
\begin{equation}\label{N3-W9p}
\begin{aligned}
\|N_3(t)\|_{W^{9,p}}
&\lesssim \|U(t)\|_{W^{2,\infty}}\|U(t)\|_{W^{11,\infty}}\|U(t)\|_{W^{10,p}}\\
&\lesssim \|U(t)\|_{W^{6,\infty}}^{3-\beta-\beta'}\|U(t)\|_{W^{N-2,\infty}}^\beta\|U(t)\|_{W^{N-1,2}}^{\beta'}\\
&\lesssim \|U(t)\|_{W^{6,\infty}}^{3-\beta-\beta'}\ep_1^{\beta+\beta'}.
\end{aligned}
\end{equation}
To bound $\|(1+|x|)^\alpha U\|_{L^q}$, first note that for any $x\in\R^2$,
we can find a smooth function $\chi$ supported in $B(x, 2)$ and equal to 1 on $B(x, 1)$. Then by Sobolev embedding,
\begin{equation}\label{Lq-H2}
\|U\|_{L^q(B(x,1))}\le\|\chi U\|_{L^q}\lesssim \|\chi U\|_{H^2}\lesssim \|U\|_{H^2(B(x,2))}.
\end{equation}
We now insert the weight by covering $\R^2$ with balls $\{B(x_n, 1)\}$ of radius 1 such that each point in $\R^2$ is covered by $O(1)$ balls.
Then each point in $\R^2$ is covered by $O(1)$ balls in $\{B(x_n, 2)\}$, so
\begin{equation}\label{insert-weight}
\begin{aligned}
\|(1+|x|)^\alpha U\|_{L^q}
&\approx\|(1+|x_n|)^\alpha\|U\|_{L^q(B(x_n,1))}\|_{\ell_n^q}\\
&\le\|(1+|x_n|)^\alpha\|U\|_{L^q(B(x_n,1))}\|_{\ell_n^2},\quad(\text{since }q>2)\\
\|(1+|x_n|)^\alpha\|U\|_{H^2(B(x_n,2))}\|_{\ell_n^2}
&\approx\|(1+|x_n|)^\alpha\|(U,\nabla U,\nabla^2U)\|_{L^2(B(x_n,2))}\|_{\ell_n^2}\\
&\approx\|(1+|x|)^\alpha\langle\nabla\rangle^2U\|_{L^2}.
\end{aligned}
\end{equation}
Hence taking an $\ell_n^2$ sum of the local inequality (\ref{Lq-H2}) and using (\ref{U-Z}) we get
\begin{equation}\label{U-Zq}
\|(1+|x|)^\alpha U\|_{L^q}\lesssim \|(1+|x|)^\alpha\langle\nabla\rangle^2U\|_{L^2}\lesssim \|U\|_Z\lesssim (1+t)^\alpha\ep_1.
\end{equation}
By (\ref{U-Zq}) and Proposition \ref{BN-ZW},
\begin{equation}\label{N3-Zq}
\|(1+|x|)^\alpha N_3\|_{L^q}
\lesssim \|(1+|x|)^\alpha\langle\nabla\rangle^2N_3\|_{L^2}
\lesssim \|U\|_{W^{5,\infty}}^2(1+t)^\alpha\ep_1.
\end{equation}
By the Gagliardo--Nirenberg interpolation again, if $p > 2(N - 7)/3$ then,
for $\beta'' = (3 - 2/p)/(N - 8) \in (0, 1)$,
\begin{equation}\label{U-W9p}
\|U\|_{W^{9,p}}
\lesssim \|U\|_{W^{6,\infty}}^{1-\beta''}\|U\|_{W^{N-1,2}}^{\beta''}
\lesssim \|U\|_{W^{6,\infty}}^{1-\beta''}\ep_1^{\beta''}.
\end{equation}
Now we put (\ref{N3-W9p}), (\ref{U-Zq}), (\ref{N3-Zq}) and (\ref{U-W9p})
into (\ref{H4-UZ-NZ}). Note that (\ref{N3-Zq}) already has the factor $\|U\|_{W^{6,\infty}}^2$. For (\ref{N3-W9p}) to have such a factor,
we need $\beta+\beta'=(9-2/p)/(N-8)<1$, which is true for $N \ge 17$.
Thus we get a factor of $\|U\|_{W^{6,\infty}}^{2+}$ in both terms on the right-hand side of (\ref{H4-UZ-NZ}). Now the $L^2L^\infty$ norm assumption in (\ref{growthZ}) gives the desired bound for the second term.
\end{proof}

\subsection{$L^2L^\infty$ and $Z$-norm estimates}
\begin{proof}[Proof of (\ref{growth-L2Loo}) and (\ref{growthZ2})]
Recall Duhamel's formula (\ref{Duhamel}) for $U$.
We bound the terms on the right-hand side separately.

{\bf Part 1:} The linear term.
The $Z$-norm bound for $U(0)$ is assumed in (\ref{Z0}).
The $L^2L^\infty$ norm bound for $e^{it\Lambda}U(0)$ then follows from Lemma \ref{dispersiveZ} (iii).

{\bf Part 2:} The quadratic terms. Proposition \ref{W-Z} shows their $Z$-norm bounds, and Lemma \ref{dispersiveZ} (iii) shows the $L^2L^\infty$ norm bound for $e^{-it\Lambda}W_{\mu\nu}(0)$.

To bound the $L^2L^\infty$ norm of $e^{-it\Lambda}W_{\mu\nu}(t)$,
we argue as in the proof of Proposition \ref{W2-C6} to get
\begin{align*}
\|P_ke^{-it\Lambda}W_{\mu\nu}(t)\|_{L^\infty}
&\lesssim 2^{k^--7k^+}\|W_{\mu\nu}(t)\|_{H^8}
\lesssim 2^{k^--7k^+}\|U(t)\|_{W^{9,4}}^2\\
&\lesssim 2^{k^--7k^+}\|U(t)\|_{C_*^6}\|U(t)\|_{H^{12.5}}
\lesssim 2^{k^--7k^+}\|U(t)\|_{C_*^6}\ep_1.
\end{align*}
Integrating in $t$ and using the $L^2L^\infty$ norm assumption in (\ref{growthZ}) we get
\[
\|(1+s)^{(\alpha-\delta)/2}P_ke^{-is\Lambda}W_{\mu\nu}(s)\|_{L_s^2([0,t])L^\infty}
\lesssim_\delta 2^{-7k^++k^-}\ep_1^2.
\]

{\bf Part 3:} The cubic terms
\[
e^{-it\Lambda}\sum_{\mu\nu} \int_0^t H_{\mu\nu,3}(s)ds\quad\text{and}\quad
U_3^\circ(t)=e^{-it\Lambda}\int_0^t e^{is\Lambda}N_3^\circ(s)ds.
\]
Propositions \ref{H3-Z} and \ref{N30-Z} show their $Z$-norm bounds,
and Lemma \ref{dispersiveZ} (iii) shows their $L^2L^\infty$ norm bounds.

{\bf Part 4:} The quartic (and higher) terms
\[
e^{-it\Lambda}\sum_{\mu\nu} \int_0^t H_{\mu\nu,4}(s)ds\quad\text{and}\quad
U_4(t)=e^{-it\Lambda}\int_0^t e^{is\Lambda}N_4(s)ds.
\]
Propositions \ref{H4-Z} and \ref{N4-Z} show their $Z$-norm bounds,
and Lemma \ref{dispersiveZ} (iii) shows their $L^2L^\infty$ norm bounds.
\end{proof}

\section{The periodic case}\label{Period}
In this section we explain how Theorem \ref{Thm1} and Theorem \ref{Thm2} generalize to the periodic case. Most part of the proof works in either case, which we only sketch here, emphasizing the new features on the torus.

\subsection{Linear and multilinear estimates}\label{LinEst-per}
In this subsection we adapt the linear and multilinear estimates to periodic functions. We begin with linear dispersive estimates.
\begin{lemma}\label{dispersive-per}
For $k \in \Z$, $1 \le p \le q \le \infty$ with $1/p + 1/q = 1$ we have
\begin{align}
\label{dispersive1-per}
\|P_ke^{-it\Lambda}u\|_{L^q}&\lesssim [(1+2^{k/2}t)^{-1}(2^{-k/2}t/R+1)^22^{2k}]^{1/p-1/q}\|u\|_{L^p}\\
\label{dispersive2-per}
&\lesssim [(1+t)^{-1}(t/R+1)^22^{(3k^++k)/2}]^{1/p-1/q}\|u\|_{L^p},\\
\|P_ke^{-it\Lambda}u\|_{L^\infty}&\lesssim 2^{(3k^++k)/2}[(1+t)^{-1}(t/R+1)^2]^{1/p-1/q}\|u\|_{L^p}.
\label{dispersive3-per}
\end{align}
\end{lemma}
\begin{proof}
By Poisson summation,
\begin{align*}
P_ke^{-it\Lambda}u(x)&=\int G_k(x,y,t)u(y)dy, &
G_k(x,y,t)&=\sum_{z\in(R\Z)^2} \mathcal K_k(x,y+z,t),
\end{align*}
where $G_k$ and $\mathcal K_k$ are the kernels for the operator $P_ke^{-it\Lambda}$ on the torus and in the Euclidean space, respectively,
see Lemma 3.1 of \cite{Zh}. It was also shown there that $\mathcal K_k(x,y+z,t)$ decays faster than an integrable function in $z$ when $x-y-z$ is so large that $x-y-z-t\nabla\Lambda(\xi)$ never vanishes on the support of $\varphi_k$, e.g., when $|x-y-z|>2^{-k/2}t$. This holds for all but $O(2^{-k/2}t/R+1)^2$ many values of $z$, so we can multiply (\ref{dispersive1}) by this factor to get (\ref{dispersive1-per}).

To get (\ref{dispersive2-per}) we simply use the inequalities $1+2^{k/2}t\ge2^{k^-/2}(1+t)$ and $2^{-k/2}t/R+1\le2^{-k^-/2}(t/R+1)$.

To get (\ref{dispersive3-per}) we also need the Bernstein inequality (\ref{dispersive-Bernstein}).
\end{proof}

\begin{lemma}\label{dispersiveTT*-per}
(i) For $k \in \Z$ we have
\begin{align*}
\|P_ke^{-is\Lambda}u\|_{L^2([0,t])L^\infty}&\lesssim c_{k,t,R}\|u\|_{L^2},\\
c_{k,t,R}&=(2^{3k/4}\sqrt{k^+}+2^{k^+/4+k/2}\sqrt{\min(\mathcal L,\mathcal L_R)})\sqrt{1+t/R}.
\end{align*}

(ii)
\[
\|e^{-is\Lambda}u\|_{L^2([0,t]){W^{6,\infty}}}\lesssim \sqrt{\mathcal L_R(1+t/R)}\|u\|_{H^7}.
\]
\end{lemma}
\begin{proof}
(i) When $t\le2^{k/2}R$, $c_{k,t}\lesssim c_{k,t,R}$, so the result follows from the same $TT^*$ argument as in the proof of Lemma \ref{dispersiveTT*} (i).
When $t > 2^{k/2}R$, we can use unitarity of $e^{it\Lambda}$ to take an $\ell^2$ sum of time intervals of length $2^{k/2}R$. The number of such intervals is $O(1+2^{-k/2}t/R)=O(2^{-k^-/2}(1+t/R))$, so
\[
c_{k,t,R}\lesssim c_{k,t}\sqrt{2^{-k^-/2}(1+t/R)}
\lesssim (2^{3k/4}\sqrt{k^+}+2^{k^+/4+k/2}\sqrt{\mathcal L_R})\sqrt{1+t/R}.
\]

To get (ii) we sum (i) over $k \in \Z$, and add the constant component:
\[
\|e^{-is\Lambda}1\|_{L^2([0,t])W^{6,\infty}} = \sqrt t
\lesssim \sqrt{t/R} \cdot R = \sqrt{t/R} \|1\|_{H^7}.
\]
\end{proof}

\begin{lemma}\label{dispersiveZ-per}
(i) For $k \in \Z$ we have
\[
\|P_ke^{-it\Lambda}u\|_{L^\infty}
\lesssim 2^{-19k^+/3+k^-/2}(1+t)^{-\frac{2}{3}+}(t/R+1)^{4/3}\|u\|_Z.
\]

(ii)
\[
\|e^{-it\Lambda}u\|_{W^{6,\infty}}\lesssim (1+t)^{-\frac{2}{3}+}(t/R+1)^{4/3}\|u\|_Z.
\]

(iii)
\[
\|(1+s)^{\frac{1}{3}-}P_ke^{-is\Lambda}u\|_{L^2([0,t])L^\infty}
\lesssim 2^{-6.4k^++k^-/2}(t/R+1)^{3/2}\|u\|_Z.
\]
\end{lemma}
\begin{proof}
The proofs of (i) and (iii) are the same as the corresponding parts of Lemma \ref{dispersiveZ}.
To get (ii) we sum (i) over $k \in \Z$, and add the constant component:
\[
\|e^{-it\Lambda}1\|_{W^{6,\infty}} = 1 \lesssim (1+t)^{-\frac{2}{3}+}(t/R+1)^{4/3}R^{2/3} = (1+t)^{-\frac{2}{3}+}(t/R+1)^{4/3}\|1\|_Z.
\]
\end{proof}

Now we proceed to adapt the multilinear operators and their estimates.
Much of the material is already present in Sections 3.2 and 3.3 of \cite{Zh}, so we only state the results here, drawing the reader's attention to the differences from there.

\begin{definition}
Given a symbol $a = a(x, \zeta): \R^2 \times (\R^2 \backslash 0) \to \C$, define the operator $T_a$ using the following recipe:
\[
\mathcal F(T_af)(\xi)=\frac{C}{R^2}\sum_{\eta\in(2\pi\Z/R)^2}
\varphi_{\le-10}\left( \frac{|\xi-\eta|}{|\xi+\eta|} \right)\mathcal F_x a\left( \xi-\eta,\frac{\xi+\eta}2 \right)\hat f(\eta),
\]
where $C$ is a normalization constant (independent of $R$) such that $T_1=\text{id}$.
\end{definition}
\begin{remark}
When $\xi+\eta=0$ (even when $\xi=\eta=0$), the factor $\varphi_{\le-10}$ is taken as 0, so $a(x,0)$ will never be used. In particular $T_af$ always has zero average.
\end{remark}

Lemma \ref{para2diff} continues to hold, with the exception that (iii) should read

(iii') If $a=P(\zeta)$, then $T_af=P(D)f-P(0)R^{-2}\int_{(\R/R\Z)^2} f$.

This is because in later proofs regarding the periodic GWW equation,
the functions $f$ being operated on do not necessarily have zero average,
so we have to subtract the mean to ensure that $T_af$ does have zero average.

Now we define the $\mathcal L_m^p$ norm for periodic symbols as follows:
\[
\begin{aligned}
|a|(x,\zeta)
&=\sum_{|I|\le9} |\zeta|^{|I|}|\partial_{\zeta_I} a(x,\zeta)|, &
\|a\|_{\mathcal L_m^p}
&=\sup_{\zeta\in\R^2\backslash0} (1+|\zeta|)^{-m}\||a|(x,\zeta)\|_{L^p_x((\R/R\Z)^2)}.
\end{aligned}
\]
Lemmas \ref{Sm-Lm}, \ref{Lmq=Lq}, \ref{prod-Lm} and \ref{Taf-Lp} remain true. We still define $H(f,g)$ by the formula in Definition \ref{H-def}.
Now we examine Lemma \ref{PkH-Lp}: (i) remains true. 
The nonzero frequencies of (ii) follows from the Coifman--Meyer theorem (Theorem 3.7 of \cite{MuSc}), which is also valid on the torus (Problem 3.4 of \cite{MuSc}). The zero frequency of $H(f,g)$ is simply that of $fg$,
which also has the desired bound by H\"older's inequality. The last part (iii) follows from (i) and (ii) as before.

The operators $E$ and $E_1$ are the same as those in Definitions \ref{Eaf-def} and \ref{E1f-def}, respectively. As such, Lemmas \ref{Eaf-Lp} and \ref{E1f-Lp} continue to hold.

Regarding multilinear paraproducts and their estimates, we still use Definition \ref{Soo-def}. As such, Lemma \ref{Soo-Cn} remains true.
For the $L^p$ bounds of paraproducts, we quote Lemma 3.19 of \cite{Zh}:
\begin{lemma}\label{paraprod-per}
Fix $p,p_j\in[1,\infty]$ ($j=1,\dots,n$) and $1/p=1/p_1+\cdots+1/p_n$. Let
\[
\mathcal Ff(\xi)
=\frac{1}{R^{2n-2}}\sum_{\xi_j\in(2\pi\Z/R)^2\atop\xi_1+\dots+\xi_n=\xi} m(\xi_1,\dots,\xi_n)\prod_{j=1}^n \mathcal F f_j(\xi_j).
\]
Then
\[
\|f\|_{L^p}\lesssim \|m\|_{S^\infty}\prod_{j=1}^n\|f_j\|_{L^{p_j}},
\]
\footnote{The factor $R^{-(2n-2)}$ is missing in \cite{Zh}; the proofs there are not affected, however.}
where the $L^p$ norms are taken on $(\R/R\Z)^2$.
\end{lemma}

\subsection{Paralinearization of the Zakharov system}\label{ParLin-per}
We still define the function $u$ by straightening the boundary,
see (\ref{u-def}). The fixed-point formulation of $\nabla_{x,y}u$,
given in section \ref{BVP}, remains unchanged.

Now we check the validity of the estimates, starting with Lemma \ref{K-bound}. When $k<-\log R$, $P_kf$ vanishes so the bounds are trivial. We remark that thanks to the derivatives in the operator $\mathcal K$, the bounds are also trivially true if $P_k$ is replaced with the operator extracting the zero frequency (i.e., the mean). When $k\ge-\log R$, by the Poisson summation formula, the convolution kernel $\mathcal K_k(y,y',x)$ of $\mathcal K(y,y')\circ P_k$ is the sum of the Euclidean kernel over the lattice $(R\Z)^2$, so it has the bound
\begin{align*}
|\mathcal K_k(y,y',x)|
&\lesssim_L \sum_{x'\in(R\Z)^2} 2^{3k}(1+2^k|x+x'|+2^k\min(|y+y'|,|y-y'|))^{-L}\\
&\lesssim_L 2^{3k}(1+2^k\|x\|+2^k\min(|y+y'|,|y-y'|))^{-L}.
\end{align*}
Then all parts of Lemma \ref{K-bound} follow as before, which then shows that Propositions \ref{u-Hs}, \ref{u-Cr} and \ref{u-Hs2} continue to hold, for their proofs boil down to applying various parts of Lemma \ref{K-bound} to the fixed-point formulation, and the zero frequency has no contribution. Moreover, in the proof of the $L_y^2X_x^r$ bound in Proposition \ref{u-Cr},
the linear evolution $e^{y|\nabla|}\phi$ is decomposed into three parts: $P_{<k}e^{y|\nabla|}\phi$, $P_{[k,0]}e^{y|\nabla|}\phi$ and $P_{>0}e^{y|\nabla|}\phi$. If $k = -\log R$, then the first part vanishes, and the other two parts have their $L_y^2X_x^r$ norms bounded by $\mathcal L_R\||\nabla|^{1/2}\phi\|_{X^r}$. Continuing the proof of Proposition \ref{u-Cr}, we know that this is also a bound for $\|\nabla_{x,y}^ju\|_{L_y^2X_x^{r-j+1}}$ in part (i). Continuing the proof of Proposition \ref{u-Hs2}, we know that the two bounds there still hold if $C_{k,r}[f]$ is replaced with
$\mathcal L_R\|f\|_{C_*^r}$.

Since we have similar Sobolev bounds for $\nabla_{x,y}u$ and multilinear operator bounds for paraproducts in the periodic case, the paralinearization of the Dirichlet-to-Neumann operator, and hence the Zakharov system (\ref{Zakharov}), is exactly the same, with the same error bounds. In particular, Propositions \ref{ht-paralin} and \ref{wt-paralin} remain true, and the constant $C_{k,r}[f]$ in the bounds can be replaced with $\mathcal L_R\|f\|_{C_*^r}$.

Since the identities involving the quantities $a$, $B$, $V$, $G(h)\phi$ and their derivatives can be derived formally, they carry over to the periodic case; hence all the estimates of their Taylor remainders in section \ref{Taylor} continue to hold.

\subsection{Quartic energy estimates}\label{EneEst-per}
We still refer to section \ref{EneEst-def} for the defintion of the energy $\mathcal E$ to be controlled and its decomposition into the parts $\mathcal E_4$, $\mathcal E_s$, $\mathcal E_Q$ and $\mathcal E_{\tilde Q}$.

The estimates for $\mathcal E_4$ and $\mathcal E_s$, namely Propositions \ref{E4} and \ref{Es}, remain mostly unchanged, with the only modification being that the integral over the frequency variables in the definition of $\mathcal E_s^{\mu\nu}$ should be replaced with summation, suitably normalized. In the summation, the zero frequencies are not included because the support of $n_3$ ensures that $|\xi_1+\xi_2| \gtrsim 1$, and since $\mathcal E_S$ is a semilinear term, it further ensures that $|\xi_1|$ and $|\xi_2| \gtrsim 1$.

Next we examine the inclusion of the zero frequncies in $\mathcal E_Q$.
Again the factor $\varphi_{\ge0}(\xi_1+\xi_2)$ ensures that $|\xi_1+\xi_2| \gtrsim 1$, but since $\mathcal E_Q$ is a quasilinear term, this only ensures that $|\xi_2| \gtrsim 1$, which leaves open the possibility of $\xi_1 = 0$. This however, is of no concern because in this case, the multiplier $q(\xi_1, \xi_2)$ vanishes. Hence the bound for $\mathcal E_Q$, namely Proposition \ref{Eq}, also stands.

We now turn to the last term $\mathcal E_{\tilde Q}$. As before the factors $n_2(\xi_2)$ and $n_3(\xi_1+\xi_2)$ in $\tilde q$ ensure that $|\xi_2|$ and $|\xi_1+\xi_2| \gtrsim 1$. If $\xi_1 = 0$ then the multiplier $\tilde q$ still vanishes, which leaves open the possibility of $\theta = 0$ on the right-hand side of (\ref{Iq2-4}). This however, is not a problem either, because later we summed over $k_4 \in \N \cup \{\le 0\}$, which includes the zero frequency of $V_j$. Hence the bound for $\mathcal E_{\tilde Q}$, namely Proposition \ref{Eq-tld}, also remains unchanged.

All told, the growth of $\mathcal E$ can be estimated in exactly the same way as in the Euclidean case. Also the proofs of Proposition 2.1 of \cite{GeMaSh2} and (1.1.15) of \cite{AlDe2} carry over to the torus.

We are now ready to show the energy estimates in the bootstrap assumptions.
\begin{proof}[Proof of (\ref{growth-Em-X})]
This is exactly the same as before.
\end{proof}

\begin{proof}[Proof of (\ref{growth-Em-X-per})]
We use $\mathcal L_R\|U\|_{C_*^6}$ instead of $C_{-\mathcal L,6}[U]$,
which then yields
\[
\mathcal E(t)=\mathcal E(0)+O(\ep_1^3+\mathcal L_R\ep_1^2\ep_2^2)
\lesssim\ep^2+\ep_1^3+\mathcal L_R\ep_1^2\ep_2^2
\]
and the same bound holds for $\|P_{\ge0}\tilde U(t)\|_{H^N}^2$,
and hence for $\|\tilde U(t)\|_{H^N}^2$. Taking the square root gives (\ref{growth-Em-X-per}).
\end{proof}

\begin{proof}[Proof of (\ref{growth-Em-Z-per})]
Integrate Lemma \ref{dispersiveZ-per} (ii) in $t$ and put it in (\ref{growth-Em-X}) (not (\ref{growth-Em-X-per}),
which loses a factor of $\sqrt{\mathcal L_R}$).
\end{proof}

\subsection{Strichartz estimates for small $R$}\label{StrEst-per}
We still split $N = N_2 + N_3$ as in (\ref{N23}), decompose
\[
U(t) = e^{-it\Lambda}U(0) + U_2(t) + U_3(t)
\]
as in (\ref{Duhamel}) and define the profiles $\Upsilon$, $\Upsilon_2$ and $\Upsilon_3$ as before. Then we still have $\Upsilon_t = e^{it\Lambda}N$.

For $\Upsilon_2$ we still decompose it as in (\ref{F=W+H}),
where the multipliers are given by (\ref{m-prod}).
In particular, if either $\xi_1$, $\xi_2$ or $\xi_1 + \xi_2$ vanishes,
so does $m_{\mu\nu}(\xi_1, \xi_2)$, so the zero frequency does not play any role in the bilinear or trilinear terms. Apart from that, the multipliers still enjoy the bounds stated in Lemma \ref{m-Soo-lem}.
The three bounds following the decomposition (\ref{W2k}) also remain true.

We need to check that Theorem C.1 in \cite{GeMaSh2} continues to hold in the periodic case. Since the Coifman--Meyer theorem holds on the torus,
we can still assume one of the frequencies, for example $\eta$, is small.
Then the remainder of the Taylor expansion (\ref{m-Taylor}) is also good.
For the first few terms, it boils down to the $L^p$ boundedness of the operator $Z(\nabla/|\nabla|)$, where $Z$ is a spherical harmonic.
Such boundedness can be seen by expanding
\[
Z(\nabla/|\nabla|)=\sum_{k\ge-\log R} \varphi_k(\nabla)Z(\nabla/|\nabla|)
\]
and noting that the Euclidean kernel $\mathcal K_k$ for the $k$-th term is Calderon--Zygmund and decays rapidly outside the ball of radius $2^{-k}\le R$, so by Poisson summation, the periodic kernel for the $k$-th term
\[
G_k(x) = \sum_{x'\in(R\Z)^2} \mathcal K_k(x+x')
\]
has the same property, and hence the kernel for the left-hand side
\[
G = \sum_{k\ge-\log R} G_k
\]
is also Calderon--Zygmund. Since the weight in the $Z$ norm still belongs to the class $A_2$, the estimates in the equation (C.1) in \cite{GeMaSh2} still go through, showing that part (i) of Theorem C.1 in \cite{GeMaSh2} still holds. Part (ii) follows from part (i) as before.

\begin{proof}[Proof of (\ref{growthX2-per})]
By Lemma \ref{dispersiveTT*-per} (ii) and (\ref{growthX1-per}),
\[
\|e^{-is\Lambda}U(0)\|_{L^2([0,t])C_*^6}\lesssim \sqrt{\mathcal L_R(1+t/R)}\cdot\ep_1.
\]

For $W_{\mu\nu}$, (\ref{W0-H7}) still holds,
and the periodic analog of (\ref{W0-L2C6}) is
\[
\|e^{-is\Lambda}W_{\mu\nu}(0)\|_{L^2([0,t])C_*^6}\lesssim \sqrt{\mathcal L_R(1+t/R)}\cdot\ep_1^2.
\]
The bound (\ref{Ws-L2C6}) for the other boundary term remains unchanged,
so the two quadratic boundary terms have the desired bound.

Since (\ref{N3-Hs}) and (\ref{U3-L2Loo}) still hold, $U_3$ also has the desired bound.

Since (\ref{H-L1H7}) still holds, $H_{\mu\nu}$ also has the desired bound.
\end{proof}

\subsection{$L^2L^\infty$ and $Z$-norm estimates for large $R$}
\label{ZEst-per}
We will still make use of Lemma \ref{int-part-cor} to adapt Propositions \ref{W-Z}, \ref{H3-Z} and \ref{N30-Z} to the periodic case.
Note that now $\alpha = 2/3 \in (1/2, 1)$.

\begin{proposition}\label{W-Z-per}
Assume $N \ge 30$ and (\ref{growthZ-per}). Then
\[
\|W_{\mu\nu}(t)\|_Z\lesssim (t^{5/3+6/N}/R^2+1)\ep_1^2.
\]
\end{proposition}
\begin{proof}
We refer to the decomposition of $W_{\mu\nu}$ in (\ref{W2k}) and the case distinction in the proof of Proposition \ref{W-Z}.

Cases 1, 2 and 3 are the same, as they do not rely on dispersive estimates.
We now focus on Case 4.

{\bf Case 4:} $\kappa:=\min k_i>-5\max(j,\mathcal L)/6$ and $k_1$, $k_2<j/N$.

{\bf Case 4.1:} $j\le-\kappa/2+\mathcal L+5$. In this case $\kappa>-10(\mathcal L+5)/7$, $j<12(\mathcal L+5)/7$ and $k_i\le1.8\mathcal L/N+O(1)$. The case distinction is easier than that of Proposition \ref{W-Z}.

{\bf Case 4.1.1} $j_1 \le j_2 - \kappa/2$. When $q = \infty$ and $p = 1$,
(\ref{dispersive2-per}) has an extra factor of $(t/R + 1)^2$ compared to (\ref{dispersive2}). Hence
\[
\|\langle\nabla\rangle^8P_{k_3}W^{\mu\nu}_{j_1,k_1,j_2,k_2}(t)\|_{L^2}
\lesssim 2^{\sum k_i/2+j_1/3-2j_2/3}\frac{(t/R+1)^2}{1+t}\ep_1^2.
\]
We sum over $j_1\le j_2-\kappa/2$ and $j_2\in\Z$ and use $\kappa/2\le\mathcal L+5-j$ to get
\[
\sum_{j_1\le j_2-\kappa/2,j_2\in\Z\atop\kappa/2\le\mathcal L+5-j}
\|\langle\nabla\rangle^8P_{k_3}W^{\mu\nu}_{j_1,k_1,j_2,k_2}(t)\|_{L^2}
\lesssim 2^{(\sum k_i-\kappa)/2+2(\mathcal L-j)/3}\frac{(t/R+1)^2}{1+t}\ep_1^2.
\]

{\bf Case 4.1.2} $j_1 > j_2 - \kappa/2$. In this case we have
\[
\|\langle\nabla\rangle^8P_{k_3}W^{\mu\nu}_{j_1,k_1,j_2,k_2}(t)\|_{L^2}
\lesssim 2^{(k_3+3k_2^++2k_2)/2+j_2/3-2j_1/3}\frac{(t/R+1)^2}{1+t}\ep_1^2.
\]
We sum over $j_1>j_2-\kappa/2\ge j_2+j-\mathcal L-5$ and $j_2\in\Z$ to get
\[
\sum_{j_1>j_2-\kappa/2,j_2\in\Z\atop\kappa/2\le\mathcal L+5-j}
\|\langle\nabla\rangle^8P_{k_3}W^{\mu\nu}_{j_1,k_1,j_2,k_2}(t)\|_{L^2}
\lesssim 2^{(k_3+3k_2^++2k_2)/2+2(\mathcal L-j)/3}\frac{(t/R+1)^2}{1+t}\ep_1^2.
\]

Combining the two cases above and summing over $-2(\mathcal L+5)<k_i\le 1.8\mathcal L/N+O(1)$ we get
\[
\sum_{\text{Case 4.1}}
\|\langle\nabla\rangle^8P_{k_3}W^{\mu\nu}_{k_1,k_2}(t)\|_{L^2}
\lesssim \frac{2^{-2j/3}\mathcal L(t/R+1)^2}{(1+t)^{1/3-5.4/N}}\ep_1^2.
\]
Then we sum over $0 \le j \lesssim \mathcal L + 1$ and use $N \ge 18$ to get
\begin{align*}
&\left\| \sum_{\text{Case 4.1}}
2^{2j/3}\|Q_j\langle\nabla\rangle^8P_{k_3}W^{\mu\nu}_{k_1,k_2}(t)\|_{L^2} \right\|_{\ell_{0\le j\lesssim\mathcal L+1}}\\
\lesssim&(t^{5/3+6/N}/R^2+(1+t)^{-1/3+6/N})\ep_1^2
\lesssim (t^{5/3+6/N}/R^2+1)\ep_1^2.
\end{align*}

{\bf Case 4.2:} $j > -\kappa/2 + \mathcal L + 5$. Note that in the proof of the corresponding case of Proposition \ref{W-Z}, no dispersive estimate is used, so the proof carries over to the periodic setting.
%
\end{proof}

\begin{proposition}\label{H3-Z-per}
Assume $N \ge 41$ and (\ref{growthZ-per}). Then
\[
\int_0^t \|H_{\mu\nu,3}(s)\|_Zds\lesssim (t^{11/3+13.6/N}/R^4+1)\ep_1^3.
\]
\end{proposition}
\begin{proof}
Again since $m_{\mu\nu}(\xi_1,\xi_2)=0$ when $\xi_1$, $\xi_2$ or $\xi_1+\xi_2=0$, $H_{\mu\nu,3}$ vanishes if any of the ``intermediate frequences" is zero (i.e., $k_i$ ($1\le i\le 4$) or $l=-\infty$ below). We refer to the decomposition of $H_{\mu\nu,3}$ in (\ref{H3k}) and the case distinction in the proof of Proposition \ref{H3-Z}. Also note that the bounds (\ref{H3k-L2}), (\ref{H3k-L2l}) and (\ref{H3k-L2l2}) still hold.

Cases 1 and 3 do not use dispersive estimates and are the same as before.
The parameters are slightly different, however. To get integrable decay in $t$ in Case 1 we need $k_3 \ge (j + 1.4\mathcal L)/N$ and $N \ge 39$.
To get such decay in Case 3 we need $l \le -3\max(j, \mathcal L)/4$.

In Cases 2 and 4 we assume $k_1$, resp. $k_4 \le -3\max(j, \mathcal L)/4$.
The extra factor $(t/R + 1)^{4/3}$ in Lemma \ref{dispersiveZ-per} compared to Lemma \ref{dispersiveZ} gives a bound of
\[
\frac{(t/R+1)^{4/3}\ep_1^3}{(1+t)^{(5/4)-}}
\]
for Case 2, and a larger bound of
\[
\frac{(t/R+1)^{4/3}\ep_1^3}{(1+t)^{(9/8)-}}
\]
for Case 4. Now we focus on Case 5.

{\bf Case 5:} $-3\max(j,\mathcal L)/4<k_i$, $l\le(j+1.4\mathcal L)/N+O(1)$, $1\le i\le 4$.
To get better bounds we will treat this case in a slightly different way than the proof of Proposition \ref{H3-Z}.

{\bf Case 5.1:} $j\le-\min k_i/2+\mathcal L+5$. Similarly to Case 5.1 in that proof the worse case is $l$, $\min k_i>-3j/4\ge3\min k_i/8-3(\mathcal L+5)/4$, which implies $l$, $\min k_i>-6(\mathcal L+5)/5$, $j<8(\mathcal L+5)/5$, and $k_i$, $l\le3\mathcal L/N+O(1)$.

We decompose
\[
H^{\mu\nu\rho\sigma}_{k_1,k_2,k_3,l}=\sum_{j_1,j_2,j_3} H^{\mu\nu\rho\sigma}_{j_1,k_1,j_2,k_2,j_3,k_3,l}
\]
as in Case 4.1 in the proof of Proposition \ref{W-Z}.

{\bf Case 5.1.1:} $j_1 \le j_3 - \min k_i/2$ and $j_2 \le j_3$.
By Lemma \ref{paraprod-per}, (\ref{mk3-Soo}), (\ref{dispersive3-per}), unitarity of $e^{it\Lambda}$, H\"older's inequality and Lemma \ref{Z-bound} (ii),
\begin{align*}
&\|\langle\nabla\rangle^8P_{k_4}H^{\mu\nu\rho\sigma}_{j_1,k_1,j_2,k_2,j_3,k_3,l}(t)\|_{L^2}
\lesssim 2^{8k_3^++(k_4+l+3k_3)/2}\\
\times&\|e^{-it\mu\Lambda}P_{[k_1-1,k_1+1]}Q_{j_1}P_{k_1}\Upsilon_\mu(t)\|_{L^\infty}\|e^{-it\nu\Lambda}P_{[k_2-1,k_2+1]}Q_{j_2}P_{k_2}\Upsilon_\nu(t)\|_{L^\infty}\|Q_{j_3}P_{k_3}\Upsilon(t)\|_{L^2}\\
\lesssim&2^{8k_3^++(k_4+l+3k_3)/2+(3k_2^++k_2+3k_1^++k_1)/2}
\frac{(t/R+1)^4}{(1+t)^2}\\
\times&\|Q_{j_1}P_{k_1}\Upsilon(t)\|_{L^1}\|Q_{j_2}P_{k_2}\Upsilon(t)\|_{L^1}\|Q_{j_3}P_{k_3}\Upsilon(t)\|_{L^2}\\
\lesssim&2^{k_3+(l+\sum k_i)/2+(j_1+j_2-2j_3)/3}\frac{(t/R+1)^4}{(1+t)^2}v_{j_1,k_1}(t)v_{j_2,k_2}(t)v_{j_3,k_3}(t),
\end{align*}
where $v_{j,k} = 2^{8k^++2j/3}\|Q_jP_k\Upsilon\|_{L^2}$.
By the AM-GM inequality and Lemma \ref{Z-bound} (ii),
\begin{align*}
v_{j_1,k_1}v_{j_2,k_2}v_{j_3,k_3}&\lesssim v_{j_1,k_1}^3+v_{j_2,k_2}^3+v_{j_3,k_3}^3,\\
\|v_{j,k}\|_{\ell^3_j}&\le\|v_{j,k}\|_{\ell^2_j}\lesssim\|\Upsilon\|_Z
\lesssim\ep_1.
\end{align*}
We sum over $j_1 \le j_3 - \min k_i/2$, $j_2 \le j_3$, $j_3 \in \Z$ to get,
when $\min k_i/2 \le \mathcal L + 5 - j$,
\begin{align*}
\sum_{j_1\le j_3-\min k_i/2\atop j_2\le j_3, j_3\in\Z} \|\langle\nabla\rangle^8P_{k_4}H^{\mu\nu\rho\sigma}_{j_1,k_1,j_2,k_2,j_3,k_3,l}(t)\|_{L^2}
&\lesssim 2^{k_3+(l+\sum k_i-\min k_i)/2+2(\mathcal L-j)/3}\\
&\times(t/R+1)^4(1+t)^{-2}\ep_1^3.
\end{align*}

{\bf Case 5.1.2:} $j_1 \le j_2 - \min k_i/2$ and $j_2 > j_3$. Similarly we put $L^\infty$ norms on the first and third factors to get
\begin{align*}
&\|\langle\nabla\rangle^8P_{k_4}H^{\mu\nu\rho\sigma}_{j_1,k_1,j_2,k_2,j_3,k_3,l}(t)\|_{L^2}\\
\lesssim&2^{8k_3^++(k_4+l+3k_3)/2+(3k_3^++k_3+3k_1^++k_1)/2}
\frac{(t/R+1)^4}{(1+t)^2}\\
\times&\|Q_{j_1}P_{k_1}\Upsilon(t)\|_{L^1}\|Q_{j_2}P_{k_2}\Upsilon(t)\|_{L^2}\|Q_{j_3}P_{k_3}\Upsilon(t)\|_{L^1}\\
\lesssim&2^{3k_3^++(l+k_4+k_3+k_1)/2+(j_1+j_3-2j_2)/3}\frac{(t/R+1)^4}{(1+t)^2}v_{j_1,k_1}(t)v_{j_2,k_2}(t)v_{j_3,k_3}(t).
\end{align*}
We sum over $j_1 \le j_2 - \min k_i/2$, $j_3 < j_2$, $j_2 \in \Z$ to get,
when $\min k_i/2 \le \mathcal L + 5 - j$,
\begin{align*}
\sum_{j_1\le j_2-\min k_i/2\atop j_3<j_2, j_2\in\Z} \|\langle\nabla\rangle^8P_{k_4}H^{\mu\nu\rho\sigma}_{j_1,k_1,j_2,k_2,j_3,k_3,l}(t)\|_{L^2}
&\lesssim 2^{3k_3^++(l+k_4+k_3+k_1-\min k_i)/2+2(\mathcal L-j)/3}\\
&\times(t/R+1)^4(1+t)^{-2}\ep_1^3.
\end{align*}

{\bf Case 5.1.3:} $j_1>\max(j_2,j_3)-\min k_i/2$. Similarly we put $L^\infty$ norms on the last two factors to get
\begin{align*}
&\|\langle\nabla\rangle^8P_{k_4}H^{\mu\nu\rho\sigma}_{j_1,k_1,j_2,k_2,j_3,k_3,l}(t)\|_{L^2}\\
\lesssim&2^{8k_3^++(l+k_4+3k_3^++4k_3+3k_2^++k_2)/2}\frac{(t/R+1)^4}{(1+t)^2}\\
\times&\|Q_{j_1}P_{k_1}\Upsilon(t)\|_{L^2}\|Q_{j_2}P_{k_2}\Upsilon(t)\|_{L^1}\|Q_{j_3}P_{k_3}\Upsilon(t)\|_{L^1}\\
\lesssim&2^{(l+k_4+3k_3^++4k_3)/2+(j_2+j_3-2j_1)/3}\frac{(t/R+1)^4}{(1+t)^2}v_{j_1,k_1}(t)v_{j_2,k_2}(t)v_{j_3,k_3}(t).
\end{align*}
We sum over $j_1>\max(j_2,j_3)-\min k_i/2\ge\max(j_2,j_3)+j-\mathcal L-5$ and $j_2$, $j_3\in\Z$ to get
\begin{align*}
\sum_{j_1>\max(j_2,j_3)-\min k_i/2\atop j_2,j_3\in\Z}
\|\langle\nabla\rangle^8P_{k_4}H^{\mu\nu\rho\sigma}_{j_1,k_1,j_2,k_2,j_3,k_3,l}(t)\|_{L^2}
&\lesssim 2^{(l+k_4+3k_3^++4k_3)/2+2(\mathcal L-j)/3}\\
&\times(t/R+1)^4(1+t)^{-2}\ep_1^3.
\end{align*}

Combining Case 5.1.1 through Case 5.1.3 and summing over $-2(\mathcal L+5)<k_i,l\le3\mathcal L/N+O(1)$ we get
\[
\sum_{\text{Case 5.1}} \|\langle\nabla\rangle^8P_{k_4}H^{\mu\nu\rho\sigma}_{k_1,k_2,k_3,l}(t)\|_{L^2}
\lesssim \frac{2^{2(\mathcal L-j)/3}(t/R+1)^4}{(1+t)^{2-13.5/N}}\ep_1^3.
\]
Then we sum over $0\le j\lesssim\mathcal L+1$ and use $N \ge 41$ to get
\begin{align*}
&\left\| \sum_{\text{Case 5.1}} 2^{2j/3}\|Q_j\langle\nabla\rangle^8P_kH^{\mu\nu\rho\sigma}_{k_1,k_2,k_3,l}(t)\|_{L^2} \right\|_{\ell^2_{0\le j\lesssim\mathcal L+1}}\\
\lesssim&(t^{8/3+13.6/N}/R^4+(1+t)^{-4/3+13.6/N})\ep_1^3\\
\lesssim&(t^{8/3+13.6/N}/R^4+(1+t)^{-1.001})\ep_1^3.
\end{align*}

{\bf Case 5.2:} $j>-\min k_i/2+\mathcal L+5$. In this case $j>-(j+2\mathcal L)/(2N)+\mathcal L$, which implies $j>(1-1.5/N)\mathcal L$ and $\mathcal L<1.1j$ (because $N\ge16$).

We decompose
\begin{align*}
H^{\mu\nu\rho\sigma}_{k_1,k_2,k_3,l}
&=A^{\mu\nu\rho\sigma}_{k_1,k_2,k_3,l}+B^{\mu\nu\rho\sigma}_{k_1,k_2,k_3,l}
\end{align*}
as in Case 5.2.1 in the proof of Proposition \ref{H3-Z}.

For $A$ we have, by (\ref{H3k-L2l2}), unitarity of $e^{it\Lambda}$ and Lemma \ref{Z-bound} (ii),
\begin{align*}
\|\langle\nabla\rangle^8A^{\mu\nu\rho\sigma}_{k_1,k_2,k_3,l}\|_{L^2}
&\lesssim 2^{k_1+k_2+8k_3^++2k_3+l/2}\prod_{i=1}^3\|Q_{\ge j-4}P_{k_i}\Upsilon\|_{L^2}\\
&\lesssim 2^{-|k_1|-|k_2|+2k_3+l/2-2j}\ep_1^3.
\end{align*}
We sum over $k_1$, $k_2\in\Z$, $k_3$, $l\le(j+1.4\mathcal L)/N+O(1)$ and $j>(1-1.5/N)\mathcal L$ and use $N \ge 32$ to get
\begin{align*}
&\left\| \sum_{\text{Case 5.2}} 2^{2j/3}\|Q_j\langle\nabla\rangle^8A^{\mu\nu\rho\sigma}_{k_1,k_2,k_3,l}(t)\|_{L^2} \right\|_{\ell^2_{j>(1-1.5/N)\mathcal L}}\\
\lesssim&\sum_{j>(1-1.5/N)\mathcal L} 2^{(2.5/N-4/3)j}(1+t)^{3.5/N}\ep_1^3\\
\lesssim&(1+t)^{-4/3+8/N}\ep_1^3\lesssim (1+t)^{-13/12}\ep_1^3.
\end{align*}

The bound for the term $B$ is similar to that in Cases 5.2.1 and 5.2.2 in the proof of Proposition \ref{H3-Z}. Note that since two cases are taken together there is no need to compare $\alpha j$ and $1.1\mathcal L$.
Also note that the conditions $N \ge 15$ and $N \ge 6/(1 - \alpha) = 18$ are both satisfied.

Combining Case 1 through Case 5 above (with integration in $t$ when necessary) shows the result. Note that Case 4 is dominated by Case 5.1.
\end{proof}

The term $N_3^\circ$ has the same bound as $H_{\mu\nu,3}$,
just as in the Euclidean case.
\begin{proposition}\label{N30-Z-per}
Assume $N \ge 41$ and (\ref{growthZ-per}). Then
\[
\int_0^t \|e^{is\Lambda}N_3^\circ(s)\|_Zds\lesssim (t^{11/3+13.6/N}/R^4+1)\ep_1^3.
\]
\end{proposition}

To bound the $Z$-norms of the quartic bulk terms $N_4$ and $H_{\mu\nu,4}$,
we first observe that Lemma \ref{UZ-VZ} remains true in the periodic case.
\begin{proposition}\label{N4-Z-per}
Assume $N \ge 14$, (\ref{growthZ-per}), $t\le R^{2-\gamma}$ and $\gamma>0$.
Then
\[
\int_0^t \|e^{is\Lambda}N_4(s)\|_Zds \lesssim_\gamma [(1+t)^{-\frac{2}{3}+}(t/R+1)^{4/3}]^{2(1-\beta)/3}\ep_1^2\ep_2^2,\text{ where }\beta = 5/(N - 8).
\]
\end{proposition}
\begin{proof}
Note that in addition to the two powers of $\|U\|_{W^{6,\infty}}$,
whose $L^2L^\infty$ norm is assumed in (\ref{growthZ-per}),
the first term on the right-hand side of (\ref{N4-UZ-VZ}) has $2(1 - \beta)/3$ extra powers, and the second has $1 - \beta$ extra powers.
By Lemma \ref{dispersiveZ-per} (ii), each factor of $\|U\|_{W^{6,\infty}}$ is bounded by $(1 + t)^{-\frac{2}{3}+}(t/R + 1)^{4/3}\ep_1 \lesssim 1$.
\end{proof}

Finally we turn to $H_{\mu\nu,4}$. Lemma \ref{paraprod-Z} remains true.
\begin{proposition}\label{H4-Z-per}
Assume $N \ge 17$, (\ref{growthZ-per}), $t\le R^{2-\gamma}$ and $\gamma>0$.
Then
\[
\int_0^t \|e^{is\Lambda}H_{\mu\nu,4}(s)\|_Zds \lesssim_\gamma [(1+t)^{-\frac{2}{3}+}(t/R+1)^{4/3}]^{\min(2(1-\beta)/3,1-\beta')}\ep_1^2\ep_2^2,
\]
where $\beta = 5/(N - 8)$ and $\beta' = 9/(N - 8)$.
\end{proposition}
\begin{proof}
It suffices to note that in addition to the two powers of $\|U\|_{W^{6,\infty}}$, the first term on the right-hand side of (\ref{H4-UZ-VZ}) has $2(1 - \beta)/3$ extra powers, and the second term has $1 - \beta'$ extra powers.
\end{proof}
\begin{remark}
If $N \ge 25$, then $1 - \beta' \ge 2(1 - \beta)/3$.
\end{remark}

\begin{proof}[Proof of (\ref{growth-L2Loo-per}) and (\ref{growthZ2-per})]
Recall Duhamel's formula (\ref{Duhamel}) for $U$.
We bound the terms on the right-hand side separately.

{\bf Part 1:} The linear term. This follows from the assumption (\ref{Z0}) and Lemma \ref{dispersiveZ-per} (iii) as before.

{\bf Part 2:} The quadratic terms. Proposition \ref{W-Z-per} shows their $Z$-norm bounds.

{\bf Part 2.1:} $e^{-it\Lambda}W_{\mu\nu}(0)$. Its $L^2L^\infty$ norm bound follows from Lemma \ref{dispersiveZ-per}.

{\bf Part 2.2:} $e^{-it\Lambda}W_{\mu\nu}(t)$. As before we have
\begin{align*}
\|P_ke^{-it\Lambda}W_{\mu\nu}(t)\|_{L^\infty}
\lesssim 2^{k^--7k^+}\|U(t)\|_{C_*^6}\ep_1.
\end{align*}
Integrating in $t$ and using the $L^2L^\infty$ norm assumption in (\ref{growthZ-per}) we get
\[
\|(1+s)^{1/3-\delta}P_ke^{-is\Lambda}W_{\mu\nu}(s)\|_{L_s^2([0,t])L^\infty}
\lesssim_\delta 2^{-7k^++k^-}\ep_1^2.
\]

{\bf Part 3:} The cubic terms. This follows from Propositions \ref{H3-Z-per} and \ref{N30-Z-per} and Lemma \ref{dispersiveZ-per} (iii) as before.

{\bf Part 4:} The quartic terms. This follows from Propositions \ref{H4-Z-per} and \ref{N4-Z-per} and Lemma \ref{dispersiveZ-per} (iii) as before.
\end{proof}

\section{Appendix: Paralinearization and energy estimates}\label{ParLinEst}
In this appendix we prove the estimates used in section \ref{ParLin}.

\subsection{Estimates of the solution to the Dirichlet problem}\label{BVPEst}
We use a fixed-point argument to estimate the solution to (\ref{u-Diri}), which can be written as
\begin{equation}\label{u-Diri-f}
\begin{cases}
(\partial_y^2+\Delta_x)u=f=\partial_yf_1+|\nabla_x|f_2,\\
u(x,0,t)=\phi(x,t),\\
\nabla_{x,y}u(x,y,t)\to0\text{ as $y \to -\infty$.}
\end{cases}
\end{equation}
where
\begin{align}\label{f12-def}
f_1&=-|\nabla h|^2\partial_yu+\nabla h\cdot\nabla_xu, &
f_2&=|\nabla_x|^{-1}\nabla_x\cdot(\partial_yu\nabla h).
\end{align}

First we show that for fixed $f$ and $\phi$, (\ref{u-Diri-f}) has at most one solution. It suffices to show that when $f = 0$ and $\phi = 0$, (\ref{u-Diri-f}) has only the zero solution, which follows from the lemma below.
\begin{lemma}\label{u-Diri-unique}
Let $\Omega$ be the region in $\R^3$ below the graph of a continuous function. Let $u\in C(\bar\Omega)$ be harmonic in $\Omega$, vanish on $\Gamma = \partial\Omega$, and grow slower than $|z|$ as $z \to -\infty$. Then $u = 0$ on $\Omega$.
\end{lemma}
\begin{proof}
Let $\ep > 0$ and $u_\ep = u + \ep z$. Because $u$ grows slower than $|z|$,
we know that $u_\ep \to -\infty$ as $z \to -\infty$. Since $u_\ep$ is continuous in $\bar\Omega$, its maximum is attained. Since $u_\ep$ is harmonic in $\Omega$, its maximum is attained on $\Gamma$.
Since $u_\ep \le \ep \sup h$ on $\Gamma$, so is it on $\Omega$.
Letting $\ep \to 0$ we know that $u \le 0$ on $\Omega$.
Similarly $u \ge 0$ on $\Omega$, and hence $u = 0$ on $\Omega$.
\end{proof}

When $f_1 = 0$ and $\phi = 0$, the solution to (\ref{u-Diri-f}) is, by the method of image,
\[
u(x,y) = \frac{1}{2}\int_{-\infty}^0 e^{(y+y')|\nabla_x|}f_2(x,y')dy'
- \frac{1}{2}\int_{-\infty}^0 e^{-|y-y'||\nabla_x|}f_2(x,y')dy'
\]
provided that it does satisfy the boundary condition at $-\infty$, i.e.,
the third line in (\ref{u-Diri-f}).

When $f_2 = 0$ and $\phi = 0$, we replace $f_2$ with $|\nabla_x|^{-1}\partial_yf_1$ in the right-hand side of the above and integrate by parts to get
\[
u(x,y) = -\frac{1}{2}\int_{-\infty}^0 e^{(y+y')|\nabla_x|}f_1(x,y')dy'
+ \frac{1}{2}\int_{-\infty}^0 e^{-|y-y'||\nabla_x|}\sgn(y-y')f_1(x,y')dy'
\]
with the same proviso.

Since $e^{y|\nabla|}\phi$ is harmonic on $\{y < 0\}$, has the boundary value $\phi$ on $\{y = 0\}$, and satisfies the boundary condition at $-\infty$, the solution to (\ref{u-Diri-f}) is (see also (B.5) of \cite{Wa2DG} or (B.21) of \cite{DeIoPaPu})
\begin{align*}
u(x,y,t)&=e^{y|\nabla|}\phi(x,t)+\frac12\int_{-\infty}^0 e^{(y+y')|\nabla_x|}
(-f_1(x,y',t)+f_2(x,y',t))dy'\\
&+\frac12\int_{-\infty}^0 e^{-|y-y'||\nabla_x|}(\sgn(y-y')f_1(x,y',t)
-f_2(x,y',t))dy'
\end{align*}
whose gradient is (see also (B.7)--(B.9) of \cite{Wa2DG} or (1.1.20)--(1.1.21) of \cite{AlDe2})
\begin{equation}\label{gr-u}
\begin{aligned}
(\nabla_xu,\partial_yu)^T(x,y,t)
&=e^{y|\nabla|}(\nabla\phi,|\nabla|\phi)^T(x,t)\\
&+\int_{-\infty}^0 \mathcal K(y,y')M(\nabla h)(\nabla_xu,\partial_yu)^T(x,y',t)dy'\\
&+[0,f_1(x,y,t)]^T,
\end{aligned}
\end{equation}
where the operator (which collects all the Fourier multipliers)
\begin{equation}\label{K-def}
\mathcal K(y,y')=\frac{e^{(y+y')|\nabla_x|}}{2}
\begin{pmatrix}
-|\nabla_x| & \nabla_x\\
\nabla_x & |\nabla_x|
\end{pmatrix}
+\frac{e^{-|y-y'||\nabla_x|}}{2}
\begin{pmatrix}
|\nabla_x| & \sgn(y'-y)\nabla_x\\
\sgn(y-y')\nabla_x & |\nabla_x|
\end{pmatrix}
\end{equation}
and the multiplicative factor (which captures the coefficients)
\begin{equation}\label{M-def}
M(\nabla h)=
\begin{pmatrix}
0 & \nabla h\\
-\nabla h & |\nabla h|^2
\end{pmatrix}
\end{equation}
provided that the solution does satisfy the boundary condition at $-\infty$.

To bound the $H^s$ norm of the solution in terms of the initial data,
we use
\begin{lemma}\label{K-bound}
Fix $1\le p_1\le p_2\le\infty$ and $1\le q_1\le q_2\le\infty$.

(i) For any $k\in\Z$ we have
\[
\left\| \int_{-\infty}^0 \mathcal K(y,y')P_kf(y')dy' \right\|_{L_y^{q_2}L_x^{p_2}}
\lesssim 2^{(2/p_1-2/p_2+1/q_1-1/q_2)k}\|P_kf\|_{L_y^{q_1}L_x^{p_1}}.
\]

(ii) For fixed $s\in\R$ and $q\ge2$ we have
\[
\left\| |\nabla|^{1/q-1/2}\int_{-\infty}^0 \mathcal K(y,y')f(y')dy' \right\|_{L_y^qH_x^s}\lesssim \|f\|_{L_y^2H_x^s}.
\]

(iii) If $(p_1,q_1)\neq(p_2,q_2)$ then for any $k\in\Z$ we have
\[
\left\| \int_{-\infty}^0 \mathcal K(y,y')P_{<k}f(y')dy' \right\|_{L_y^{q_2}L_x^{p_2}}
\lesssim 2^{(2/p_1-2/p_2+1/q_1-1/q_2)k}\|P_{<k}f\|_{L_y^{q_1}L_x^{p_1}}.
\]
\end{lemma}
\begin{proof}
(i) We have
\[
\int_{-\infty}^0 \mathcal K(y,y')P_kf(y')dy'
= \int_{-\infty}^0 \int_{\R^2} \mathcal K_k(y,y',x-x')f(x')dx'dy'
\]
where the convolution kernel is
\begin{align*}
\mathcal K_k(y,y',x) &= C\int_{\R^2} e^{ix\xi+(y+y')|\xi|}
\begin{pmatrix}
-|\xi| & i\xi\\
i\xi & |\xi|
\end{pmatrix}\varphi_k(\xi)d\xi\\
&+ C\int_{\R^2} e^{ix\xi-|y-y'||\xi|}
\begin{pmatrix}
|\xi| & i\sgn(y'-y)\xi\\
i\sgn(y-y')\xi & |\xi|
\end{pmatrix}\varphi_k(\xi)d\xi.
\end{align*}
Integrating by parts in $\xi$ for $L$ times, we get
\[
|\mathcal K_k(y,y',x)|
\lesssim_L 2^{3k}(1+2^k|x|+2^k\min(|y+y'|,|y-y'|))^{-L}.
\]
Then (i) follows from H\"older's inequality.

Taking a weighted $\ell^2$ sum of (i) over $k\in\Z$ and using Minkowski's inequality we get (ii).

Taking an $\ell^1$ sum of (i) over frequencies $\lesssim 2^k$ we get (iii).
\end{proof}

Now we show Propositions \ref{u-Hs}, \ref{u-Cr} and \ref{u-Hs2},
which we restate here.
\begin{proposition}
Fix $s>j\ge1$.

(i) If $\|\nabla h\|_{H^s}<c_s$ is sufficiently small then
\[
\|\nabla_{x,y}^ju\|_{L_y^2H_x^{s-j+1}}\lesssim \||\nabla|^{1/2}\phi\|_{H^s}.
\]

(i') If in addition we have $s>j+1$ then $\|\nabla_{x,y}^ju\|_{H_x^{s-j}}\to0$ as $y\to-\infty$.

(ii) If $\|\nabla h\|_{H^{s+1/2}}<c_s$ is sufficiently small then
\[
\|\nabla_{x,y}^ju\|_{L_y^\infty H_x^{s-j+1}}\lesssim \||\nabla|^{1/2}\phi\|_{H^{s+1/2}}.
\]
\end{proposition}
\begin{proof}
We first show the bound for $j=1$. We estimate the three terms on the right of (\ref{gr-u}) separately. For (i), the first term can be bounded in the frequency space directly. By Lemma \ref{K-bound} (ii), the second term is bounded by $\|M(\nabla h)(\nabla_xu,\partial_yu)^T\|_{L_y^2H_x^s}$. Clearly the third term is bounded by $\|f_1\|_{L_y^2H_x^s}$.
By (\ref{f12-def}), (\ref{M-def}), $s>1$ and the Sobolev multiplication theorem then,
\[
\|\nabla_{x,y}u\|_{L_y^2H_x^s}\lesssim_s \||\nabla|^{1/2}\phi\|_{H^s}+(c_s+c_s^2)\|\nabla_{x,y}u\|_{L_y^2H_x^s}.
\]
If $c_s$ is small enough, the second term on the right-hand side can be absorbed in the left-hand side, which shows (i) for $j=1$.

For $j=1$, (ii) follows in the same way as (i), except that when bounding the second term on the right of (\ref{gr-u}), we set $q=\infty$ in Lemma \ref{K-bound} (ii).

Taking derivatives in $x$ and using (\ref{u-Diri}), Sobolev multiplication and $\|(\nabla h,\Delta h)\|_{H^{s-1}}<c_s$ give the desired bounds for $\partial_y^2u$ and all terms for $j\ge2$ by induction.

For (i') we have $\|\nabla_{x,y}^{j+1}u\|_{L_y^2H_x^{s-j}}<\infty$ by (i).
Integrating in $y$ and using the Cauchy--Schwarz inequality then show that
$\|\nabla_{x,y}^ju\|_{H_x^{s-j}}$ is uniformly continuous in $y$. Combining this with $\|\nabla_{x,y}^ju\|_{L_y^2H_x^{s-j}}<\infty$ gives (i').
\end{proof}
\begin{remark}
If $s > 2$, then $H^{s-1}$ embeds in $C^0$, so $u$ is continuous on $\{y \le 0\}$ and satisfies the boundary condition at $-\infty$ in (\ref{u-Diri-f}).
\end{remark}

Recall the Besov norms, which for $r>0$ is
\[
\|u\|_{B^r_{p,q}}=\|P_{<0}u\|_{L^p}+\|2^{rk}\|P_ku\|_{L^p}\|_{\ell^q_{k\ge0}}.
\]
Also recall that $C_*^r=B^r_{\infty,\infty}$ and $X^r=B^r_{\infty,2}$.
Clearly, $C_*^{r+}\subset X^r\subset C_*^r$. Also $H^{r+1}\subset C_*^r\cap X^r$. For $r\in\N^+$, both $C_*^{r+}$ and $X^{r+}\subset W^{r,\infty}$.
By \cite{Tr}, Theorem 2 (i), for $r>0$, both $C_*^r$ and $X^r$ are closed under multiplication.

\begin{proposition}
Fix $1\le j<r+1$. If $\|\nabla h\|_{H^{r+1}}<c_r$ is small enough then:

(i) For any integer $k<0$,
\[
\|\nabla_{x,y}^ju\|_{L_y^2X_x^{r-j+1}}\lesssim |k|\||\nabla|^{1/2}\phi\|_{X^r}+2^k\||\nabla|^{1/2}\phi\|_{L^2}.
\]

(i') If $r>j$ then $\|\nabla_{x,y}^ju\|_{X_x^{r-j}}\to0$ as $y\to-\infty$.

(ii)
\[
\|\nabla_{x,y}^ju\|_{L_y^\infty C_*^{r-j+1}}\lesssim \||\nabla|^{1/2}\phi\|_{C_*^{r+1/2}}.
\]
\end{proposition}
\begin{proof}
First we show the bound for $j=1$. We estimate the three terms on the right-hand side of (\ref{gr-u}) separately. For both (i) and (ii), the first term is the convolution of $(\nabla\phi,|\nabla|\phi)$ with the kernel of $e^{y|\nabla|}$, so by Young's inequality,
\begin{align*}
\|\nabla_{x,y}e^{y|\nabla|}\phi\|_{L_y^\infty C_*^r}
&\lesssim \sum_{l<0} 2^{l/2}\||\nabla|^{1/2}\phi\|_{L^\infty}
+\sup_{l\ge0} 2^{(r+1/2)l}\|P_l|\nabla|^{1/2}\phi\|_{L^\infty},\\
\|\nabla_{x,y}e^{y|\nabla|}\phi\|_{L_y^2X_x^r}
&\lesssim \sum_{l<k} 2^l\||\nabla|^{1/2}\phi\|_{L^2}
+\sum_{k\le l<0} \|P_l|\nabla|^{1/2}\phi\|_{L^\infty}
+\|2^{rl^+}\|P_l|\nabla|^{1/2}\phi\|_{L^\infty}\|_{\ell^2_{l\ge0}}.
\end{align*}
Both terms have their desired bounds.

By Lemma \ref{K-bound} (iii) and (i), the second term of (\ref{gr-u})
in $L_y^pB^r_{\infty,p}$ ($p\in\{2,\infty\}$)
\begin{align*}
&\lesssim \|P_{<0}M(\nabla h)(\nabla_xu,\partial_yu)^T\|_{L_y^pL_x^2}
+\|M(\nabla h)(\nabla_xu,\partial_yu)^T\|_{L_y^pB^r_{\infty,p}}\\
&\lesssim \|\nabla h\|_{L^2}(1+\|\nabla h\|_{L^\infty})\|\nabla_{x,y}u\|_{L_y^pL_x^\infty}+(\|\nabla h\|_{B^r_{\infty,p}}+\|\nabla h\|_{B^r_{\infty,p}}^2)\|\nabla_{x,y}u\|_{L_y^pB^r_{\infty,p}}\\
&\lesssim (c_r+c_r^2)\|\nabla_{x,y}u\|_{L_y^pB^r_{\infty,p}}.
\end{align*}
The third term satisfies the same bound because it is the second component of $M(\nabla h)(\nabla_xu,\partial_yu)^T$. When $c_r$ is small enough we can absorb the right-hand side in the left-hand side and the bound for $j=1$ follows.

Taking derivatives in $x$ and using (\ref{u-Diri}) and the bound $\|(\nabla h,\Delta h)\|_{X^{r-1}\cap C_*^{r-1}}\lesssim\|\nabla h\|_{H^{r+1}}$
give the desired bounds for $j\ge2$ by induction.

Finally, the implication (i) $\to$ (i') is similar to that in Proposition \ref{u-Hs}.
\end{proof}

\begin{proposition}
Fix $s>1$.

(i) If $\|\nabla h\|_{H^s}<c_s$ is sufficiently small then for any integer $k<0$,
\[
\|\nabla_{x,y}(u-e^{y|\nabla|}\phi)\|_{L_y^2H_x^s}
\lesssim \|\nabla h\|_{L^\infty}\||\nabla|^{1/2}\phi\|_{H^s}
+\|\nabla h\|_{H^s}C_{k,1}[|\nabla|^{1/2}\phi],
\]
where
\[
C_{k,r}[f]=|k|\|f\|_{C_*^r}+2^k\|f\|_{L^2}.
\]

(ii) If $\|\nabla h\|_{H^{s+1/2}}<c_s$ is sufficiently small then for any integer $k<0$,
\[
\|\nabla_{x,y}(u-e^{y|\nabla|}\phi)\|_{L_y^\infty H_x^s}
\lesssim \|\nabla h\|_{L^\infty}\||\nabla|^{1/2}\phi\|_{H^{s+1/2}}
+\|\nabla h\|_{H^{s+1/2}}C_{k,1}[|\nabla|^{1/2}\phi].
\]
\end{proposition}
\begin{proof}
This follows from Lemma \ref{K-bound} (ii) applied to (\ref{gr-u}),
the Sobolev multiplication theorem, Proposition \ref{u-Hs} (i), Proposition \ref{u-Cr} (i) and the embedding $C_*^1\subset X^{0+}$.
\end{proof}

\subsection{Paralinearization of the Zakharov system}
Now we bound various terms appearing in the process of paralinearization. First recall $\mathcal S_0$ and $\mathcal C_0$, defined in (\ref{S0}) and (\ref{C0}) respectively.
\begin{proposition}
If $s>2$ and $\|\nabla h\|_{H^{s-1}\cap H^{3+}}$ is sufficiently small then
\begin{align}
\label{S0-bound}
\|\mathcal S_0\|_{L_y^2H_x^s}&\lesssim_s \|\nabla h\|_{W^{2,\infty}}\||\nabla|^{1/2}\phi\|_{H^{s-1}},\\
\|\mathcal C_0\|_{L_y^2H_x^s}&\lesssim_s \|\nabla h\|_{W^{1,\infty}}C_{k,3}[|\nabla|^{1/2}\phi]\|h\|_{H^s}\ (k<0),
\label{C0-bound}
\end{align}
where $C_{k,r}[\cdot]$ is defined as above.
\end{proposition}
\begin{proof}
For $\mathcal S_0$, (\ref{S0-bound}) follows from Lemma \ref{PkH-Lp} (iii) (with $m = 2$) and Proposition \ref{u-Hs} (i).

For $\mathcal C_0$, by Lemma \ref{Lmq=Lq}, Lemma \ref{Taf-Lp} (ii), Lemma \ref{PkH-Lp} (iii) (with $s>m=1$), Lemma \ref{Eaf-Lp} (ii) and the smallness of $\|\nabla h\|_{L^\infty}\lesssim\|\nabla h\|_{H^2}$,
\[
\|\mathcal C_0\|_{L_y^2H_x^s}\lesssim_s \|\nabla h\|_{W^{1,\infty}}
(\|\partial_y^3u\|_{L_y^2L_x^\infty}+\|\partial_y^2u\|_{L_y^2W_x^{1,\infty}})\|h\|_{H^s}.
\]
Then (\ref{C0-bound}) follows from Proposition \ref{u-Cr} (i) and the embedding $C_*^3\subset X^{2+}$.
\end{proof}

Then we show operator norms of $\mathcal R_1$ and $\mathcal R_2$, defined in (\ref{R1}) and (\ref{R2}) respectively.
\begin{proposition}
If $s>4$ and $\|h\|_{H^s}<c_s$ is sufficiently small then for $j\in\{1, 2\}$,
\begin{equation}\label{Rj-bound}
\|\mathcal R_jw\|_{L_y^2H_x^s}\lesssim_s \|\nabla h\|_{W^{3,\infty}}^j\||\nabla|^{1/2}\phi\|_{H^{s-1}}.
\end{equation}
\end{proposition}
\begin{proof}
By Lemma \ref{Taf-Lp} (ii), Lemma \ref{Eaf-Lp} (ii) and Lemma \ref{E1f-Lp} (ii), applied to (\ref{w-def}), (\ref{ab-soln}) and (\ref{R2}),
\begin{align*}
\|\mathcal R_jw\|_{L_y^2H_x^s}
&\lesssim_s \|\nabla h\|_{W^{3,\infty}}^j
(\|\partial_y^2w\|_{L_y^2H_x^{s-2}}+\|\nabla_{x,y}w\|_{L_y^2H_x^{s-1}})\\
&\lesssim_s \|\nabla h\|_{W^{3,\infty}}^j
(\|\partial_y^2u\|_{L_y^2H_x^{s-2}}+\|\nabla_{x,y}u\|_{L_y^2H_x^{s-1}}\\
&+(\|\partial_y^3u\|_{L_y^2L_x^\infty}+\|\partial_y^2u\|_{L_y^2L_x^\infty}+\|\partial_yu\|_{L_y^2W_x^{1,\infty}})\|h\|_{H^s}).
\end{align*}
By Proposition \ref{u-Hs} (i) and $s>2$, $\|\nabla_{x,y}^ju\|_{L_y^2H_x^{s-j}} \lesssim_s \||\nabla|^{1/2}\phi\|_{H^{s-1}}$ for $j\le2$.
By Sobolev embedding and $s>4$, $\|\nabla_{x,y}^ju\|_{L_y^2L_x^\infty} \lesssim_s \||\nabla|^{1/2}\phi\|_{H^{s-1}}$ for $j\le3$.
Since $\|h\|_{H^s}$ is small, the last two terms are bounded by the first two, and (\ref{Rj-bound}) follows.
\end{proof}

Next we paralinearize $\partial_yw$.
Recall $\mathcal A$, $\mathcal B$ from (\ref{AB-ab}) through (\ref{ab-soln}), and $\mathcal Q$ and $\mathcal S$ from (\ref{QS}).
\begin{proposition}
If $s>4$ and $\|h\|_{H^s}<c_s$ is sufficiently small then
\begin{equation}\label{Taw}
T_{\sqrt{1+\alpha}}\partial_yw=(\mathcal{A+B})w+\mathcal Q+\mathcal S+\mathcal C_2,
\end{equation}
where, for any integer $k<0$,
\begin{equation}\label{C2-bound}
\|\mathcal C_2\|_{L_y^\infty H_x^{s+1/2}}
\lesssim_s \|\nabla h\|_{W^{3,\infty}}(C_{k,4}[|\nabla|^{1/2}\phi]\|h\|_{H^s}+\|\nabla h\|_{W^{3,\infty}}\||\nabla|^{1/2}\phi\|_{H^{s-1}}).
\end{equation}
\end{proposition}
\begin{proof}
Let
\begin{equation}\label{W-def}
W=(T_{\sqrt{1+\alpha}}\partial_y-\mathcal{A-B})w.
\end{equation}
Then $\mathcal C_2=W-\mathcal{Q-S}$.
From (\ref{Pw}) and (\ref{P-factor}) it follows that
\begin{align*}
(T_{\sqrt{1+\alpha}}\partial_y-\mathcal{A+B})W
&=\mathcal S_0+\mathcal R_1w+\mathcal C_0+\mathcal R_2w.
\end{align*}
Let $\sigma=(b-ia)/\sqrt{1+\alpha}$. Then $\mathcal{B-A}=T_{\sigma\sqrt{1+\alpha}}$. Let $W'=T_{\sqrt[4]{1+\alpha}}W$. Then
\begin{equation}\label{P-factor2}
\begin{aligned}
(\partial_y+T_\sigma)W'
&=\mathcal S_0+\mathcal R_1w+\mathcal C_1,\\
\mathcal C_1&=T_{1/\sqrt[4]{1+\alpha}-1}(\mathcal S_0+\mathcal R_1w)
+T_{1/\sqrt[4]{1+\alpha}}(\mathcal C_0+\mathcal R_2w)+\mathcal R_2'W,\\
\mathcal R_2'&=T_\sigma T_{\sqrt[4]{1+\alpha}}-T_{1/\sqrt[4]{1+\alpha}}T_{\sigma\sqrt{1+\alpha}}.
\end{aligned}
\end{equation}
Let $\beta=\sqrt[4]{1+\alpha}$. Since
$\{\sigma,\beta\}=-\nabla_\zeta\sigma\cdot\nabla\beta
=\nabla(1/\beta)\nabla_\zeta(\sigma\beta^2)=\{1/\beta,\sigma\beta^2\}$,
$\mathcal R_2'$ is a paradifferential operator of order $-1$ with coefficients at least degree 2 in $h$, so
\begin{equation}\label{R2'-bound}
\|\mathcal R_2'W\|_{L_y^2H_x^s}
\lesssim_s \|\nabla h\|_{W^{3,\infty}}^2\|W\|_{L_y^2H_x^{s-1}}
\lesssim_s \|\nabla h\|_{W^{3,\infty}}^2\||\nabla|^{1/2}\phi\|_{H^{s-1}}.
\end{equation}
By (\ref{S0-bound}), (\ref{Rj-bound}), (\ref{C0-bound}) and (\ref{R2'-bound}), $\|\mathcal C_1\|_{L_y^2H_x^s}\lesssim_s$ the right-hand side of (\ref{C2-bound}).

We can rewrite (\ref{P-factor2}) as
\begin{equation}\label{W'y}
(\partial_y+T_\sigma)W'=\mathcal R_1[e^{y|\nabla|}\phi]+\mathcal S_0[h,e^{y|\nabla|}\phi]+\mathcal C_1',
\end{equation}
where, by Lemma \ref{Taf-Lp} (ii), Lemma \ref{PkH-Lp} (iii), Lemma \ref{Eaf-Lp} (ii), Lemma \ref{E1f-Lp} (ii), Proposition \ref{u-Hs} (i) and Proposition \ref{u-Hs2} (i), applied to (\ref{w-def}), (\ref{S0}) and (\ref{W-def}),
\begin{align*}
\mathcal C_1'&=\mathcal C_1+\mathcal R_1[w-e^{y|\nabla|}\phi]+\mathcal S_0[h,u-e^{y|\nabla|}\phi]
\end{align*}
satisfies the same bound as $\mathcal C_1$.
(\ref{W'y}) can be further rewritten as
\[
(\partial_y+|\nabla_x|)W'=(|\nabla_x|-T_\sigma)W'+\mathcal R_1[e^{y|\nabla|}\phi]+\mathcal S_0[h,e^{y|\nabla|}\phi]+\mathcal C_1'.
\]

Put $W'' = W' - \mathcal{Q - S}$. Then
\begin{equation}\label{W''y}
\begin{aligned}
(\partial_y+|\nabla_x|)W''&=(|\nabla_x|-T_\sigma)W''+\mathcal C_1'',\\
\mathcal C_1''&=\mathcal C_1'+(|\nabla_x|-T_\sigma)(\mathcal Q+\mathcal S).
\end{aligned}
\end{equation}
By Lemma \ref{Taf-Lp} (ii) and Proposition \ref{QS-bound},
$\mathcal C_1''$ satisfies the same bound as $\mathcal C_1$,

As $y\to-\infty$, by Proposition \ref{u-Hs} (i'), Proposition \ref{u-Cr} (i') and the embedding $X^{r+} \subset W^{r,\infty}$, we know $\|\partial_y^2u\|_{L_x^\infty}$, $\|\partial_yu\|_{W_x^{1,\infty}}$ and $\|\nabla_{x,y}u\|_{H_x^2} \to 0$,
so by Lemma \ref{Taf-Lp} (ii) applied to (\ref{w-def}) and (\ref{W-def}),
$\nabla_{x,y}w$, $W$ and hence $W'=T_{\sqrt[4]{1+\alpha}}W\to0$ in $H_x^2$.
By Proposition \ref{QS-bound}, $\|\mathcal{Q+S}\|_{H_x^2}\lesssim \|e^{y|\nabla|}|\nabla|^{1/2}\phi\|_{H_x^{1/2}}\to0$,
so $W'' = W' - \mathcal{Q - S} \to 0$ in $H_x^2$.
Thus we can integrate (\ref{W''y}) from $y = -\infty$ to get
\[
W''(x,y)=\int_{-\infty}^y e^{(y'-y)|\nabla_x|}[(|\nabla_x|-T_\sigma)W''
+\mathcal C_1''](x,y')dy'.
\]
By Lemma \ref{K-bound} (ii) and Lemma \ref{Taf-Lp} (ii) then,
\begin{align*}
\|\nabla_xW''\|_{L_y^2H_x^s}+\|W''\|_{L_y^\infty H_x^{s+1/2}}
&\lesssim_s \|(|\nabla_x|-T_\sigma)W''+\mathcal C_1''\|_{L_y^2H_x^s}\\
&\lesssim_s \|\nabla h\|_{W^{1,\infty}}\|\nabla_xW''\|_{L_y^2H_x^s}
+\|\mathcal C_1''\|_{L_y^2H_x^s}.
\end{align*}
After absorbing the first term on the right-hand side in the left-hand side, we know that $W''$ satisfies the bound (\ref{C2-bound}).
By Lemma \ref{Taf-Lp} (ii),
$T_{1/\sqrt[4]{1+\alpha}}W''=W-T_{1/\sqrt[4]{1+\alpha}}(\mathcal{Q+S})$
has the same bound. By Lemma \ref{Taf-Lp} (ii) and Proposition \ref{QS-bound}, $T_{1/\sqrt[4]{1+\alpha}-1}(\mathcal{Q+S})$ has the same bound, so does $\mathcal C_2=W-\mathcal{Q-S}$, the sum of the two.
\end{proof}

Then we paralinearize the Dirichlet-to-Neumann operator (\ref{DN-u}).
\begin{proposition}\label{Gh-phi-w2}
If $s>4$ and $\|h\|_{H^s}<c_s$ is sufficiently small then
\[
G(h)\phi=[T_{\sqrt{1+\alpha}}^2\partial_yw-T_{\nabla h}\nabla_xw-\nabla\cdot T_Vh-H(\nabla h,\nabla_xu)+\mathcal C_3]_\Gamma,
\]
where $\|\mathcal C_3|_\Gamma\|_{H^{s+1/2}}\lesssim_s$ the right-hand side of (\ref{C2-bound}).
\end{proposition}
\begin{proof}
By the display after (B.54) of \cite{DeIoPaPu}, the left-hand side equals
\[
T_{1+\alpha}\partial_yw-T_{\nabla h}\nabla_xw-T_{\nabla_xu}\nabla h+T_{\nabla h}T_{\partial_yu}\nabla h-H(\nabla h,\nabla_xu)+\mathcal C_3',
\]
where
\[
\mathcal C_3'=T_{1+\alpha}T_{\partial_y^2u}h+T_{\partial_yu}\alpha
-2T_{\nabla h}T_{\partial_yu}\nabla h-T_{\nabla h}T_{\nabla_x\partial_yu}h
+H(\alpha,\partial_yu).
\]
Recall $B=[\partial_yu]_\Gamma$ and $V=[\nabla_xu-B\nabla h]_\Gamma$ are horizontal and vertical components of the boundary velocity.
Then we obtain the claimed identity, with
\begin{align*}
\mathcal C_3&=\mathcal C_3'+T_{\nabla\cdot V}h+E(\nabla h,\partial_yu)\nabla h-E_1(\sqrt{1+\alpha}-1,\sqrt{1+\alpha}-1)\partial_yw\\
&=(T_{1+\alpha}T_{\partial_y^2u}-T_{\nabla h}T_{\nabla_x\partial_yu}+T_{\nabla\cdot V})h+T_{\partial_yu}\alpha-2T_{\nabla h}T_{\partial_yu}\nabla h+E(\nabla h,\partial_yu)\nabla h\\
&+H(\alpha,\partial_yu)-E_1(\sqrt{1+\alpha}-1,\sqrt{1+\alpha}-1)\partial_yw,\\
\mathcal C_3|_\Gamma&=(E(\alpha,\partial_y^2u|_\Gamma)-E(\nabla h,\nabla B))h
+2E_1(B,\nabla h)\nabla h-E_1(\nabla h,B)\nabla h\\
&+T_BH(\nabla h,\nabla h)+H(\alpha,B)-E_1(\sqrt{1+\alpha}-1,\sqrt{1+\alpha}-1)(B+T_{\partial_y^2u|_\Gamma}h).
\end{align*}
Note that we have used the fact that $h$ is independent of $y$,
and the fact that the elliptic equation in (\ref{u-Diri}) satisfied by $u$ takes the form
\[
(1+\alpha)\partial_y^2u-\nabla h\cdot\nabla B+\nabla\cdot V=0
\]
when restricted to $\Gamma$. By Lemma \ref{Lmq=Lq}, Lemma \ref{Taf-Lp} (ii), Lemma \ref{PkH-Lp} (iii), Lemma \ref{Eaf-Lp} (ii) and Lemma \ref{E1f-Lp} (ii),
\begin{align*}
\|\mathcal C_3|_\Gamma\|_{H^{s+1/2}}
&\lesssim_s \|\nabla h\|_{W^{2,\infty}}(\|\partial_y^2u\|_{L_y^\infty W_x^{1,\infty}}+\|B\|_{W^{2,\infty}})\|h\|_{H^{s-1/2}}\\
&+\|\nabla h\|_{W^{2,\infty}}^2(\|B\|_{H^{s-3/2}}
+\|\partial_y^2u\|_{L_{x,y}^\infty}\|h\|_{H^{s-3/2}}).
\end{align*}
Then the desired bound follows from Proposition \ref{u-Hs} (ii) and Proposition \ref{u-Cr} (ii). Note that the last term is dominated by the first one.
\end{proof}

Now we can paralinearize the Zakharov system (\ref{Zakharov}).
Recall $\lambda$ from (\ref{l-def}).
\begin{proposition}
If $s>4$ and $\|h\|_{H^s}<c_s$ is sufficiently small then
\[
h_t=T_\lambda(w|_\Gamma)-\nabla\cdot T_Vh+\mathcal Q|_\Gamma+\mathcal S_h+\mathcal C_h,
\]
where $\mathcal Q$ is given by (\ref{QS}),
\[
\mathcal S_h=\mathcal S|_\Gamma-H(\nabla h,\nabla\phi)
\]
and $\|\mathcal C_h\|_{H^{s+1/2}}\lesssim_s$ the right-hand side of (\ref{C2-bound}).
\end{proposition}
\begin{proof}
Let $\tilde\alpha = \sqrt{1 + \alpha} - 1$.
Putting (\ref{Taw}) into Proposition \ref{Gh-phi-w2} gives
\begin{equation}\label{Gh-phi-w3}
\begin{aligned}
G(h)\phi&=[T_{\sqrt{1+\alpha}}((\mathcal{A+B})w+\mathcal{Q+S+C}_3)
-T_{\nabla h}\nabla_xw-\nabla\cdot T_Vh+\mathcal C_2\\
&-H(\nabla h,\nabla_xu)]_\Gamma\\
&=[T_{\sqrt{1+\alpha}}(\mathcal{A+B})w-T_{\nabla h}\nabla_xw-\nabla\cdot T_Vh+\mathcal Q+\mathcal S_h+\mathcal C_{h1}]_\Gamma,\\
\mathcal C_{h1}&=\mathcal C_2+T_{\tilde\alpha}(\mathcal{Q+S})
+T_{\sqrt{1+\alpha}}\mathcal C_3.
\end{aligned}
\end{equation}
By Lemma \ref{Taf-Lp} (ii), Proposition \ref{QS-bound} and (\ref{C2-bound}), $\mathcal C_{h1}|_\Gamma$ satisfies the desired bound.

We also have
\begin{equation}\label{TaB}
\begin{aligned}
T_{\sqrt{1+\alpha}}\mathcal Bw
&=T_{b\sqrt{1+\alpha}}w+\frac i2T_{\{\tilde\alpha,b\}}w+E_1(\tilde\alpha,b)w
=T_{b\sqrt{1+\alpha}}w+\frac i2T_{\{\tilde\alpha,b^{(1)}\}\varphi_{\ge0}(\zeta)}w+\mathcal C_{h2},\\
\mathcal C_{h2}&=\frac i2T_{\{\tilde\alpha,b^{(1)}\}\varphi_{<0}(\zeta)
+\{\tilde\alpha,b^{(0)}\}}w+E_1(\tilde\alpha,b)w.
\end{aligned}
\end{equation}
Note that near $\zeta=0$, the operator in $\mathcal C_{h2}$ is $\frac i2T_{\{\tilde\alpha,b^{(1)}\}}+E_1(\tilde\alpha,b)=E(\tilde\alpha,b)$,
which vanishes to degree 1.
Hence by Lemma \ref{Taf-Lp} (ii), Lemma \ref{E1f-Lp} (ii) and Proposition \ref{u-Hs} (ii) applied to (\ref{w-def}) and (\ref{ab-soln}),
$\mathcal C_{h2}|_\Gamma$ satisfies the desired bound. Also, since
\begin{align*}
T_{\sqrt{1+\alpha}}\mathcal A-T_{\nabla h}\circ\nabla_x
&=T_{i\nabla h\cdot\zeta}-T_{\nabla h}\circ\nabla_x+iT_{a^{(0)}\sqrt{1+\alpha}}-\frac12T_{\{\tilde\alpha,a\}}+iE_1(\tilde\alpha,a)\\
&=\frac12T_{\Delta h}+iT_{a^{(0)}\sqrt{1+\alpha}}-\frac12T_{\{\tilde\alpha,a\}}+iE_1(\tilde\alpha,a),
\end{align*}
we have
\begin{equation}\label{TaA}
\begin{aligned}
(T_{\sqrt{1+\alpha}}\mathcal A-T_{\nabla h}\nabla_x)w
&=iT_{a^{(0)}\sqrt{1+\alpha}}w+\frac12T_{(\Delta h-\{\tilde\alpha,a^{(1)}\})\varphi_{\ge0}(\zeta)}w+\mathcal C_{h3},\\
\mathcal C_{h3}&=\frac12T_{(\Delta h-\{\tilde\alpha,a^{(1)}\})\varphi_{<0}(\zeta)-\{\tilde\alpha,a^{(0)}\}}w+iE_1(\tilde\alpha,a)w.
\end{aligned}
\end{equation}
For the same reason, $\mathcal C_{h3}|_\Gamma$ satisfies the desired bound.

From the third line of (\ref{ab-soln}), we know that the second term in $T_{\sqrt{1+\alpha}}\mathcal Bw$ and the first term in $T_{\sqrt{1+\alpha}}\mathcal A-T_{\nabla h}\nabla_xw$ cancel,
so putting (\ref{TaB}) and (\ref{TaA}) in (\ref{Gh-phi-w3}) we get
\begin{align*}
G(h)\phi &= \left[ T_{b\sqrt{1+\alpha}}w + \mathcal C_{h2}
+\frac12T_{(\Delta h-\{\tilde\alpha,a^{(1)}\})\varphi_{\ge0}(\zeta)}w+\mathcal C_{h3}-\nabla\cdot T_Vh+\mathcal Q+\mathcal S_h+\mathcal C_{h1} 
\right]_\Gamma\\
&=T_\lambda(w|_\Gamma)-\nabla\cdot T_Vh+\mathcal Q|_\Gamma+\mathcal S_h+\mathcal C_h,
\end{align*}
where $\mathcal C_h = [\sum_{j=1}^3 \mathcal C_{hj}]_\Gamma$ satisfies the desired bound.
\end{proof}


We now show the paralinearization of the second equation in (\ref{Zakharov}).
\begin{proposition}
\begin{align*}
w_t|_\Gamma&=-T_ah-T_V\cdot\nabla(w|_\Gamma)+\mathcal S_w+\mathcal C_w,\\
a&=1+B_t+V\cdot\nabla B,\\
\mathcal S_w&=\frac12(H(B,B)-H(V,V)).
\end{align*}
If $s>5/2$ and $\|\nabla h\|_{H^{s-1}\cap H^5}<c_s$ is sufficiently small, then
\[
\|\mathcal C_w\|_{H^{s+1/2}}
\lesssim_s \||\nabla|^{1/2}\phi\|_{C_*^3}(\|\nabla h\|_{W^{2,\infty}}\||\nabla|^{1/2}\phi\|_{H^{s-1}}+\||\nabla|^{1/2}\phi\|_{C_*^4}\|h\|_{H^{s-3/2}}).
\]
\end{proposition}
\begin{proof}
By (\ref{w-def}), (\ref{Zakharov}), (\ref{phi-t}) and (\ref{Gh-phi}),
\begin{align*}
w_t|_\Gamma&=\phi_t-T_{B_t}h-T_Bh_t\\
&=-T_{1+B_t}h+\frac12(B^2-2BV\cdot\nabla h-|V|^2)-T_BB+T_B(\nabla h\cdot V)\\
&=-T_{1+B_t}h-T_{\nabla h\cdot V}B-T_V\cdot V+\mathcal S_w-H(B,\nabla h\cdot V).
\end{align*}
By Definition \ref{E1f-def} and (\ref{BV-u}),
\begin{align*}
T_{\nabla h\cdot V}B+T_VV&=T_V\cdot(T_{\nabla h}B+V)-E_1(V,\nabla h)B\\
&=T_V\cdot\nabla\phi-E_1(V,\nabla h)B-T_V\cdot T_B\nabla h-T_V\cdot H(B,\nabla h)\\
&=T_V\cdot\nabla(w|_\Gamma)+T_V\cdot T_{\nabla B}h-E_1(V,\nabla h)B-T_V\cdot H(B,\nabla h)\\
&=T_V\cdot\nabla(w|_\Gamma)+T_{V\cdot\nabla B}h+E_1(V,\nabla B)h-E_1(V,\nabla h)B-T_V\cdot H(B,\nabla h)
\end{align*}
so
\[
w_t|_\Gamma=-T_ah-T_V\cdot\nabla(w|_\Gamma)+\mathcal S_w+\mathcal C_w
\]
where
\[
\mathcal C_w=E_1(V,\nabla h)B-E_1(V,\nabla B)h+T_V\cdot H(B,\nabla h)
-H(B,V\cdot\nabla h).
\]

By Lemma \ref{Taf-Lp} (ii), Lemma \ref{PkH-Lp} (iii) and Lemma \ref{E1f-Lp} (ii),
\[
\|\mathcal C_w\|_{H^{s+1/2}}
\lesssim_s \|V\|_{W^{2,\infty}}(\|\nabla h\|_{W^{2,\infty}}\|B\|_{H^{s-3/2}}+\|\nabla B\|_{W^{2,\infty}}\|h\|_{H^{s-3/2}}).
\]
Then the result follows from Proposition \ref{u-Hs} (ii) and Proposition \ref{u-Cr} (ii).
\end{proof}

\subsection{Taylor expansion}\label{Taylor}
Here we bound the remainders of various Taylor expansions.
Recall from (\ref{a-def}) that $a=1+B_t+V\cdot\nabla B$,
so its first order approximation is $a\approx 1+|\nabla|\phi_t\approx 1-|\nabla|h$, and $a_t\approx -|\nabla|h_t=\Delta\phi$.
\begin{proposition}
(i) If $\|\nabla h\|_{H^{r+2}}<c_r$ is sufficiently small then
\begin{equation}
\|(G(h)\phi-|\nabla|\phi,B-|\nabla|\phi,V-\nabla\phi)\|_{C_*^r}
\lesssim_r \|h\|_{C_*^{r+1}}\||\nabla|^{1/2}\phi\|_{C_*^{r+3/2}}.
\end{equation}

(ii) If $\|\nabla h\|_{H^{r+2}}<c_r$ is sufficiently small then
\begin{align}
\|a-1\|_{C_*^r}&\lesssim_r \||\nabla|^{1/2}\phi\|_{C_*^{r+3/2}}^2+\|h\|_{C_*^{r+1}},\\
\|(a-1+|\nabla|h,\sqrt a-1+|\nabla|h/2)\|_{C_*^r}&\lesssim_r \||\nabla|^{1/2}\phi\|_{C_*^{r+3/2}}^2+\|h\|_{C_*^{r+1}}\||\nabla|^{1/2}h\|_{C_*^{r+3/2}}.
\end{align}

(iii) If $\|\nabla h\|_{H^{r+3}}<c_r$ is sufficiently small then
\begin{equation}
\|(a_t-\Delta\phi,\partial_t\sqrt a-\Delta\phi/2)\|_{C_*^r}
\lesssim_r (\||\nabla|^{1/2}\phi\|_{C_*^{r+5/2}}^2+\|h\|_{C_*^{r+2}})\||\nabla|^{1/2}\phi\|_{C_*^{r+5/2}}.
\end{equation}
\end{proposition}
\begin{proof}
(i) follows from (\ref{Gh-phi-int}), (\ref{ds-Gh-phi}) and (\ref{Gh-Cr}).
The first line of (ii) follows from the second one. For the second line,
we only show the bound of $a-1+|\nabla|h$, the other one being similar.

We first show an expression for $a$.
By (\ref{Zakharov}), (\ref{phi-t}) and (\ref{Gh-phi}),
\begin{equation}\label{phi-t-Bh-t}
\phi_t-Bh_t=-h+\frac12(B^2-2BV\cdot\nabla h-|V|^2)-B(B-\nabla h\cdot V)
=-h-\frac12(B^2+|V|^2)
\end{equation}
so by (A.3.10) of \cite{AlDe2},
\begin{equation}\label{dt-Gh-phi}
\partial_tG(h)\phi=G(h)[\phi_t-Bh_t]-\nabla\cdot(h_tV)
=-\frac12G(h)[B^2+|V|^2+2h]-\nabla\cdot(h_tV).
\end{equation}
For $V$ we have, by (\ref{BV-u}), (\ref{phi-t}), (\ref{ht=Gh-phi}) and (\ref{Gh-phi}),
\begin{equation}\label{Vt}
\begin{aligned}
V_t&=\nabla\phi_t-B_t\nabla h-B\nabla h_t\\
&=-\nabla h+B\nabla B-V\cdot\nabla V-(V\cdot\nabla h)\nabla B-B\nabla(V\cdot\nabla h)-B_t\nabla h-B\nabla B+B\nabla(V\cdot\nabla h)\\
&=-\nabla h-V\cdot\nabla V-(V\cdot\nabla h)\nabla B-B_t\nabla h\\
&=-V\cdot\nabla V-a\nabla h+(V\cdot\nabla B)\nabla h-(V\cdot\nabla h)\nabla B.
\end{aligned}
\end{equation}
For $B$ we have, by (\ref{Gh-phi}), (\ref{dt-Gh-phi}), (\ref{ht=Gh-phi}) and (\ref{Vt}),
\begin{equation}\label{Bt}
\begin{aligned}
B_t&=\partial_tG(h)\phi+(\nabla h\cdot V)_t\\
&=-\frac12G(h)[B^2+|V|^2+2h]-(G(h)\phi)\nabla\cdot V\\
&-\nabla h\cdot(V\cdot\nabla V+a\nabla h+(V\cdot\nabla B)\nabla h-(V\cdot\nabla h)\nabla B).
\end{aligned}
\end{equation}
Recall from (\ref{a-def}) that $a=1+B_t+V\cdot\nabla B$, so
\begin{align}
\label{a-expr}
a&=\frac1{1+|\nabla h|^2}\left( 1+\tilde a-\frac12G(h)[B^2+|V|^2+2h] \right),\\
\label{tld-a-def}
\tilde a&=V\cdot\nabla B-(G(h)\phi)\nabla\cdot V-\nabla h\cdot(V\cdot\nabla V)-(V\cdot\nabla B)|\nabla h|^2+(V\cdot\nabla h)(\nabla B\cdot\nabla h).
\end{align}
By (\ref{Gh-Cr}), if $\|\nabla h\|_{H^{r+2}}<c_r$ is sufficiently small then
\begin{align}
\label{tilde-a-Cr2}
\|\tilde a\|_{C_*^r}&\lesssim_r \||\nabla|^{1/2}\phi\|_{C_*^{r+3/2}}^2,\\
\label{Gh-VB-Cr2}
\|G(h)[|V|^2+B^2]\|_{C_*^r}&\lesssim_r \||V|^2+B^2\|_{C_*^{r+1}}
\lesssim_r \||\nabla|^{1/2}\phi\|_{C_*^{r+3/2}}^2.
\end{align}
By (\ref{a-expr}), (\ref{tilde-a-Cr2}) and (\ref{Gh-VB-Cr2}) we get (ii).

(iii) Similarly we only show the bound for $a_t-\Delta\phi$. We differentiate (\ref{a-expr}) with respect to $t$, and note that
\begin{equation}\label{Bt}
B_t=a-1-V\cdot\nabla B.
\end{equation}
By (ii) and (\ref{Vt}), if $\|\nabla h\|_{H^{r+2}}<c_r$ is sufficiently small then
\begin{align}
\label{Bt-Vt-Cr}
\|B_t\|_{C_*^r}+\|V_t\|_{C_*^r}
&\lesssim_r \||\nabla|^{1/2}\phi\|_{C_*^{r+3/2}}^2+\|h\|_{C_*^{r+1}},\\
\label{BBt-VVt-Cr}
\|\partial_t(|V|^2+B^2)\|_{C_*^r}
&\lesssim_r (\||\nabla|^{1/2}\phi\|_{C_*^{r+3/2}}^2+\|h\|_{C_*^{r+1}})\||\nabla|^{1/2}\phi\|_{C_*^{r+1/2}}.
\end{align}
Hence if $\|\nabla h\|_{H^{r+3}}<c_r$ is sufficiently small then
\begin{equation}\label{tilde-a-t-Cr}
\|\partial_t\tilde a\|_{C_*^r}
\lesssim_r (\||\nabla|^{1/2}\phi\|_{C_*^{r+5/2}}^2+\|h\|_{C_*^{r+2}})\||\nabla|^{1/2}\phi\|_{C_*^{r+3/2}}.
\end{equation}
Now we bound $\partial_tG(h)[|V|^2+B^2]$.
By (\ref{Gh-Cr}), if $\|\nabla h\|_{H^{r+1}}<c_r$ is sufficiently small then
\begin{align}\label{ht-Cr}
\|B(h)h\|_{C_*^r}+\|V(h)h\|_{C_*^r}&\lesssim_r \|h\|_{C_*^{r+1}}
\end{align}
and if $\|\nabla h\|_{H^{r+2}}<c_r$ is sufficiently small then
\begin{equation}\label{BV-V2-B2}
\|(B(h),V(h))[|V|^2+B^2]\|_{C_*^r}\lesssim_r \||V|^2+B^2\|_{C_*^{r+1}}\lesssim_r \||\nabla|^{1/2}\phi\|_{C_*^{r+3/2}}^2.
\end{equation}
Then by (\ref{Gh-Cr}) and (\ref{ht-Cr}), under the same condition we have
\begin{equation}\label{dt-Gh-h-Cr}
\|\partial_tG(h)h\|_{C_*^r}\lesssim_r \|(h_t,h_tB(h)h,h_tV(h)h)\|_{C_*^{r+1}}\lesssim_r \||\nabla|^{1/2}\phi\|_{C_*^{r+3/2}}.
\end{equation}
By (\ref{dt-Gh-phi}), (\ref{Gh-Cr}), (\ref{BBt-VVt-Cr}) and (\ref{BV-V2-B2}), if $\|\nabla h\|_{H^{r+3}}<c_r$ is sufficiently small then
\begin{align}
\nonumber
\|\partial_tG(h)[|V|^2+B^2]\|_{C_*^r}
&\lesssim_r \|(\partial_t(|V|^2+B^2),(h_tB(h),h_tV(h))[|V|^2+B^2]\|_{C_*^{r+1}}\\
&\lesssim_r\text{right-hand side of (\ref{tilde-a-t-Cr})}.
\label{dt-Gh-VB-Cr2}
\end{align}
By (\ref{a-expr}), (\ref{tilde-a-t-Cr}), (\ref{dt-Gh-VB-Cr2}), (\ref{Gh-Cr}), (\ref{dt-Gh-h-Cr}) and (ii), under the same condition we have
\begin{align*}
\|\partial_t(a+G(h)h)\|_{C_*^r}
&\lesssim_r \left\| \frac1{1+|\nabla h|^2} \right\|_{C_*^r} \left\| \partial_t\tilde a-\frac12\partial_tG(h)[|V|^2+B^2]
+\partial_t(|\nabla h|^2G(h)h) \right.\\
&\left. -\frac{2a\nabla h\cdot\nabla h_t}{1+|\nabla h|^2} \right\|_{C_*^r}
\lesssim_r\text{right-hand side of (\ref{tilde-a-t-Cr})}.
\end{align*}
By (\ref{Gh-Cr}), (\ref{dt-Gh-phi}) and (\ref{ht-Cr}), if $\|\nabla h\|_{H^{r+2}}<c_r$ is sufficiently small then
\[
\|\partial_tG(h)h-G(h)^2\phi\|_{C_*^r}=\|\partial_tG(h)h-G(h)h_t\|_{C_*^r}\lesssim_r \|h\|_{C_*^{r+2}}\||\nabla|^{1/2}\phi\|_{C_*^{r+3/2}}.
\]
By (\ref{Gh-Cr}) and (i), if $\|\nabla h\|_{H^{r+3}}<c_r$ is sufficiently small then
\begin{align*}
\|G(h)^2\phi+\Delta\phi\|_{C_*^r}
&\le\|G(h)[G(h)\phi-|\nabla|\phi]\|_{C_*^r}+\|(G(h)-|\nabla|)|\nabla|\phi\|_{C_*^r}\\
&\lesssim_r \|h\|_{C_*^{r+2}}\||\nabla|^{1/2}\phi\|_{C_*^{r+5/2}}.
\end{align*}
Hence under the same condition we obtain (iii).
\end{proof}

\subsection{Estimates on the norms of multipliers}\label{MulEst}
Here we bound the norms of the multipliers $s$, $p$, $q$, $\tilde q$ and $m_{\mu\nu}$,
defined in (\ref{s-mult}), (\ref{p-mult}), (\ref{q-mult}), (\ref{til-q-mult}) and (\ref{m-prod}) respectively.
\begin{lemma}
(i) For $k_1$, $k_2$, $k_3\in\Z$ we have
\begin{align*}
\|p\|_{S^\infty_{k_1,k_2;k_3}}&\lesssim 2^{-\max(k_1,k_2)}, &
\|s\|_{S^\infty_{k_1,k_2;k_3}}&\lesssim 2^{3\max(k_1,k_2)/2}.
\end{align*}

(ii) For $L\ge0$ we have
\[
|\nabla^L\Phi_{\mu\nu}^{-1}|\lesssim_L \min(|\xi_1|,|\xi_2|,|\xi_1+\xi_2|)^{-L-1/2}.
\]

(iii) For $k_1$, $k_2$, $k_3\in\Z$ we have
\[
\|\Phi_{\mu\nu}^{-1}\|_{S^\infty_{k_1,k_2;k_3}}\lesssim 2^{-\min(k_1,k_2,k_3)/2}.
\]

(iii') If in addition to (iii) we have $k_1\le k_2-3$ and $\nu=-$ then
\[
\|\Phi_{\mu\nu}^{-1}\|_{S^\infty_{k_1,k_2;k_3}}\lesssim 2^{-k_2/2}.
\]
\end{lemma}
\begin{proof}
(i) To bound $p$, after a change of variables it suffices to show that $\|(|\xi_1|+|\xi_2|)^{-1}\|_{S_{k_1,k_2;k_3}^\infty}\lesssim 2^{-\max(k_2,k_3)}$.
Since $|\nabla^L(|\xi_1|+|\xi_2|)^{-1}|\lesssim_L(|\xi_1|+|\xi_2|)^{-L-1}$,
Lemma \ref{Soo-Cn} (ii) gives a bound of $2^{-\max(k_1,k_2)}$.
To finish the proof note that $\max(k_1,k_2)=\max(k_2,k_3)+O(1)$.

The bound for $s$ follows from that for $p$, $\|\varphi_k(\xi)|\xi|\|_{S^\infty}\lesssim 2^k$ and Lemma \ref{Soo-Cn} (i).

(ii) We first show a lower bound of $\Phi_{\mu\nu}$. Indeed, by convexity of $|\xi|^{1/2}$ we get
\begin{equation}\label{1/F-C0}
|\Phi_{\mu\nu}|\gtrsim \min(|\xi_1|,|\xi_2|,|\xi_1+\xi_2|)^{1/2}.
\end{equation}
This shows the claim for $L=0$. For $L>0$ we use induction on $L$,
the identity $\nabla^L(\Phi_{\mu\nu}\Phi_{\mu\nu}^{-1})=0$.
and the bound $|\nabla^L\Phi_{\mu\nu}|\lesssim_L \min(|\xi_1|,|\xi_2|,|\xi_1+\xi_2|)^{1/2-L}$.

(iii) and (iii') Without loss of generality we assume $k_1\le k_2\le k_3$.

{\bf Case 1:} $k_1\ge k_2-2$. Then the bound follows from (ii) and Lemma \ref{Soo-Cn} (ii). It can also be derived from the homogeneity of $|\xi|^{1/2}$.

{\bf Case 2:} $k_1\le k_2-3$. We still have $|\xi_1|^L|\nabla^L_{\xi_1}\Phi_{\mu\nu}|\lesssim_L |\xi_1|^{1/2}\lesssim |\Phi_{\mu\nu}|$ ($L > 0$).
For $\nabla^L_{\xi_2}\Phi_{\mu\nu}$ ($L > 0$) we further have two cases.

{\bf Case 2.1:} $\nu=+$. Since $\nabla^L_{\xi_2}\Phi_{\mu+}=\nabla^L\Lambda(|\xi_1+\xi_2|)-\nabla^L\Lambda(|\xi_2|)$ we have
\[
|\nabla^L_{\xi_2}\Phi_{\mu+}|
\le |\xi_1|\sup_{t\in[0,1]} |\nabla\nabla^L\Lambda(|t\xi_1+\xi_2|)|
\lesssim_L |\xi_1||\xi_2|^{-L-1/2}
\]
so $|\xi_2|^L|\nabla^L_{\xi_2}\Phi_{\mu+}|\lesssim_L |\xi_1|^{1/2}$.
By (\ref{1/F-C0}) and induction on $L$ then, $|\xi_2|^L|\nabla^L_{\xi_2}\Phi_{\mu+}^{-1}|\lesssim_L |\xi_1|^{-1/2}$, and the bound follows from Lemma \ref{Soo-Cn} (ii).

{\bf Case 2.2:} $\nu=-$. Then $|\xi_2|^L|\nabla_{\xi_2}^L\Phi_{\mu-}|\lesssim_L |\xi_2|^{1/2}\lesssim\Phi_{\mu-}$,
so by (\ref{1/F-C0}) and Lemma \ref{Soo-Cn} (ii), the bound can be improved to $\|\phi_{\mu-}^{-1}\|_{S^\infty_{k_1,k_2;k_3}}\lesssim 2^{-k_3/2}$.
\end{proof}

\begin{lemma}
For $k_1$, $k_2$, $k_3\in\Z$ we have
\[
\|q\|_{S^\infty_{k_1,k_2;k_3}}\lesssim 2^{2k_1+k_2}(2^{2(k_1-k_2)}+1_{k_2\le2})1_{k_1\le k_2-6}
\lesssim 2^{2(k_1+k_1^+)-|k_2|}1_{k_1\le k_2-6}.
\]
\end{lemma}
\begin{proof}
The factor $1_{k_1\le k_2-6}$ comes from the $\varphi_{\le10}$ factor and is assumed to be nonzero thereafter. Since $||\xi_1+\xi_2|-|\xi_2||\le|\xi_1|\lesssim 2^{k_1}$, $|\xi_1+2\xi_2|\approx|\xi_2|$ and $|\xi_1+\xi_2/2|>|\xi_2|/4$, it suffices to show that $D:=|\xi_1+\xi_2||\xi_2|-(\xi_1+\xi_2)\cdot\xi_2$ satisfies $\|D\|_{S^\infty_{k_1,k_2;k_3}}\approx 2^{2k_1}$. Indeed, this follows from the identity
\[
D=\frac{\det(\xi_1+\xi_2,\xi_2)^2}{|\xi_1+\xi_2||\xi_2|+(\xi_1+\xi_2)\cdot\xi_2}
=\frac{\det(\xi_1,\xi_2)^2}{|\xi_1+\xi_2||\xi_2|+(\xi_1+\xi_2)\cdot\xi_2}.
\]
\end{proof}

\begin{lemma}
For $k_1$, $k_2$, $k_3\in\Z$ we have
\begin{align*}
\|m_{\mu\nu}\|_{S^\infty_{k_1,k_2;k_3}}&\lesssim 2^{(k_3+\max k_j+\min k_j)/2},\\
\|m_{\mu\nu}/\Phi_{\mu\nu}\|_{S^\infty_{k_1,k_2;k_3}}&\lesssim 2^{(k_3+\max k_j)/2}.
\end{align*}
\end{lemma}
\begin{proof}
The second bound follows from the first one, Lemma \ref{Soo-Cn} (i) and Lemma \ref{ps-1/F-Soo} (iii). Now we show the first bound. Since $m_{\mu\nu}$ is homogeneous of total degree 3/2, if $k_1$, $k_2$ and $k_3$ are within $O(1)$ of each other the bound is trivial. Now we assume the contrary. Recall that $m_{\mu\nu}$ is a linear combination of
\begin{align*}
m_1 &= \frac{|\xi_1+\xi_2||\xi_2|-(\xi_1+\xi_2)\cdot\xi_2}{\sqrt{|\xi_2|}}, &
m_2 &= \sqrt{|\xi_1+\xi_2|}\frac{|\xi_1||\xi_2|+\xi_1\cdot\xi_2}{\sqrt{|\xi_1||\xi_2|}}
\end{align*}
whose asymptotics are (see also Section 3 of \cite{GeMaSh2})
\begin{itemize}
\item If $|\xi_2|/|\xi_1|$ is sufficiently small, then $m_1=\sqrt{|\xi_2|}\cdot(|\xi_1+\xi_2|-(\xi_1+\xi_2)\cdot\xi_2/|\xi_2|)$, so $m_1/\sqrt{|\xi_2|}$ is a smooth function of $\xi_2/|\xi_2|$ and $\xi_1$.

\item If $|\xi_1+\xi_2|/|\xi_1|$ is sufficiently small, then similarly $m_1/|\xi_1+\xi_2|$ is a smooth function of $(\xi_1+\xi_2)/|\xi_1+\xi_2|$ and $\xi_2$.

\item If $|\xi_1|/|\xi_2|$ is sufficiently small, then
\[
\frac{|\xi_1+\xi_2||\xi_2|-(\xi_1+\xi_2)\cdot\xi_2}{|\xi_1|^2}
=\frac{\det(\xi_1/|\xi_1|,\xi_2)^2}{|\xi_1+\xi_2||\xi_2|+(\xi_1+\xi_2)\cdot\xi_2}
\]
is a smooth function of $\xi_1/|\xi_1|$, $|\xi_1|$ and $|\xi_2|$,
so $m_1/|\xi_1|^2$ is also such a function.
\end{itemize}

Similarly for $m_2$ we have
\begin{itemize}
\item If $|\xi_1|/|\xi_2|$ is sufficiently small, then $m_2/\sqrt{|\xi_1|}=\sqrt{|\xi_1+\xi_2|}(|\xi_2|+\xi_2\cdot\xi_1/|\xi_1|)$ is a smooth function of $\xi_2$, $\xi_1/|\xi_1|$ and $|\xi_1|$.

\item If $|\xi_2|/|\xi_1|$ is sufficiently small, then $m_2/\sqrt{|\xi_2|}$ is a smooth function of $\xi_1$, $\xi_2/|\xi_2|$ and $|\xi_2|$.

\item If $|\xi_1+\xi_2|/|\xi_1|$ is sufficiently small, then $m_2/|\xi_1+\xi_2|^{5/2}$ is a smooth function of $(\xi_1+\xi_2)/|\xi_1+\xi_2|$, $|\xi_1+\xi_2|$ and $\xi_1$.
\end{itemize}

Now it is easy to verify the desired bounds.
\end{proof}

\section{Appendix: Remainder estimates of $G(h)\phi$}
In this section we show weighted Sobolev estimates for the remainder of the Taylor expansion of $G(h)\phi$. The argument closely follows Appendix F of \cite{GeMaSh2}, with necessary modifications to adapt to the weights in the norms. We first introduce some notation for the anticipated types of bounds.

\begin{definition}\label{O-hn-g-def}
Let $n \ge 1$ be an integer. We say $f = O_Z(h^n \cdot g)$ if $f$ satisfies
\[
\|f\|_Z\lesssim_\alpha \|h\|_{W^{2,\infty}}^{n-1}\|h\|_{W^{11,\infty}}\|g\|_Z.
\]
Let $k \ge 0$ be another integer, $1 < p < \infty$, $1 < q \le \infty$ and
$1/r = 1/p + 1/q$. We say $f = O_W(h^n \cdot g)$ if $f$ satisfies
\[
\|f\|_{W^{k,r}}\lesssim \|h\|_{W^{2,\infty}}^{n-1}\|h\|_{W^{k+3,q}}\|g\|_{W^{k,p}}.
\]
We say $f=O(h^n \cdot g)$ if both $f=O_Z(h^n \cdot g)$ and $f=O_{W;k,p,q}(h^n \cdot |\nabla|^{1/2}g)$.
\end{definition}
\begin{remark}
If the implicit constant depends on any other variable,
they are listed in subscripts.
\end{remark}

We need to bound the $Z$ norm of a product, as done in Lemma F.3 of \cite{GeMaSh2}.
\begin{lemma}\label{prod-Z}
Let $\alpha \in (0, 1)$ and $A$, $B_1, \dots, B_n \ge 0$ be integers. Then
\begin{align*}
&\|(1+|x|)^\alpha\nabla^A|\nabla|f\nabla^{B_1}g_1\cdots\nabla^{B_n}g_n\|_{L^2}\\
\le&C_{A+B_1+\dots+B_n}^n\||\nabla|^{1/2}f\|_{W^{A+B_1+\cdots+B_n+1,\infty}}\|(1+|x|)^\alpha g_1\|_{L^2}\|g_2\|_{L^\infty}\cdots\|g_n\|_{L^\infty}\\
+&\|(1+|x|)^\alpha|\nabla|^{1/2}f\|_{L^2}\sum_{i=1}^n \|g_i\|_{W^{A+B_1+\cdots+B_n+1,\infty}}\prod_{j\neq i}\|g_j\|_{L^\infty}.
\end{align*}
\end{lemma}
\begin{proof}
We only show the case $n = 1$, $B_1 = B$; the general case being similar.

We decompose
\[
\nabla^A|\nabla|f\nabla^Bg
=\sum_{k\in\Z} \nabla^A|\nabla|P_kf\nabla^BP_{<k}g+\nabla^A|\nabla|P_{\le k}f\nabla^BP_kg.
\]
For the first term, by Lemma \ref{Z-bound} (iii) we have
\begin{align*}
&\left\| (1+|x|)^\alpha\sum_{k\in\Z} \nabla^A|\nabla|P_kf\nabla^BP_{<k}g \right\|_{L^2}\\
\le&\sum_{k\in\Z} \|(1+|x|)^\alpha\nabla^A|\nabla|P_kf\nabla^BP_{<k}g\|_{L^2}\\
\le&\sum_{k\in\Z} \|\nabla^A|\nabla|P_kf\|_{L^\infty}\|(1+|x|)^\alpha\nabla^BP_{<k}g\|_{L^2}\\
\lesssim_{A,B}&\sum_{k\in\Z} 2^{k^-/2-(B+1/2)k^+}\||\nabla|^{1/2}f\|_{W^{A+B+1,\infty}}\cdot2^{Bk}\|(1+|x|)^\alpha g\|_{L^2}\\
\lesssim_{A,B}&\||\nabla|^{1/2}f\|_{W^{A+B+1,\infty}}\|(1+|x|)^\alpha g\|_{L^2}.
\end{align*}
Similarly the second term is bounded by $\|(1+|x|)^\alpha|\nabla|^{1/2}f\|_{L^2}\|g\|_{W^{A+B+1}}$.
\end{proof}

Now we show the analog of Claim F.2 in \cite{GeMaSh2}. By (F.3) in \cite{GeMaSh2} we can write
\begin{equation}\label{single-layer}
\Phi(x,z)=\frac{1}{2\pi}\int_{\R^2} \frac{\rho(y)}{(|x-y|^2+|z-h(y)|^2)^{1/2}}dy
\end{equation}
on $\Omega(t)$. Then
\begin{equation}\label{gr-phi-Taylor}
|\nabla|\phi=\rho+\sum_{n=1}^\infty \frac{1}{2\pi}{-1/2 \choose n}K_n\rho,
\end{equation}
where
\begin{equation}\label{Kn-def}
K_n\rho(x)=|\nabla|\int_{\R^2} \rho(y)\frac{|h(x)-h(y)|^{2n}}{|x-y|^{2n+1}}dy.
\end{equation}
This can be inverted using Von Neumann's series to give
\begin{equation}\label{rho-Taylor}
\rho=|\nabla|\phi+\sum_{N=1}^\infty \sum_{m_1,\dots,m_N\ge1} c_{m_1,\dots,m_N}K_{m_1}\cdots K_{m_N}|\nabla|\phi
\end{equation}
where $|c_{m_1,\dots,m_N}| \le C^{m_1+\cdots+m_N}$. We now bound all the terms in this series.

As in (F.8) of \cite{GeMaSh2}, we decompose $K_n$ into two parts.
Let $\chi_{\ge1}$ and $\chi_{<1}$ be a partition of unity supported in $\R^2\backslash\overline{B(1)}$ and $B(2)$, respectively. For $\iota\in\{\ge1,<1\}$ let $\Gamma_{n,\iota}(x)=\chi_\iota(x)/|x|^{2n+1}$, and
\begin{equation}\label{Kni-def}
K_{n,\iota}\rho(x)=\nabla\int_{\R^2} \rho(y)|h(x)-h(y)|^{2n}\Gamma_{n,\iota}(x-y)dy.
\end{equation}
Then $K_n = -|\nabla|^{-1}\nabla\cdot(K_{n,\ge1} + K_{n,<1})$.
By Lemma \ref{Z-bound}, the operator $|\nabla|^{-1}\nabla$ is bounded on $W^{k,r}$ and $Z$ spaces, so it suffices to bound $K_{n,\ge1}$ and $K_{n,<1}$ separately.

\begin{lemma}\label{Kn<1-Wkp}
For any integer $n \ge 1$,
\[
K_{n,<1}\rho=C_{k,p,q}^nO_W(h^{2n} \cdot \rho).
\]
\end{lemma}
\begin{proof}
By Taylor's formula,
\begin{equation}\label{LR-def}
\begin{aligned}
h(y) - h(x) &= \underbrace{\nabla h(x) \cdot (y - x)}_{L(x,y)}\\
&+ \underbrace{\int_0^1 (1 - t)(y - x)^T\nabla^2h(x + t(y - x))(y - x)dt}_{R(x,y)}
\end{aligned}
\end{equation}
so $K_{n,<1} = K_{n,<1,L} + K_{n,<1,R}$, where
\begin{align*}
K_{n,<1,L}\rho(x)
&=\nabla\int_{\R^2} \Gamma_{n,<1}(x-y)(\nabla h(x)\cdot(x-y))^{2n}\rho(y)dy\\
&=\nabla\int_{\R^2} \Gamma_{n,<1}(y)(\nabla h(x)\cdot y)^{2n}\rho(x-y)dy.
\end{align*}

{\bf Part 1:} $K_{n,<1,L}\rho$. We only pass $k$ derivatives to the integrand to write $\nabla^kK_{n,<1,L}\rho$ as a sum of $2^{2n+k}$ terms like
\[
\nabla\left( \nabla_x^l(\partial_{I_1}h(x)\cdots\partial_{I_{2n}}h(x))
\int_{\R^2} \Gamma_{n,<1}(y)y_{I_1}\cdots y_{I_{2n}}\nabla^{k-l}\rho(x-y)dy \right),
\]
where $I_1, \dots, I_{2n} \in \{1, 2\}$. We put the derivatives of $(\nabla h)^{\otimes 2n}$ in $L^q$, and the integral and its derivative in $L^p$,
noting that the convolution kernel is homogeneous of degree $-1$ or $-2$ near the origin, so convolution with it is bounded on $L^p$ for $p\in(1,\infty)$.
Then by Sobolev multiplication,
\[
\|\nabla^kK_{n,<1,L}\rho\|_{L^r}\le C_k^n \|(\nabla h)^{\otimes 2n}\|_{W^{k+1,q}}\|\text{integral}\|_{L^p}\le C_{k,p,q}^n\|h\|_{W^{1,\infty}}^{2n-1}\|h\|_{W^{k+2,q}}\|\rho\|_{W^{k,p}}.
\]
Note that if one keeps track of the constants, one finds that it is at most exponential in $n$.

{\bf Part 2:} $K_{n,<1,R}\rho$. This is a sum of $2^k(4^n-1)$ terms like
\begin{align*}
&\nabla\int_{\R^2} \Gamma_{n,<1}(y)\underbrace{\nabla_x^l(L(x,x-y)^{2n-j}R(x,x-y)^j)}_{F(x,y)}\nabla^{k-l}\rho(x-y)dy\\
=&\nabla\int_{\R^2} \Gamma_{n,<1}(x - y)F(x, x - y)\nabla^{k-l}\rho(y)dy\\
=&\int_{\R^2} \nabla\Gamma_{n,<1}(x - y)F(x, x - y)\nabla^{k-l}\rho(y)dy\\
+&\int_{\R^2} \Gamma_{n,<1}(x - y)(\nabla_1F(x, x - y) + \nabla_2F(x, x - y))\nabla^{k-l}\rho(y)dy\\
=&\int_{\R^2} \nabla\Gamma_{n,<1}(y)F(x, y)\nabla^{k-l}\rho(x - y)dy\\
+&\int_{\R^2} \Gamma_{n,<1}(y)(\nabla_xF(x, y) + \nabla_yF(x, y))\nabla^{k-l}\rho(x - y)dy.
\end{align*}
Since $\Gamma_{n,<1}(y)|y|^{2n+1}$ and its gradient are in $L^1$,
with the norm polynomial in $n$, it suffices to bound, uniformly in $|y| \le 2$,
\[
|y|^{-2n-1}\|(F, \nabla F)(x, y)\nabla^{k-l}\rho(x - y)\|_{L_x^r}
\le |y|^{-2n-1}\|(F, \nabla F)(x, y)\|_{L_x^q}\|\rho\|_{W^{k,p}}.
\]
From the expressions (\ref{LR-def}) of $L$ and $R$ it follows that
\begin{align*}
L(x, x - y) &= -\nabla h(x) \cdot y,\\
R(x, x - y) &= \int_0^1 (1 - t)y^T\nabla^2h(x - ty)ydt
\end{align*}
so by the Kato--Ponce inequality (Lemma X.4 of \cite{KaPo}),
\begin{align*}
\|(F, \nabla F)(x, y)\|_{L_x^q}
&\le C\|L(x, x - y)^{2n-j}R(x, x - y)^j\|_{W^{k+1,q}}\\
&\le C_{k,q}^n\|h\|_{W^{2,\infty}}^{2n-1}\|h\|_{W^{k+3,q}}.
\end{align*}

Combining the two parts shows the claim.
\end{proof}

\begin{lemma}\label{Kn-Z}
For any integer $n \ge 1$,
\begin{align*}
K_n\rho &= C^nO_Z(h^{2n}\cdot\rho), &
K_n|\nabla|\phi &= C^nO_Z(U^{2n}\cdot U).
\end{align*}
\end{lemma}
\begin{proof}
As before it suffices to bound $K_{n,\ge1}$ and $K_{n,<1}$ separately.

{\bf Part 1:} $K_{n,\ge1}$. Using the binomial formula for $(h(x) - h(y))^{2n}$ we decompose $K_{n,\ge1}$ as a sum of $4^n$ terms like
\[
K_{n,\ge1,j}\rho=\nabla(h^j(\Gamma_{n,\ge1}*\rho h^{2n-j})).
\]
For $0 \le k \le 8$ we have $\nabla^kK_{n,\ge1,j}\rho$ is a sum of $2^{k+1}$ terms like
\begin{equation}\label{Kn-ge1-int}
\nabla^l(h^j)(\nabla^{k+1-l}\Gamma_{n,\ge1}*\rho h^{2n-j})
\end{equation}
so it suffices to bound (\ref{Kn-ge1-int}). We have
\[
|(1+|\cdot|)^\alpha f*g|
\lesssim|(1+|\cdot|)^\alpha f|*|g|+|f|*|(1+|\cdot|)^\alpha g|.
\]
For $l\ge0$ we have $\nabla^l\Gamma_{n,\ge1}(x)\lesssim n^l(1+|x|)^{-2n-1-l}$, so $(1+|x|)^\alpha\nabla^l\Gamma_{n,\ge1}\in L^1$ as $\alpha<1$,
with the norm polynomial in $n$. Hence
\begin{equation}\label{Gn-ge1-bound}
\begin{aligned}
\|(1+|x|)^\alpha(\nabla^{k+1-l}\Gamma_{n,\ge1}*\rho h^{2n-j})\|_{L^2}
&\le C_k^n \|(1+|x|)^\alpha\rho h^{2n-j}\|_{L^2}\\
&\le C_k^n \|(1+|x|)^\alpha\rho\|_{L^2}\|h\|_{L^\infty}^{2n-j}.
\end{aligned}
\end{equation}
Since $\nabla^l(h^j)$ expands into $O(n^l + 1)$ terms, each bounded by $\|h\|_{L^\infty}^{j-1}\|h\|_{W^{k+1,\infty}}$, it follows that
\begin{equation}\label{Kn-ge1-bound}
\|(1+|x|)^\alpha\nabla^kK_{n,\ge1,j}\rho\|_{L^2}
\le C_k^n\|h\|_{L^\infty}^{2n-1}\|h\|_{W^{k+1,\infty}}\|(1+|x|)^\alpha\rho\|_{L^2}.
\end{equation}
Summing over $0 \le k \le 8$, $0 \le j \le 2n$ we get the desired bound for $K_{n,\ge1}\rho$.

For $K_{n,\ge1}|\nabla|\phi$, the right-hand side of (\ref{Gn-ge1-bound}) becomes
\[
C_k^n\|(1 + |x|)^\alpha|\nabla|\phi\|_{L^2}\|h\|_{L^\infty}^{2n-j}
\lesssim C_k^n \||\nabla|^{1/2}\phi\|_Z\|h\|_{L^\infty}^{2n-j}
\]
by Lemma \ref{Z-bound} (iii), so the desired bound also follows.

{\bf Part 2:} $K_{n,<1}$. We bound $K_{n,<1}|\nabla|\phi$ first.
For $0 \le k \le 8$ we have $\nabla^kK_{n,<1}|\nabla|\phi(x)$ is a sum of $2^{k+1}$ terms like
\begin{equation}\label{Kn-<1-int}
\int_{\R^2} \Gamma_{n,<1}(y)\nabla_x^l(h(x) - h(x - y))^{2n}\nabla^{k+1-l}|\nabla|\phi(x - y)dy.
\end{equation}
By Taylor's formula,
\[
h(x) - h(x - y) = \int_0^1 y\cdot\nabla h(x - ty)dt.
\]
Since $\Gamma_{n,<1}(y)|y|^{2n} \in L^1$, with the norm polynomial in $n$,
\[
\|(1+|x|)^\alpha(\ref{Kn-<1-int})\|_{L^2}\\
\lesssim \sup_{|y|\le2\atop t_j\in[0,1]} \|(1+|x|)^\alpha
\nabla_x^l\left( \otimes_{j=1}^{2n} \nabla h(x-t_jy) \right)\nabla^{k+1-l}|\nabla|\phi(x-y)\|_{L_x^2}.
\]
Since for $|y| \le 2$ and $t \in [0, 1]$, $1 + |x| \approx 1 + |x - t_jy|$,
by Lemma \ref{prod-Z}, the above is bounded by
\begin{align*}
&C_k^n\|\nabla h\|_{L^\infty}^{2n-1}(\||\nabla|^{1/2}\phi\|_{W^{k+2,\infty}}\|(1+|x|)^\alpha\nabla h\|_{L^2}+\|(1+|x|)^\alpha|\nabla|^{1/2}\phi\|_{L^2}\|\nabla h\|_{W^{k+2,\infty}})\\
\le&C_k^n\|U\|_{W^{1,\infty}}^{2n-1}\|U\|_{W^{k+3,\infty}}\|U\|_Z.
\end{align*}
Summing over $0 \le k \le 8$ gives the desired bound.

For $K_{n,<1}\rho$, by Lemma \ref{Kn<1-Wkp} we have
\[
\|K_{n,<1}\rho\|_{H^8}\le C^n\|h\|_{W^{2,\infty}}^{2n-1}\|h\|_{W^{11,\infty}}\|\rho\|_{H^8}.
\]
Since $\Gamma_{n,<1}$ is supported in $B(0,2)$, we can improve this to
\[
\|K_{n,<1}\rho\|_{H^8(B(x,1))}\le C^n\|h\|_{W^{2,\infty}}^{2n-1}\|h\|_{W^{11,\infty}}\|\rho\|_{H^8(B(x,3))}.
\]
Using the same argument as (\ref{insert-weight}), we can insert the weight in the $H^8$ norm and get the desired bound for $K_{n,<1}\rho$.
\end{proof}

We also need the analog of Lemma \ref{Kn-Z} with no weights involved.
\begin{lemma}\label{Kn-Wkp}
For any integer $n \ge 1$,
\[
K_n\rho=C_{k,p,q}^nO_W(h^{2n} \cdot \rho).
\]
\end{lemma}
\begin{proof}
As before it suffices to bound $K_{n,\ge1}\rho$ and $K_{n,<1}\rho$ separately.

{\bf Part 1:} $K_{n,\ge1}\rho$. Referring back to the proof of Lemma \ref{Kn-Z} we have
\[
\|(\ref{Kn-ge1-int})\|_{L^r}
\le C_k^n \|h^j\|_{W^{k+1,q}}\|\rho h^{2n-j}\|_{L^p}
\le C_{k,q}^n \|h\|_{L^\infty}^{2n-1}\|h\|_{W^{k+1,q}}\|\rho\|_{L^p}.
\]
The first inequality is the consequence of the $L^1$ norm bound on the kernel, derived in the proof of Lemma \ref{Kn-Z}, while the second one follows from repeated application of the Sobolev multiplication theorem,
with a constant at most exponential in $n$.

{\bf Part 2:} $K_{n,<1}\rho$. This was covered by Lemma \ref{Kn<1-Wkp}.
\end{proof}

Now we are ready to show the weighted Sobolev bounds for the Taylor remainders $B_j$ and $N_j$, $j=2$, 3, 4, as defined in (\ref{B23}), (\ref{N23}), (\ref{B34}) and (\ref{N34}).
\begin{proposition}\label{BN-ZW}
If all the Sobolev norms of $U$ appearing in Definition \ref{O-hn-g-def} are sufficiently small, then for $j = 2, 3, 4$ we have $B_j$ and $N_j = O(U^{j-1} \cdot U)$.
\end{proposition}
\begin{proof}
By Lemma \ref{Kn-Z}, Lemma \ref{Kn-Wkp} and (\ref{gr-phi-Taylor}),
\begin{equation}\label{gr-phi-Taylorrem}
|\nabla|\phi=\rho-\frac{1}{4\pi}K_1\rho+O(h^4\cdot\rho)
\end{equation}
so by (\ref{rho-Taylor}),
\begin{equation}\label{rho-Taylorrem}
\rho=|\nabla|\phi+\frac{1}{4\pi}K_1|\nabla|\phi+O(U^4\cdot U).
\end{equation}
Note that the constants in Lemma \ref{Kn-Z} and Lemma \ref{Kn-Wkp} depend exponentially on $n$, so the series (\ref{rho-Taylor}) converges if the corresponding norm of $U$ is sufficiently small. Using the identity
\begin{align*}
|\nabla|\phi(y)\cdot|h(x)-h(y)|^2
&=|\nabla|\phi(y)\cdot(h(x)^2-2h(x)h(y)+h(y)^2)\\
&=h(x)^2(|\nabla|\phi(y)-|\nabla|\phi(x))\\
&-2h(x)(h(y)|\nabla|\phi(y)-h(x)|\nabla|\phi(x))\\
&+h(y)^2|\nabla|\phi(y)-h(x)^2|\nabla|\phi(x)
\end{align*}
and the expression for the Riesz potential
\[
|\nabla|f(x)=\frac{1}{2\pi}\text{p.v.}\int_{\R^2} \frac{f(x)-f(y)}{|x-y|^3}dy
\]
we obtain
\[
\frac{1}{4\pi}K_1|\nabla|\phi=-\frac{1}{2}|\nabla|(h^2|\nabla|^2\phi-2h|\nabla|(h|\nabla|\phi)+|\nabla|(h^2|\nabla|\phi))=B_3^\circ+|\nabla h|^2|\nabla|\phi,
\]
see (\ref{B34}), so
\[
\rho=|\nabla|\phi+B_3^\circ+|\nabla h|^2|\nabla|\phi+O(U^4\cdot U).
\]

Now by (F.4) of \cite{GeMaSh2}, modulo a sign flip,
\[
G(h)\phi=\rho+\frac{1}{2\pi}\int_{\R^2} |\nabla|\phi(y)\frac{\nabla h(x)\cdot(x-y)+h(y)-h(x)}{|x-y|^3}dy+O(U^3\cdot U).
\]
Using the identity
\[
|\nabla|\phi(y)\cdot(h(y)-h(x))=(h(y)|\nabla|\phi(y)-h(x)|\nabla|\phi(x))+h(x)(|\nabla|\phi(x)-|\nabla|\phi(y))
\]
and the expression for the Riesz transform
\[
\nabla\phi=\frac{\nabla}{|\nabla|}|\nabla|\phi
=\frac{1}{2\pi}\text{p.v.}\int_{\R^2} \frac{y-x}{|x-y|^3}|\nabla|\phi(y)dy
\]
we obtain
\begin{equation}\label{Gh-phi-int3}
\begin{aligned}
G(h)\phi
&=\rho-\nabla h\cdot\nabla\phi-|\nabla|(h|\nabla|\phi)+h|\nabla|^2\phi
+O(U^3\cdot U)\\
&=|\nabla|\phi-\nabla\cdot(h\nabla\phi)-|\nabla|(h|\nabla|\phi)+B_3^\circ+|\nabla h|^2|\nabla|\phi+O(U^3\cdot U).
\end{aligned}
\end{equation}
Comparing (\ref{Gh-phi-int2}) and (\ref{Gh-phi-int3}) we get
\[
B_4+|\nabla h|^2(B-|\nabla|\phi)=O(U^3\cdot U).
\]
By (\ref{B34}), $B_3^\circ=O(U^2\cdot U)$, so by (\ref{Gh-phi-int3}),
$G(h)\phi-|\nabla|\phi=O(U^1\cdot U)$. By (\ref{BVh-phi}),
\[
B-|\nabla|\phi=\frac{G(h)\phi-|\nabla|\phi+\nabla h\cdot\nabla\phi-|\nabla h|^2|\nabla|\phi}{1+|\nabla h|^2}=O(U^1\cdot U)
\]
so $|\nabla h|^2(B-|\nabla|\phi)=O(U^3\cdot U)$; hence the bound for $B_4$.
Now (\ref{B34}) and (\ref{B23}) give the bounds for $B_3$ and $B_2$;
(\ref{N34}) and (\ref{N23}) give the bounds for $N_3$ and $N_2$.
\end{proof}

\subsection{Remainder estimates of $G(h)\phi$ on the torus}
In this subsection we show how to generalize the single layer potential representation (\ref{single-layer}) to the periodic case.
By Section 3(b) of \cite{Lint}, the Green function on the torus $(\R/R\Z)^2$
can be written as $G(x_1 - x_1', x_2 - x_2', z - z')$, where
\[
G(x_1, x_2, z) = -\frac{1}{4\pi r_{0,0}} - \frac{1}{4\pi}\sum_{(m,n)\neq(0,0)} \left( \frac{1}{r_{m,n}}-\frac{1}{R\sqrt{m^2 + n^2}} \right) + C_R
\]
with $C_R$ being a constant explicitly computable from $R$ (see (3.32) and (3.33) of \cite{Lint}), and
\[
r_{m,n} = \sqrt{(x_1 - mR)^2 + (x_2 - nR)^2 + z^2}.
\]
This Green function has the same local behavior as the Euclidean one,
but as $z \to \infty$, by (3.30) of \cite{Lint},
\begin{equation}\label{Green-per-asymp}
G(x_1, x_2, z) = \frac{|z|}{2R^2} + o_R(1)
\end{equation}
uniformly in $x_1$ and $x_2$.

Now we expand
\[
2|\nabla|\int_{(\R/R\Z)^2} \rho(y_1, y_2)G(x_1 - y_1, x_2 - y_2, h(x_1, x_2) - h(y_1, y_2))dy_1dy_2
\]
into a Taylor series with respect to $h$, to get an expression similar to (\ref{gr-phi-Taylor}), which can then be inverted to give a series expansion of $\rho$ similar to (\ref{rho-Taylor}). Thus, if we define $\rho$ using (\ref{rho-Taylor}), then (\ref{gr-phi-Taylor}) holds.
Acting $|\nabla|^{-1}$ on both sides of (\ref{gr-phi-Taylor}) we see that
\begin{equation}\label{single-layer-per}
\Phi(x_1, x_2, z) = 2\int_{(\R/R\Z)^2} \rho(y_1, y_2)G(x_1 - y_1, x_2 - y_2, z - h(y_1, y_2))dy_1dy_2 + C
\end{equation}
holds on the boundary $\Gamma$, where $C$ is a constant.
We now show that it holds everywhere.
We will use the maximum principle to identify both sides.
To that end we first extract their common properties.

\begin{proposition}
Both sides of (\ref{single-layer-per}) are:

(i) continuous in $\overline{\Omega(t)}$,

(ii) harmonic in $\Omega(t)$, and

(iii) equal on $\Gamma(t) = \partial\Omega(t)$. Also they

(iv) grow slower than $|z|$ as $z \to -\infty$.
\end{proposition}
\begin{proof}
(i) To show the continuity of $\Phi$, we recall the change of variable (\ref{u-def}) relating $\Phi$ and $u$. Then it suffices to show the continuity of $u$ on $\{y \le 0\}$. Indeed, by Lemma \ref{u-Cr} (i), $\|\partial_yu\|_{L_y^2L_x^\infty} < \infty$. Integrating in $y$ using Cauchy--Schwarz, we infer that $u(x, y)$ is continuous in $y$, uniformly in $x$. Since for each fixed $y$, $u(x, y) \in H_x^2$ is also continuous, we conclude that $u$ is continuous on $\{y \le 0\}$. Turning to the right-hand side of (\ref{single-layer-per}), we note that $\rho \in H^2$ is continuous by (\ref{rho-Taylorrem}). Then its single layer potential, i.e., the right-hand side of (\ref{single-layer-per}), is continuous by an argument similar to the proof of Theorem 3.28 of \cite{Fo}.

(ii) $\Phi$ is harmonic in $\Omega(t)$ by (\ref{phi-Diri}), so is the right-hand side of (\ref{single-layer-per}) because the Green function $G$ is harmonic and the charge $\rho \in L^2 \subset L^1$.

(iii) This has already been shown previously.

(iv) For $\Phi$, it suffices to show that $u(x, y)$ grows slower than $|y|$ as $y \to -\infty$. Indeed, this follows from integrating $\|\partial_yu\|_{L_y^2L_x^\infty} < \infty$ in $y$ using Cauchy--Schwarz.
Turning to the other side, by (\ref{Green-per-asymp}) we have, as $z \to -\infty$,
\begin{align*}
\text{right-hand side of }(\ref{single-layer-per})
&=\int_{(\R/R\Z)^2} \rho(y) \left( \frac{|z - h(y)|}{R^2} + o_R(1) \right)dy\\
&=R^{-2}|z|\hat\rho(0) + O(\|\rho\|_{L^\infty}\|h\|_{L^\infty}) + o_R(\|\rho\|_{L^1})
\end{align*}
is bounded because $\hat\rho(0) = 0$ by its definition (\ref{rho-Taylor}).
\end{proof}

Taking the difference of the two sides we know that (\ref{single-layer-per}) holds everywhere thanks to Lemma \ref{u-Diri-unique},
which extends easily to the periodic case.
Since only $|\nabla|^{1/2}\phi$ instead of $\phi$ is included in $U$,
the difference by a constant is immaterial,
so Proposition \ref{BN-ZW} carries over.

\end{document}